\documentclass[11pt]{amsart}

\usepackage{amssymb, amscd}
\usepackage{epsfig, mathtools, tensor}
\usepackage[vmargin=1in, hmargin=1.25in]{geometry}
\usepackage[font=small,format=plain,labelfont=bf,up,textfont=it,up]{caption}
\usepackage{microtype}
\usepackage[shortlabels]{enumitem}
\usepackage[backref=page, bookmarks, bookmarksdepth=2, colorlinks=true, linkcolor=blue, citecolor=blue, urlcolor=blue]{hyperref}
\usepackage{etoolbox}
\usepackage{tikz-cd}
\usepackage{tabularray}
\apptocmd{\thebibliography}{\raggedright}{}{}
\usepackage[only,llbracket,rrbracket,llparenthesis, rrparenthesis]{stmaryrd}
\usepackage{xcolor}

\mathtoolsset{showonlyrefs}

\newcommand\Figure[1]{\centerline{\psfig{file=#1,scale=1}}}

\makeatletter
\patchcmd{\@maketitle}{\global\topskip42\p@\relax}
  {\global\topskip42\p@\relax \vspace*{-38pt}}
  {}{}
\makeatother

\setenumerate[0]{label=(\alph*)}

\renewcommand*{\backref}[1]{}
\renewcommand*{\backrefalt}[4]{%
    \ifcase #1 (Not cited.)%
    \or        (Cited on page~#2.)%
    \else      (Cited on pages~#2.)%
    \fi}

\newcommand{\arxiv}[1]{\href{http://arxiv.org/abs/#1}{{\tt arXiv:#1}}}

\newcommand*{\Cdot}[1][1.25]{%
  \mathpalette{\CdotAux{#1}}\cdot%
}
\newdimen\CdotAxis
\newcommand*{\CdotAux}[3]{%
  {%
    \settoheight\CdotAxis{$#2\vcenter{}$}%
    \sbox0{%
      \raisebox\CdotAxis{%
        \scalebox{#1}{%
          \raisebox{-\CdotAxis}{%
            $\mathsurround=0pt #2#3$%
          }%
        }%
      }%
    }%
    \dp0=0pt %
    \sbox2{$#2\bullet$}%
    \ifdim\ht2<\ht0 %
      \ht0=\ht2 %
    \fi
    \sbox2{$\mathsurround=0pt #2#3$}%
    \hbox to \wd2{\hss\usebox{0}\hss}%
  }%
}

\numberwithin{equation}{section}

\theoremstyle{plain}
\newtheorem{theorem}{Theorem}[section]
\newtheorem{maintheorem}{Theorem}
\newtheorem{maincorollary}[maintheorem]{Corollary}
\newtheorem{maintheoremprime}{Theorem}
 
\newtheorem{proposition}[theorem]{Proposition}
\newtheorem{lemma}[theorem]{Lemma}

\newtheorem{corollary}[theorem]{Corollary}

\newtheorem*{unnumberedclaim}{Claim}

\newtheorem{stepx}{Step}
\newenvironment{step}[1]
 {\renewcommand\thestepx{#1}\stepx}
 {\endstepx}

\newtheorem{casex}{Case}
\newenvironment{case}[1]
 {\renewcommand\thecasex{#1}\casex}
 {\endcasex}

\newtheorem{claimx}{Claim}
\newenvironment{claim}[1]
 {\renewcommand\theclaimx{#1}\claimx}
 {\endclaimx}

\theoremstyle{definition}
\newtheorem{asm}[theorem]{Assumption}
\newenvironment{assumption}[1][]{\begin{asm}[#1]\pushQED{\qed}}{\popQED \end{asm}}
\newtheorem{defn}[theorem]{Definition}
\newenvironment{definition}[1][]{\begin{defn}[#1]\pushQED{\qed}}{\popQED \end{defn}}
\newtheorem{notn}[theorem]{Notation}

\theoremstyle{remark}
\newtheorem{rmk}[theorem]{Remark}
\newenvironment{remark}[1][]{\begin{rmk}[#1] \pushQED{\qed}}{\popQED \end{rmk}}
\newtheorem{eg}[theorem]{Example}
\newenvironment{example}[1][]{\begin{eg}[#1] \pushQED{\qed}}{\popQED \end{eg}}
\newtheorem{cvn}[theorem]{Convention}

\theoremstyle{plain}

\DeclareMathOperator{\Hom}{Hom}

\DeclareMathOperator{\coker}{coker}
\DeclareMathOperator{\Coker}{coker}
\DeclareMathOperator{\Image}{Im}

\DeclareMathOperator{\GL}{GL}
\DeclareMathOperator{\KL}{KL}

\DeclareMathOperator{\SL}{SL}

\DeclareMathOperator{\Sp}{Sp}

\DeclareMathOperator{\Mat}{Mat}

\newcommand\C{\ensuremath{\mathbb{C}}}
\newcommand\Z{\ensuremath{\mathbb{Z}}}
\newcommand\Q{\ensuremath{\mathbb{Q}}}

\DeclareMathOperator{\HH}{H}
\newcommand\RH{\ensuremath{\widetilde{\HH}}}
\DeclareMathOperator{\CC}{C}

\DeclareMathOperator{\Ind}{Ind}
\DeclareMathOperator{\Res}{Res}
\DeclareMathOperator{\Int}{Int}

\DeclareMathOperator{\Char}{char}
\DeclareMathOperator{\Sym}{Sym}

\DeclareMathOperator{\id}{id}

\newcommand\Span[1]{\ensuremath{\langle #1 \rangle}}
\newcommand\SpanSet[2]{\ensuremath{\langle \text{#1 $|$ #2} \rangle}}
\newcommand\Set[2]{\ensuremath{\left\{\text{#1 $|$ #2}\right\}}}
\newcommand\GroupPres[2]{\ensuremath{\left\langle \text{#1 $|$ #2} \right\rangle}}

\newcommand\cC{\ensuremath{\mathcal{C}}}

\newcommand\cK{\ensuremath{\mathcal{K}}}

\newcommand\fK{\ensuremath{\mathfrak{K}}}

\newcommand\fc{\ensuremath{\mathfrak{c}}}
\newcommand\fd{\ensuremath{\mathfrak{d}}}

\newcommand\fh{\ensuremath{\mathfrak{h}}}

\newcommand\bL{\ensuremath{\mathbf{L}}}

\newcommand\bU{\ensuremath{\mathbf{U}}}
\newcommand\bV{\ensuremath{\mathbf{V}}}
\newcommand\bW{\ensuremath{\mathbf{W}}}

\newcommand\bk{\ensuremath{\mathbf{k}}}

\newcommand\bbF{\ensuremath{\mathbb{F}}}

\newcommand\bbI{\ensuremath{\mathbb{I}}}

\newcommand\te{\ensuremath{\widetilde{e}}}

\newcommand\tv{\ensuremath{\widetilde{v}}}

\newcommand\tsigma{\ensuremath{\widetilde{\sigma}}}

\newcommand\tphi{\ensuremath{\widetilde{\phi}}}

\newcommand\okappa{\ensuremath{\overline{\kappa}}}

\newcommand\oomega{\ensuremath{\overline{\omega}}}

\newcommand\Mod{\ensuremath{\operatorname{Mod}}}
\newcommand\Torelli{\ensuremath{\mathcal{I}}}
\newcommand\tTorelli{\ensuremath{\widetilde{\Torelli}}}
\newcommand\diag{\ensuremath{\operatorname{diag}}}
\newcommand\LLComm[1]{\ensuremath{\langle\!\langle #1 \rangle\!\rangle}}
\newcommand\LLAComm[1]{\ensuremath{\langle\!\langle #1 \rangle\!\rangle}_{\alpha}}
\newcommand\LLBComm[1]{\ensuremath{\langle\!\langle #1 \rangle\!\rangle}_{\beta}}
\newcommand\Pres[1]{\ensuremath{\llbracket #1 \rrbracket}}
\newcommand\pt{\ensuremath{\operatorname{pt}}}
\newcommand\cone{\ensuremath{\operatorname{Cone}}}

\newcommand\ssE{\ensuremath{\operatorname{E}}}
\newcommand\ssF{\ensuremath{\operatorname{F}}}

\title{The second rational homology of the Torelli group}

\author{Daniel Minahan}
\address{Dept of Mathematics; University of Chicago; Chicago, IL 60637}
\email{dminahan@uchicago.edu}

\author{Andrew Putman}
\address{Dept of Mathematics; University of Notre Dame; 255 Hurley Hall; Notre Dame, IN 46556}
\email{andyp@nd.edu}

\thanks{AP was supported by NSF grant DMS-2305183.  DM was supported by NSF grant DMS-2402060.}

\begin{document}

\newpage

\begin{abstract}
We calculate the second rational homology group of the Torelli group for $g \geq 6$.
\end{abstract}

\maketitle
\thispagestyle{empty}

\section{Introduction}
\label{section:introduction}

Let $\Sigma_{g,p}^b$ be an oriented genus $g$ surface with $p$ marked points
and $b$ boundary components.  We often omit $p$ or $b$ if they vanish.  The
mapping class group $\Mod_{g,p}^b$ is
the group
of isotopy classes of orientation-preserving diffeomorphisms of $\Sigma_{g,p}^b$ that fix each marked point and boundary
component pointwise.  Deleting the marked points and gluing discs to the boundary components,
we get an action of $\Mod_{g,p}^b$ on $\HH_1(\Sigma_g)$ that fixes the algebraic intersection form.
This gives a surjection $\Mod_{g,p}^b \rightarrow \Sp_{2g}(\Z)$ whose
kernel $\Torelli_{g,p}^b$ is the Torelli group:
\[\begin{tikzcd}
1 \arrow{r} & \Torelli_{g,p}^b \arrow{r} & \Mod_{g,p}^b \arrow{r} & \Sp_{2g}(\Z) \arrow{r} & 1.
\end{tikzcd}\]
Johnson \cite{JohnsonFinite, JohnsonAbel} proved that $\Torelli_{g,p}^b$ is finitely generated for $g \geq 3$ and
calculated $\HH^1(\Torelli_{g,p}^b)$.  
The conjugation action of $\Mod_{g,p}^b$ on $\Torelli_{g,p}^b$
induces an action of $\Sp_{2g}(\Z)$ on each $\HH^d(\Torelli_{g,p}^b)$.  Let $H = \HH^1(\Sigma_g;\Q)$.
For $g \geq 3$, it follows from Johnson's work (see \cite[Theorem 3.5]{HainTorelli}) that there is
an $\Sp_{2g}(\Z)$-equivariant isomorphism\footnote{In this, the inclusion
$H \hookrightarrow \wedge^3 H$ takes $h \in H$ to $h \wedge \omega$, where $\omega \in \wedge^2 H$
is the algebraic intersection form.}
\[\HH^1(\Torelli_{g,p}^b;\Q) \cong H^{\oplus (p+b)} \oplus (\wedge^3 H)/H.\]
In particular, $\HH^1(\Torelli_{g,p}^b;\Q)$
is a finite-dimensional algebraic representation\footnote{A representation $\bV$ of $\Sp_{2g}(\Z)$ over a field
$\bk$ of characteristic $0$ is algebraic if the action of $\Sp_{2g}(\Z)$ on $\bV$ extends to
a polynomial representation of 
the $\bk$-points $\Sp_{2g}(\bk)$ of the algebraic group $\Sp_{2g}$.  Since
$\Sp_{2g}(\Z)$ is Zariski dense in $\Sp_{2g}(\bk)$, such an extension is unique if
it exists.}
of $\Sp_{2g}(\Z)$.

\subsection{Main theorem}
A long-standing folk conjecture\footnote{One place where this conjecture appears in print
is in work of Church--Farb; see \cite[Conjecture 1.7]{ChurchFarbAbelJacobi}.} says that $\HH^2(\Torelli_{g,p}^b;\Q)$ is also a finite-dimensional\footnote{We emphasize
that even finite-dimensionality was unknown before our work.}
algebraic representation of $\Sp_{2g}(\Z)$ for $g \gg 0$.  We prove this for $g \geq 6$.  When $p+b \leq 1$, we actually
compute $\HH^2(\Torelli_{g,p}^b;\Q)$.
The irreducible algebraic representations of $\Sp_{2g}(\Z)$ are indexed by partitions $\sigma$ with at most
$g$ parts (see \cite[\S 17]{FultonHarris}).  Let $\bV_{\sigma}$ be the representation corresponding to $\sigma$,
so $\bV_1 = H$ and $\bV_{1^3} = (\wedge^3 H)/H$.  We prove:

\begin{maintheorem}
\label{maintheorem:h2torellicalc}
For $g \geq 6$, we have
\begin{alignat*}{7}
&\HH^2(\Torelli_g;\Q)     &&\cong                         &&\bV_{1^2}            &&                   &&\oplus \bV_{1^4}            &&\oplus \bV_{2^2,1^2} &&\oplus \bV_{1^6}, \\
&\HH^2(\Torelli_g^1;\Q)   &&\cong                         &&\bV_{1^2}^{\oplus 2} &&\oplus \bV_{2,1^2} &&\oplus \bV_{1^4}^{\oplus 2} &&\oplus \bV_{2^2,1^2} &&\oplus \bV_{1^6}, \\
&\HH^2(\Torelli_{g,1};\Q) &&\cong \bV_0            \oplus &&\bV_{1^2}^{\oplus 2} &&\oplus \bV_{2,1^2} &&\oplus \bV_{1^4}^{\oplus 2} &&\oplus \bV_{2^2,1^2} &&\oplus \bV_{1^6}.
\end{alignat*}
In all three cases, these cohomology groups are spanned by cup products of elements of $\HH^1$.
\end{maintheorem}

To explain the origin of the representations in Theorem~\ref{maintheorem:h2torellicalc}, 
consider the cup product pairing $\wedge^2 \HH^1(\Torelli_{g,p}^b;\Q) \rightarrow \HH^2(\Torelli_{g,p}^b;\Q)$.  We
have 
\[\wedge^2 \HH^1(\Torelli_g;\Q) \cong \wedge^2((\wedge^3 H)/H) \quad \text{and} \quad \wedge^2 \HH^1(\Torelli_g^1;\Q) \cong \wedge^2 \HH^1(\Torelli_{g,1};\Q) \cong \wedge^2(\wedge^3 H).\]
These decompose as\footnote{This calculation can easily be done using the program ``LiE''; see \cite{LieProgram}.}
\begin{alignat*}{8}
&\wedge^2 ((\wedge^3 H)/H) &&\cong \bV_0            &&\oplus \bV_{1^2}            &&\oplus \bV_{2^2} &&                   &&\oplus \bV_{1^4}            &&\oplus \bV_{2^2,1^2} &&\oplus \bV_{1^6}, \stepcounter{equation}\tag{\theequation}\label{eqn:decomposecup}\\
&\wedge^2 (\wedge^3 H)     &&\cong \bV_0^{\oplus 2} &&\oplus \bV_{1^2}^{\oplus 3} &&\oplus \bV_{2^2} &&\oplus \bV_{2,1^2} &&\oplus \bV_{1^4}^{\oplus 2} &&\oplus \bV_{2^2,1^2} &&\oplus \bV_{1^6}.
\end{alignat*}
Hain (\cite{HainInfinitesimal}; see also \cite[Corollary 7.4]{HainKahler} and
\cite[\S 9]{HainUniversal} and \cite{HabeggerSorger}) 
computed the kernel of the cup product pairing on $\Torelli_{g,p}^b$ for $g \geq 3$ and $p+b \leq 1$. 
When $g \geq 6$, it is isomorphic to $\bV_0 \oplus \bV_{2^2}$ for $\Torelli_g$, it is
isomorphic to $\bV_0^{\oplus 2} \oplus \bV_{1^2} \oplus \bV_{2^2}$ for $\Torelli_g^1$, and
it is isomorphic to $\bV_0 \oplus \bV_{1^2} \oplus \bV_{2^2}$ for $\Torelli_{g,1}$.
Deleting these kernels from \eqref{eqn:decomposecup} gives the representations in Theorem~\ref{maintheorem:h2torellicalc}.

Kupers--Randal-Williams \cite{KupersRandalWilliamsTorelli} showed that in the above cases the image of the cup product pairing is the maximal
algebraic subrepresentation of $\HH^2(\Torelli_{g,p}^b;\Q)$.  To prove Theorem~\ref{maintheorem:h2torellicalc}, we must show that
this is all of $\HH^2(\Torelli_{g,p}^b;\Q)$.  We actually prove more:\footnote{In this theorem, we switch to homology
since that is more natural for our proofs.}

\begin{maintheorem}
\label{maintheorem:h2torelli}
Let $b,p \geq 0$.  Then $\HH_2(\Torelli_{g,p}^b;\Q)$ is finite dimensional for $g \geq 5$
and an algebraic representation of $\Sp_{2g}(\Z)$ for $g \geq 6$.
\end{maintheorem}

\begin{remark}
Our proof only works over $\Q$.  It is not known if $\HH_2(\Torelli_{g,p}^b)$ is
finitely generated.
\end{remark}

\begin{remark}
This paper supersedes a paper of of Minahan \cite{MinahanThesis} that uses a less sophisticated 
version of our argument to prove that $\HH_2(\Torelli_g;\Q)$ is finite-dimensional for $g \geq 51$.
\end{remark}

\subsection{Representation stability}

Theorem~\ref{maintheorem:h2torellicalc} implies the following:\footnote{This requires not only the
formula for $\HH^2(\Torelli_g^1;\Q)$ from Theorem~\ref{maintheorem:h2torellicalc}, but also the
explicit isomorphism underlying it.  A version of Corollary \ref{maincorollary:repstability}
without the explicit starting value $g=6$ also follows from Theorem~\ref{maintheorem:h2torelli}
along with \cite[Theorem B]{PatztFiltrations} and \cite{BoldsenDollerup}.}

\begin{maincorollary}
\label{maincorollary:repstability}
The sequence of $\Sp_{2g}(\Z)$-representations $\{\HH_2(\Torelli_g^1;\Q)\}_{g=1}^{\infty}$ is uniformly
representation stable starting at $g=6$.
\end{maincorollary}

\noindent
This was conjectured by Church--Farb \cite[Conjecture 6.1]{ChurchFarbRepStability}.  It means that for $g \geq 6$ the maps
\begin{equation}
\label{eqn:stabmap}
\begin{tikzcd}
\HH_2(\Torelli_g^1;\Q) \arrow{r} & \HH_2(\Torelli_{g+1}^1;\Q)
\end{tikzcd}
\end{equation}
induced by embedding $\Sigma_g^1$ into $\Sigma_{g+1}^1$ and extending mapping classes lying in
$\Torelli_g^1$ to $\Sigma_{g+1}^1$ by the identity are injective and (roughly speaking) match up the decompositions
of $\HH_2(\Torelli_g^1;\Q)$ and $\HH_2(\Torelli_{g+1}^1;\Q)$ into irreducible representations of
the symplectic groups.  Partial results in this direction were previously
proven by Boldsen--Dollerup \cite{BoldsenDollerup} and Miller--Patzt--Wilson \cite{MillerPatztWilson}.
However, before our work it was not even known if the maps \eqref{eqn:stabmap} were injective for $g \gg 0$.

\begin{remark}
Since the representation $\bV_{1^6}$ appears in $\HH_2(\Torelli_6^1;\Q)$ and 
is not the stabilization of a representation of $\Sp_{2g}(\Z)$ for $g=5$, the $g=6$ in
Corollary \ref{maincorollary:repstability} is optimal.
\end{remark}

\subsection{Previous work}

Questions about $\HH_{\bullet}(\Torelli_{g,p}^b)$ can generally be reduced to questions about
$\HH_{\bullet}(\Torelli_g)$ (see, e.g., \S \ref{section:deleteboundary}), so we mostly focus on this.  

\subsubsection{Low genus}
\label{section:lowgenus}

As we said, for $g \geq 3$ Johnson \cite{JohnsonFinite, JohnsonAbel} proved that $\Torelli_{g}$ 
is finitely generated and computed
$\HH_1(\Torelli_{g})$.  In contrast, McCullough--Miller \cite{McCulloughMiller} proved that $\Torelli_{2}$ 
is not finitely generated,
and later Mess \cite{MessThesis} proved that $\Torelli_2$ is an infinite rank free group.  Johnson--Millson (cf.\ \cite{MessThesis}) 
and Hain \cite{HainThreeFolds} proved that $\HH_3(\Torelli_3;\Q)$ and $\HH_4(\Torelli_3;\Q)$ are infinite dimensional.  
Spiridonov \cite{Spiridonov} later calculated
$\HH_4(\Torelli_3;\Q)$.  Some evidence that
$\HH_2(\Torelli_3)$ might not be finitely generated was given by Gaifullin \cite{GaifullinSS}.

\subsubsection{Second homology}

For $g \geq 3$, 
the only earlier finiteness result about $\HH_2(\Torelli_g)$
was a theorem of Kassabov--Putman \cite{KassabovPutman} saying
that $\HH_2(\Torelli_{g})$ is spanned by the $\Sp_{2g}(\Z)$-orbit of a finite set.  In other words, it is finitely generated
as a module over $\Z[\Sp_{2g}(\Z)]$.  Boldsen--Dollerup \cite{BoldsenDollerup} also proved a related
theorem that roughly speaking says that $\HH_2(\Torelli_g^1)$ is spanned by classes supported on
genus $6$ subsurfaces for $g \geq 6$.

The torsion in $\HH_2(\Torelli_g)$ remains mysterious.
The group $\HH_1(\Torelli_g)$ calculated by Johnson \cite{JohnsonAbel} contains $2$-torsion coming
from the Birman--Craggs--Johnson (BCJ) homomorphism.  This was constructed by Johnson \cite{JohnsonBCJ} using 
work of Birman--Craggs \cite{BirmanCraggs} on the Rochlin invariant of homology $3$-spheres.
Though they were unable to compute it completely, Brendle--Farb \cite{BrendleFarbTorelli} constructed large
parts of $\HH_2(\Torelli_g;\bbF_2)$ that are detected by the BCJ homomorphism.

\subsubsection{High degree}

Akita \cite{AkitaTorelli} proved that for each $g \geq 7$ 
there exists some $d$ such that $\HH_d(\Torelli_g;\Q)$ is infinite dimensional.
Bestvina--Bux--Margalit \cite{BestvinaBuxMargalitcd} sharpened this and proved for $g \geq 2$
that $\Torelli_g$ has cohomological dimension $3g-5$ and that
$\HH_{3g-5}(\Torelli_g;\Q)$ is infinite dimensional.  Later Gaifullin \cite{Gaifullin} proved for $g \geq 2$ that
$\HH_{d}(\Torelli_g;\Q)$ is infinite dimensional for $2g-3 \leq d \leq 3g-5$.
For $g=3$, this recovers the infinite dimensionality results discussed in \S \ref{section:lowgenus} above.  Other work constructing
high-dimensional classes can be found in \cite{ChurchFarbAbelJacobi}.

\subsection{Stability}
\label{section:coherent}

We prove Theorem~\ref{maintheorem:h2torelli} via a criterion with a representation
stability flavor.  A {\em coherent sequence} of representations of $\Sp_{2g}(\Z)$ over a field $\bk$ is a sequence
\begin{equation}
\label{eqn:cohseq}
\begin{tikzcd}
\bV_1 \arrow{r}{f_1} & \bV_2 \arrow{r}{f_2} & \bV_3 \arrow{r}{f_3} & \cdots
\end{tikzcd}
\end{equation}
of $\bk$-vector spaces connected by linear maps such that the following hold:
\begin{itemize}
\item Each $\bV_g$ is a representation of $\Sp_{2g}(\Z)$.
\item Let $\bV_g \boxtimes \bk$ be the
external tensor product of the $\Sp_{2g}(\Z)$-representation
$\bV_g$ and the trivial $\Sp_{2}(\Z)$-representation $\bk$, so
$\bV_g \boxtimes \bk$ is a representation of $\Sp_{2g}(\Z) \times \Sp_2(\Z)$
that as a vector space is isomorphic to $\bV_g$.
Then the maps\footnote{In \eqref{eqn:cohseq} the map $f_g$ has domain $\bV_g$ since
there it is just regarded as a linear map between vector spaces.  We switch its
domain to $\bV_g \boxtimes \bk$ here since we are now regarding it as a representation
of $\Sp_{2g}(\Z) \times \Sp_2(\Z)$.} $f_g\colon \bV_g \boxtimes \bk \rightarrow \bV_{g+1}$ are 
$\Sp_{2g}(\Z) \times \Sp_2(\Z)$-equivariant,
where $\Sp_{2g}(\Z) \times \Sp_2(\Z)$ acts on $\bV_{g+1}$
via the standard inclusion $\Sp_{2g}(\Z) \times \Sp_2(\Z) \hookrightarrow \Sp_{2(g+1)}(\Z)$.
\end{itemize}

\begin{example}
\label{example:cohh1}
Let $\iota_g\colon \Sigma_{g}^1 \hookrightarrow \Sigma_{g+1}^1$ be the following embedding:\\
\Figure{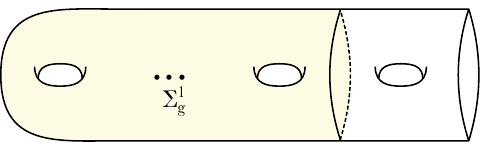}
The induced map on homology $(\iota_g)_{\ast}\colon \HH_1(\Sigma_{g}^1;\bk) \rightarrow \HH_1(\Sigma_{g+1}^1;\bk)$ fits
into a coherent sequence
\[\begin{tikzcd}
\HH_1(\Sigma_1^1;\bk) \arrow{r}{(\iota_1)_{\ast}} & \HH_1(\Sigma_2^1;\bk) \arrow{r}{(\iota_2)_{\ast}} & \HH_1(\Sigma_3^1;\bk) \arrow{r}{(\iota_3)_{\ast}} & \cdots
\end{tikzcd}\]
of representations of $\Sp_{2g}(\Z)$.
\end{example}

\begin{example}
\label{example:cohtorelli}
The map $\iota_g\colon \Sigma_{g}^1 \hookrightarrow \Sigma_{g+1}^1$ 
from Example \ref{example:cohh1} induces a map $t_g\colon \Torelli_{g}^1 \rightarrow \Torelli_{g+1}^1$ that extends
mapping classes lying in $\Torelli_{g}^1$ to $\Sigma_{g+1}^1$ by the identity.
For a fixed $d \geq 0$, the map $t_g$ induces a map $f_g\colon \HH_d(\Torelli_g^1;\bk) \rightarrow \HH_d(\Torelli_{g+1}^1;\bk)$.  These fit into
a coherent sequence
\[\begin{tikzcd}
\HH_d(\Torelli_1^1;\bk) \arrow{r}{f_1} & \HH_d(\Torelli_2^1;\bk) \arrow{r}{f_2} & \HH_d(\Torelli_3^1;\bk) \arrow{r}{f_3} & \cdots
\end{tikzcd}\]
of representations of $\Sp_{2g}(\Z)$.
\end{example}

We will use the following theorem to prove our results about the Torelli group:

\begin{maintheorem}[Stability Theorem]
\label{maintheorem:stability}
Consider a coherent sequence of representations of $\Sp_{2g}(\Z)$ over a field $\bk$ of characteristic $0$:
\[\begin{tikzcd}
\bV_1 \arrow{r}{f_1} & \bV_2 \arrow{r}{f_2} & \bV_3 \arrow{r}{f_3} & \cdots
\end{tikzcd}\]
For some $g_0 \geq 2$, assume that the following hold for $g \geq g_0$:
\begin{itemize}
\item[(i)] the cokernel of $f_{g-1}\colon \bV_{g-1} \rightarrow \bV_g$ is
a finite dimensional algebraic
representation of~\hspace{-1pt}\footnote{This cokernel is a representation of $\Sp_{2(g-1)}(\Z) \times \Sp_2(\Z)$, but we are ignoring the $\Sp_{2}(\Z)$-action.} $\Sp_{2(g-1)}(\Z)$; and
\item[(ii)] the coinvariants\footnote{If $\bW$ is a representation of a group $G$, then the
coinvariants $\bW_G$ are the largest $G$-invariant quotient of $\bW$.  This can be expressed
as $\bW/\SpanSet{$\vec{w}-x \Cdot \vec{w}$}{$\vec{w} \in W$ and $x \in G$}$.}
$(\bV_g)_{\Sp_{2g}(\Z)}$ are finite dimensional.
\end{itemize}
Then $\bV_g$ is finite dimensional for $g \geq g_0$ and an algebraic representation of $\Sp_{2g}(\Z)$ for $g \geq g_0+1$.
\end{maintheorem}

\begin{remark}
If in (i) the cokernel of $f_{g-1}$ is just assumed to be finite dimensional for $g \geq g_0$,
then our proof shows that $\bV_g$ is finite dimensional for $g \geq g_0$.
\end{remark}

\begin{remark}
We have stated Theorem~\ref{maintheorem:stability} for coherent sequences of representations
of $\Sp_{2g}(\Z)$, but the same proof works for many other kinds of representations (possibly
with worse bounds).  For instance, it also works for
coherent sequences of representations of $\SL_n(\Z)$.
\end{remark}

\subsection{Induction without a base case}

Consider a coherent sequence of representations $\{\bV_g\}$ as in Theorem
\ref{maintheorem:stability}.  One way to view Theorem~\ref{maintheorem:stability} is that it allows finiteness results
for $\bV_g$ to be proved by induction, but without a base case.  In a traditional proof
by induction, to prove that $\bV_g$ is finite-dimensional for $g \geq g_0$ you would need to prove
two things:
\begin{itemize}
\item as a base case, that $\bV_{g_0}$ is finite-dimensional; and
\item for the inductive step, that the cokernel of $f_{g-1}\colon \bV_{g-1} \rightarrow \bV_{g}$
is finite-dimensional for $g \geq g_0+1$.  
\end{itemize}
Without some further hypothesis,
a finiteness assumption about the cokernels of the $f_{g}$ implies 
nothing about the $\bV_g$ themselves.  For instance:

\begin{example}
\label{example:trivialcoh}
For all $g \geq 1$, let $\bV_g = \bk^{\infty}$ with the trivial $\Sp_{2g}(\Z)$-action and let
$f_{g}\colon \bV_{g} \rightarrow \bV_{g+1}$ be the identity map.  Then
\[\begin{tikzcd}
\bV_1 \arrow{r}{f_1} & \bV_2 \arrow{r}{f_2} & \bV_3 \arrow{r}{f_3} & \cdots
\end{tikzcd}\]
is a coherent sequence of infinite-dimensional representations of $\Sp_{2g}(\Z)$.
However, the cokernel of each $f_{g-1}\colon \bV_{g-1} \rightarrow \bV_{g}$ is $0$.
\end{example}

It is a little surprising that the weak finiteness assumption about the
coinvariants in Theorem~\ref{maintheorem:stability} allows such strong conclusions.
Since the coinvariants are the largest trivial quotient representation, this assumption
essentially just rules out things like Example \ref{example:trivialcoh}.

\subsection{Torelli proof outline}

As we will show in \S \ref{section:step1prelim}, it is straightforward to deduce that $\HH_2(\Torelli_{g,p}^b;\Q)$ is a finite dimensional algebraic representation of $\Sp_{2g}(\Z)$ from the special case of $\HH_2(\Torelli_g^1;\Q)$, so we will focus our proof outline on $\Torelli_g^1$.  We will apply Theorem~\ref{maintheorem:stability} to the coherent sequence
\[\begin{tikzcd}
\HH_2(\Torelli_1^1;\bk) \arrow{r}{f_1} & \HH_2(\Torelli_2^1;\bk) \arrow{r}{f_2} & \HH_2(\Torelli_3^1;\bk) \arrow{r}{f_3} & \cdots
\end{tikzcd}\]
of representations of $\Sp_{2g}(\Z)$ with $g_0 = 5$.

Theorem~\ref{maintheorem:stability} has two hypotheses.  Hypothesis (ii) says
that the coinvariants
$\HH_2(\Torelli_g^1;\Q)_{\Sp_{2g}(\Z)}$
are finite-dimensional for $g \geq 5$, and is an immediate consequence
of Kassabov--Putman's theorem \cite{KassabovPutman}
that $\HH_2(\Torelli_g^1;\Q)$ is finitely generated as a $\Q[\Sp_{2g}(\Z)]$-module for $g \geq 3$.  
In \S \ref{section:step1main} we will give an alternate direct proof 
that these coinvariants are finite-dimensional.

Hypothesis (i) is more substantial.  It asserts that for $g \geq 5$ the cokernel of
\begin{equation}
\label{eqn:iota}
f_{g-1}\colon \HH_2(\Torelli_{g-1}^1;\Q) \longrightarrow \HH_2(\Torelli_g^1;\Q)
\end{equation}
is a finite dimensional algebraic representation of $\Sp_{2(g-1)}(\Z)$.  Our proof of this has
three steps.

\subsubsection{Step 1 (reduction to curve stabilizers)}
\label{section:step1reduction}

The first step uses a variety of known results about the homology of the Torelli group to reduce what we
must show to a result about curve stabilizers on a closed surface.  Let $\alpha$ and $\beta$ be the following oriented curves
on $\Sigma_g$:\\
\Figure{AandBcurves}
The mapping class group acts on the set of isotopy classes of simple closed curves
on $\Sigma_g$.  Let $(\Torelli_g)_{\alpha}$ and $(\Torelli_g)_{\beta}$ be the $\Torelli_g$-stabilizers of $\alpha$ and $\beta$.
From the above figure, it is clear that both of these stabilizers contain $\Torelli_{g-1}^1$.  Let
\[\lambda\colon \HH_2((\Torelli_g)_{\alpha};\Q) \oplus \HH_2((\Torelli_g)_{\beta};\Q) \rightarrow \HH_2(\Torelli_g;\Q)\]
be the sum of the maps induced by the inclusions $(\Torelli_g)_{\alpha} \hookrightarrow \Torelli_g$
and $(\Torelli_g)_{\beta} \hookrightarrow \Torelli_g$.  Letting $f_{g-1}$ be the map \eqref{eqn:iota}, we will prove:
\begin{itemize}
\item {\bf Claim}: To prove that $\coker(f_{g-1})$ is a finite dimensional
algebraic representation of $\Sp_{2(g-1)}(\Z)$, it is enough to prove a similar result for $\coker(\lambda)$.
\end{itemize}
The idea behind the proof of this claim is as follows.  Consider the following commutative diagram
of homology groups:
\[\begin{tikzcd}
\HH_2(\Torelli_{g-1}^1;\Q) \arrow{r}{f_{g-1}} \arrow{rd}{h_{\alpha} \times h_{\beta}} & \HH_2(\Torelli_g^1;\Q) \arrow{r}{\pi} & \HH_2(\Torelli_g;\Q) \\
                                                          & \HH_2((\Torelli_g)_{\alpha};\Q) \oplus \HH_2((\Torelli_g)_{\beta};\Q) \arrow{ru}{\lambda} &
\end{tikzcd}\]
Here $\pi$ is induced by gluing a disc to $\partial \Sigma_g^1$ and extending mapping classes by $\id$, and
$h_{\alpha}$ and $h_{\beta}$ are induced by the inclusions $\Torelli_{g-1}^1 \hookrightarrow (\Torelli_g)_{\alpha}$ and
$\Torelli_{g-1}^1 \hookrightarrow (\Torelli_g)_{\beta}$.  We will show:
\begin{itemize}
\item Both $\ker(\pi)$ and $\coker(\pi)$ are finite-dimensional algebraic representations of $\Sp_{2g}(\Z)$.  This is a standard
argument using the Birman exact sequence.
\item Both $\coker(h_{\alpha})$ and $\coker(h_{\beta})$ are finite-dimensional algebraic representations of
$\Sp_{2(g-1)}(\Z)$.  This uses
a recent theorem of the authors \cite[Theorem A]{MinahanPutmanAbelian} about the first homology group of $\Torelli_g^1$
with certain infinite-dimensional twisted coefficients.  
\end{itemize}
The above claim will follow easily from these two facts and a bit of algebra.
This reduces us to proving that $\coker(\lambda)$ is a finite-dimensional algebraic representation of $\Sp_{2(g-1)}(\Z)$.  
Roughly speaking, we do this by constructing generators and relations for $\coker(\lambda)$.

\subsubsection{Step 2 (generators of cokernel)}

We identify generators for $\coker(\lambda)$
by studying the action of $\Torelli_g$ on a carefully chosen space.
Let $a = [\alpha] \in \HH_1(\Sigma_g)$ and $b = [\beta] \in \HH_1(\Sigma_g)$.
An oriented simple closed curve is an {\em $a$-curve} if its homology class is $a$ and a {\em $b$-curve} if its
homology class is $b$.  Let $\cC_a(\Sigma_g)$ be the simplicial complex whose $p$-simplices are collections $\{\gamma_0,\ldots,\gamma_k\}$
of disjoint isotopy classes of $a$-curves on $\Sigma_g$.  Similarly, define $\cC_b(\Sigma_g)$.  Simplices of these complexes
look like:\footnote{Our convention is that $a$-curves are orange and $b$-curves are blue, and we often omit the orientations.}\\
\Figure{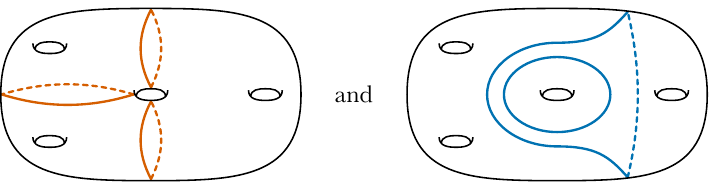}
The complexes $\cC_a(\Sigma_g)$ and $\cC_b(\Sigma_g)$ were introduced by
Putman \cite{PutmanTrick}, who proved that they are connected for $g \geq 3$.  Hatcher--Margalit \cite{HatcherMargalitGenerators}
found an alternate
proof of this, and Minahan \cite{MinahanComplex}
generalized Hatcher--Margalit's proof to show that these complexes are $(g-3)$-acyclic.\footnote{This means that $\RH_k(\cC_a)=\RH_k(\cC_b)=0$
for $k \leq g-3$.  It is not clear if they are simply connected.}

The space $\cC_{ab}(\Sigma_g)$ we use is a combination of $\cC_a(\Sigma_g)$ and $\cC_b(\Sigma_g)$ that is inspired by
the ``handle graph'' introduced by Putman \cite{PutmanSmallGenset} to find small generating sets for $\Torelli_g$.
We call it the {\em handle complex}.  It is the union of $\cC_a(\Sigma_g)$ and $\cC_b(\Sigma_g)$ with certain ``mixed simplices'' attached
that look like:\\
\Figure{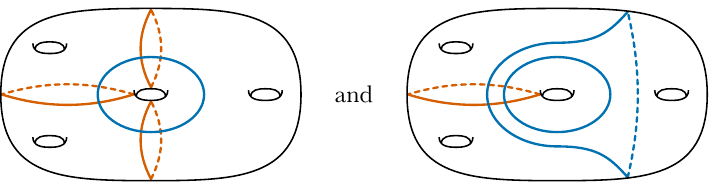}
Their key property is that a mixed simplex cannot contain both two $a$-curves and two $b$-curves.  We will define this 
carefully later.  The group $\Torelli_g$ acts on $\cC_{ab}(\Sigma_g)$.  We will prove:
\begin{itemize}
\item[(a)] $\cC_{ab}(\Sigma_g)$ is $1$-acyclic for $g \geq 4$; and
\item[(b)] the quotient $\cC_{ab}(\Sigma_g) / \Torelli_g$ is contractible.
\end{itemize}
These two results will allow us to use the action of $\Torelli_g$ on $\cC_{ab}(\Sigma_g)$
to study $\HH_2(\Torelli_g;\Q)$.
Our main tool is (Borel) equivariant homology, which for a group $G$ acting
on a space $X$ mixes together information about the topology of $X/G$ and the group homology of the
$G$-stabilizers of cells in $X$.  

Using (a), we will show that the $\Torelli_g$-equivariant homology
of $C_{ab}(\Sigma_g)$
surjects onto $\HH_2(\Torelli_g;\Q)$.  The terms
$\HH_2((\Torelli_g)_{\alpha};\Q)$ and $\HH_2((\Torelli_g)_{\beta};\Q)$
in the domain of $\lambda$ appear in the $\Torelli_g$-equivariant homology of 
$C_{ab}(\Sigma_g)$ since $\alpha$ and $\beta$ are vertices of $C_{ab}(\Sigma_g)$.  Using
(b), we will show that rest of the equivariant homology has a tractable description,
from which we will deduce generators for $\coker(\lambda)$.

\subsubsection{Step 3 (algebraicity of cokernel)}
\label{section:step3relations}

The third step identifies topologically some relations among our generators for
$\coker(\lambda)$.  Let $H = \HH_1(\Sigma_{g-1}^1;\Q)$.  We use these relations
to embed $\coker(\lambda)$ into $V=((\wedge^2 H)/\Q)^{\otimes 2}$.
Since $V$ is
a finite-dimensional algebraic representation of $\Sp_{2g}(\Z)$, this will imply
that the same is true for $\coker(\lambda)$.  This step
uses recent work of the authors \cite[Theorem A.6]{MinahanPutmanRepPresentations}
giving tools for identifying representations given by generators and relations.  This
is difficult since $\coker(\lambda)$ has infinitely many generators
and relations.  

\subsection{Outline}
Theorem~\ref{maintheorem:stability} is proved in Part \ref{part:stability}.
Theorem~\ref{maintheorem:h2torelli} is proved in Parts \ref{part:step1} -- \ref{part:step3},
which perform the three steps outlined above.

\subsection{Acknowledgments}
We would like to thank Nate Harman for helpful discussions about the material in \S \ref{section:unmixed}
and Benson Farb, Richard Hain, Dan Margalit, and Xiyan Zhong for corrections.

\part{The stability theorem}
\label{part:stability}

The proof of Theorem~\ref{maintheorem:stability} (the stability theorem) is divided
into two parts: \S \ref{section:finitedim} addresses finite dimensionality, and
\S \ref{section:algebraicity} -- \S \ref{section:algebraicityproof} addresses algebraicity.

\section{A criterion for finite dimensionality}
\label{section:finitedim}

The following theorem strengthens part
of Theorem~\ref{maintheorem:stability} (the stability theorem).  For a commutative ring $\bk$
and a group $G$, we call a $\bk[G]$-module $\bV$ a {\em representation} of
$G$ over $\bk$.  We say that $\bV$ is {\em finite dimensional} if $\bV$ is finitely generated
as a $\bk$-module.

\begin{theorem}[Finite dim criterion]
\label{theorem:finitedim}
Let $\bk$ be a commutative Noetherian ring and let $g \geq 2$.  Let $\bV_{g-1}$ be an $\Sp_{2(g-1)}(\Z)$-representation over $\bk$
and $\bV_g$ be an $\Sp_{2g}(\Z)$-representation over $\bk$.
Let $f\colon \bV_{g-1} \boxtimes \bk \rightarrow \bV_g$ be an $\Sp_{2(g-1)}(\Z) \times \Sp_2(\Z)$-equivariant map.
Assume:
\begin{itemize}
\item[(i)] the cokernel of $f$ is finite dimensional; and
\item[(ii)] the coinvariants $(\bV_g)_{\Sp_{2g}(\Z)}$ are finite dimensional.
\end{itemize}
Then $\bV_g$ is finite dimensional.
\end{theorem}
\begin{proof}
Let $G = 1 \times \Sp_2(\Z) \subset \Sp_{2g}(\Z)$ and 
let $\bW$ be the image of $f$.  The group $G$ 
acts trivially on $\bW$, and assumption (i) says that the $\bk$-submodule $\bW$ of $\bV_g$ has
finite codimension.\footnote{This simply means that $\bV_g/\bW$ is a finitely generated $\bk$-module.}
For $m \in \Sp_{2g}(\Z)$, the $\bk$-submodule $m \Cdot \bW$ of $\bV_g$ also has
finite codimension and the subgroup $m G m^{-1}$ of $\Sp_{2g}(\Z)$ acts trivially on it.
The group $\Sp_{2g}(\Z)$ is generated by finitely many conjugates of\footnote{One way to
see this is to observe that $G$ contains the transvection along a primitive element
of $0 \times \Z^2 \subset \Z^{2g}$, all transvections along primitive elements of $\Z^{2g}$ are conjugate in $\Sp_{2g}(\Z)$, and $\Sp_{2g}(\Z)$ 
is generated by
finitely many transvections along primitive elements of $\Z^{2g}$.} $G$.
Pick $m_1,\ldots,m_n \in \Sp_{2g}(\Z)$ such that $\Sp_{2g}(\Z)$ is generated by
the subgroups $\{m_i G m_i^{-1}\}_{i=1}^n$.
Set 
\[\bU = \bigcap_{i=1}^n m_i \Cdot \bW.\]
Each $m_i G m_i^{-1}$ acts trivially on $\bU$, so $\Sp_{2g}(\Z)$ itself acts trivially
on $\bU$.  Moreover, since the intersection of finitely many finite codimension submodules is
a finite codimension submodule, the submodule $\bU$ has finite codimension.  To prove that
$\bV_g$ is finite dimensional, it is therefore enough to prove that $\bU$ is finite dimensional.

Consider the short exact sequence of $\Sp_{2g}(\Z)$-representations
\begin{equation}
\label{eqn:uext}
\begin{tikzcd}
0 \arrow{r} & \bU \arrow{r} & \bV_g \arrow{r} & \bV_g / \bU \arrow{r} & 0.
\end{tikzcd}
\end{equation}
Since $\Sp_{2g}(\Z)$ acts trivially on $\bU$, we have
\[\HH_0(\Sp_{2g}(\Z);\bU) = \bU_{\Sp_{2g}(\Z)} = \bU,\]
where the subscript indicates that we are taking coinvariants.
The long exact sequence in $\Sp_{2g}(\Z)$-homology associated to \eqref{eqn:uext} thus contains the segment
\[\begin{tikzcd}
\HH_1(\Sp_{2g}(\Z);\bV_g/\bU) \arrow{r} & \bU \arrow{r} & \HH_0(\Sp_{2g}(\Z);\bV_g).
\end{tikzcd}\]
Since $\Sp_{2g}(\Z)$ is finitely generated and the quotient $\bV_g/\bU$ is finite dimensional, it follows\footnote{This uses the fact that our base ring $\bk$ is Noetherian.}
that $\HH_1(\Sp_{2g}(\Z);\bV_g/\bU)$ is finite dimensional.  Also, by assumption (ii) the coinvariants
\[\HH_0(\Sp_{2g}(\Z);\bV_g) = (\bV_g)_{\Sp_{2g}(\Z)}\]
are finite dimensional.  It follows that $\bU$ is finite dimensional, as desired.
\end{proof}

\section{A criterion for algebraicity I: statement and motivation for proof}
\label{section:algebraicity}

Theorem~\ref{theorem:finitedim} from \S \ref{section:finitedim} is a strengthening of one
part of Theorem~\ref{maintheorem:stability} (the stability theorem).  
We now turn to the following, which is a strengthening of the other part:

\begin{theorem}[Algebraicity criterion]
\label{theorem:algebraicity}
Let $\bk$ be a field of characteristic $0$ and let $g \geq 3$.  Let $\bV_{g-2}$ be an $\Sp_{2(g-2)}(\Z)$-representation over $\bk$
and $\bV_g$ be an $\Sp_{2g}(\Z)$-representation over $\bk$.
Let $f\colon \bV_{g-2} \boxtimes \bk \rightarrow \bV_g$ be an $\Sp_{2(g-2)}(\Z) \times \Sp_4(\Z)$-equivariant map.
Assume:
\begin{itemize}
\item[(a)] the representation $\bV_g$ is finite dimensional; and
\item[(b)] the cokernel of $f$ is an algebraic representation of $\Sp_{2(g-2)}(\Z)$.
\end{itemize}
Then $\bV_g$ is an algebraic representation of $\Sp_{2g}(\Z)$.
\end{theorem}

Theorem~\ref{maintheorem:stability} is an immediate consequence\footnote{The only potentially
non-obvious point here is that in Theorem~\ref{maintheorem:stability}, we are given
maps
\[\begin{tikzcd}[ampersand replacement=\&]
\bV_{g-2} \arrow{r}{f_{g-2}} \& \bV_{g-1} \arrow{r}{f_{g-1}} \& \bV_g
\end{tikzcd}\]
between representations of the appropriate symplectic groups such that $\coker(f_{g-2})$ is
a finite dimensional algebraic representation of $\Sp_{2(g-2)}(\Z)$ and
$\coker(f_{g-1})$ is a finite dimensional algebraic representation of $\Sp_{2(g-1)}(\Z)$.  To
apply Theorem~\ref{theorem:algebraicity}, we need for $\coker(f_{g-1} \circ f_{g-2})$ to
be a finite dimensional algebraic representation of $\Sp_{2(g-2)}(\Z)$.  This follows
from the fact that algebraic representations of $\Sp_{2(g-1)}(\Z)$ restrict
to algebraic representations of $\Sp_{2(g-2)}(\Z)$ along with the fact that the collection
of algebraic representations of $\Sp_{2(g-2)}(\Z)$ is closed under extensions and subquotients.}
of Theorems \ref{theorem:finitedim} and \ref{theorem:algebraicity}.

\subsection{Motivation for proof}

To motivate what we do, let us consider one way a counterexample to Theorem~\ref{theorem:algebraicity} might arise.  Fix some $\ell \geq 2$.
The prototypical non-algebraic representations of $\Sp_{2g}(\Z)$ are those that factor through the finite group $\Sp_{2g}(\Z/\ell)$.
The group $\Sp_{2g}(\Z/\ell)$ contains the subgroup $\Sp_{2(g-2)}(\Z/\ell) \times \Sp_4(\Z/\ell)$.  Via the projections
\[\Sp_{2(g-2)}(\Z/\ell) \times \Sp_4(\Z/\ell) \rightarrow \Sp_{2(g-2)}(\Z/\ell) \ \  \text{and} \ \ \Sp_{2(g-2)}(\Z/\ell) \times \Sp_4(\Z/\ell) \rightarrow \Sp_{4}(\Z/\ell),\]
we can regard representations of $\Sp_{2(g-2)}(\Z/\ell)$ and $\Sp_4(\Z/\ell)$ as representations
of the product $\Sp_{2(g-2)}(\Z/\ell) \times \Sp_4(\Z/\ell)$.
Assume there exists a nontrivial
finite-dimensional representation $\bV$ of the finite group $\Sp_{2g}(\Z/\ell)$ over $\bk$ with the following property:
\begin{itemize}
\item The restriction of $\bV$ to $\Sp_{2(g-2)}(\Z/\ell) \times \Sp_4(\Z/\ell)$ decomposes as $\bW \oplus \bW'$
with $\Sp_4(\Z/\ell)$ acting trivially on $\bW$ and $\Sp_{2(g-2)}(\Z/\ell)$ acting trivially
on $\bW'$.
\end{itemize}
Let $\bV_g = \bV$ and $\bV_{g-2} = \bW$, regarded as representations of $\Sp_{2g}(\Z)$ and $\Sp_{2(g-2)}(\Z)$ via the surjections
\[\Sp_{2g}(\Z) \longrightarrow \Sp_{2g}(\Z/\ell) \quad \text{and} \quad \Sp_{2(g-2)}(\Z) \longrightarrow \Sp_{2(g-2)}(\Z/\ell).\]
Consider the map $f\colon \bV_{g-2} \boxtimes \bk \rightarrow \bV_g$
coming from the inclusion $\bW \hookrightarrow \bV$.
We claim that this is a counterexample to Theorem~\ref{theorem:algebraicity}.  This requires checking two things:
\begin{itemize}
\item The representation $\Coker(f) = \bV / \bW \cong \bW'$ is an algebraic representation of $\Sp_{2(g-2)}(\Z)$.  This holds
since the $\Sp_{2(g-2)}(\Z)$-action on $\bW'$ is trivial.\footnote{The group $\Sp_{2(g-2)}(\Z) \times \Sp_4(\Z)$ acts on 
$\Coker(f) = \bW'$ and this action is nontrivial since $\Sp_4(\Z)$ acts nontrivially, but we only care about the action
of $\Sp_{2(g-2)}(\Z)$.}
\item The representation $\bV_g = \bV$ is a non-algebraic representation of $\Sp_{2g}(\Z)$.  This holds since the (nontrivial)
$\Sp_{2g}(\Z)$-action on it factors through a finite group.
\end{itemize}

\subsection{Outline}
It turns out that the potential existence of such $\bV$ is the main obstacle to proving Theorem~\ref{theorem:algebraicity}.
Ruling this out requires a long detour into the representation theory of finite groups,
which we do in \S \ref{section:unmixed}.  We then prove Theorem~\ref{theorem:algebraicity} in \S \ref{section:algebraicityproof}.

\section{A criterion for algebraicity II: unmixed representations of finite groups}
\label{section:unmixed}

This section studies the restriction to product subgroups of representations of finite groups.

\subsection{Mixed representations}

Let $G_1$ and $G_2$ be finite groups.  Via the projections
\[G_1 \times G_2 \rightarrow G_1 \quad \text{and} \quad G_1 \times G_2 \rightarrow G_2,\]
the irreducible representations of $G_1$ and $G_2$ over a field are irreducible 
representations of $G_1 \times G_2$.  We will call these the {\em unmixed} irreducible
representations of $G_1 \times G_2$.  The other irreducible representations of $G_1 \times G_2$ 
are the {\em mixed} irreducible
representations; they have the property that their restrictions to $G_1$ and $G_2$ are both nontrivial.\footnote{If $\bk$ is an algebraically closed
field of characteristic $0$, the mixed irreducible
representations of $G_1 \times G_2$ over $\bk$ are of the form $\bV \otimes \bW$ with $\bV$ a nontrivial irreducible representation of $G_1$
and $\bW$ a nontrivial irreducible representation of $G_2$.}
A general representation of $G_1 \times G_2$ over a field of characteristic $0$ is {\em unmixed} if 
all of its irreducible factors are unmixed.  If at least one of its irreducible
factors is mixed, then it is {\em mixed}.

\subsection{Universally mixed subgroups}

Let $G$ be a finite group and let $G_1 \times G_2$ be a subgroup of $G$.  We say that $G_1 \times G_2$
is {\em universally mixed} in $G$ if for all all representations $\bV$ of $G$ over algebraically closed
fields of characteristic $0$, the restriction $\Res^G_{G_1 \times G_2} \bV$
is either mixed or trivial.\footnote{This implies that the same holds for representations $\bV$ of $G$ over
arbitrary fields of characteristic $0$.}  We remark that it is enough to check this
for irreducible representations $\bV$.  

\subsection{\texorpdfstring{GL\textsubscript{2}}{GL2} over finite fields}

The following gives an important example of this property:

\begin{lemma}
\label{lemma:unmixedgl2}
Let $\bbF_q$ be a finite field with\footnote{While the lemma is true for $q=2$ for the (uninteresting) reason that
$\GL_1(\bbF_2) = 1$, it is false for $q=3$.  Indeed, let $G = \GL_2(\bbF_3)$ and let
$P \subset G$ be the subgroup of upper triangular
matrices.
Define a character $\rho\colon P \rightarrow \C^{\times}$ via the formula
\[\rho\left(\begin{matrix} a & u \\ 0 & b\end{matrix}\right) = \begin{cases} 1 & \text{if $a = 1$},\\ -1 & \text{if $a = -1$.}\end{cases}\]
Note that we write $\rho$ like this since $1,-1 \in \bbF_3$ are different from $1,-1 \in \C^{\times}$ despite the fact that they look the same.
Let $\C_{\rho}$ be the associated $1$-dimensional representation of $P$ and let
$\bV = \Ind_P^{G} \C_{\rho}$.  It is an easy exercise to show that $\bV$ restricts to an unmixed
representation of $\GL_1(\bbF_3) \times \GL_1(\bbF_3)$.} $q > 3$.
Then $\GL_1(\bbF_q) \times \GL_1(\bbF_q)$ is universally mixed in $\GL_2(\bbF_q)$.
\end{lemma}
\begin{proof}
Let $G = \GL_2(\bbF_q)$, let $A = \GL_1(\bbF_q) \times 1$,
and let $B = 1 \times \GL_1(\bbF_q)$.  Let $\bk$ be an algebraically closed field
of characteristic $0$ and let $\bV$ be a nontrivial irreducible representation of
$G$ over $\bk$.  It is enough to prove that $\Res^G_{A \times B} \bV$ is a mixed representation of
$A \times B$.

If $\bV$ is $1$-dimensional then the only way that the restriction of $\bV$ to
$A \times B$ can be unmixed is if the restriction 
of $\bV$ to either $A$ or $B$ is trivial.  Since
the normal closures in $G$ of both $A$ and $B$
are\footnote{This uses the fact that $q \neq 2$.} all of $G$, this would imply that $\bV$ is the trivial representation
of $G$, contrary to our assumptions.  We can thus assume without
loss of generality that $\bV$ has dimension greater than $1$.

If $\Gamma$ is a group and $\chi\colon \Gamma \rightarrow \bk^{\times}$ is a character,
then denote by $\bk_{\chi}$ the associated $1$-dimensional representation of $\Gamma$.
Since $A \times B$ is abelian and $\bk$ is algebraically closed, proving that the restriction of $\bV$ to $A \times B$ is mixed
is equivalent to finding nontrivial characters
$\chi_1 \colon A \rightarrow \bk^{\times}$ and $\chi_2 \colon B \rightarrow \bk^{\times}$ such that
letting $\chi_1 \chi_2\colon A \times B \rightarrow \bk^{\times}$ be their product,
the restriction of $\bV$ to $A \times B$ contains $\bk_{\chi_1 \chi_2}$ as a subrepresentation.

For $a,b \in \bk^{\times}$ let $\diag(a,b) \in A \times B$
be the associated diagonal matrix.  Let $Z < G$ be the central subgroup of diagonal matrices
$\diag(z,z)$.  
By Schur's Lemma, there is a character $\lambda\colon Z \rightarrow \bk^{\times}$ called
the {\em central character} such that for all $z \in Z$, the element $z$ acts on $\bV$
as multiplication by $\lambda(z)$.  

If $\chi_1\colon A \rightarrow \bk^{\times}$ is a character and $\bL$ is a subrepresentation
of $\Res^G_A \bV$ with $\bL \cong \bk_{\chi_1}$, then
for $b \in \bbF_q^{\times}$ and $\vec{x} \in \bL$ we have
\[\diag(1,b) \Cdot \vec{x} = \diag(b^{-1},1) \diag(b,b) \Cdot \vec{x} = \chi_1(\diag(b,1))^{-1} \lambda(\diag(b,b)) \vec{x}.\]
In other words, $\bL$ is also preserved by $B$.  Moreover, letting 
$\chi_2 \colon B \rightarrow \bk^{\times}$ be defined via the formula
\begin{equation}
\label{eqn:chi2formula}
\chi_2(\diag(1,b)) = \chi_1(\diag(b,1))^{-1} \lambda(\diag(b,b)) \quad \text{for all $b \in \bbF_q^{\times}$},
\end{equation}
as an $A \times B$-representation we have $\bL \cong \bk_{\chi_1 \chi_2}$.  It follows that to
prove the lemma, it is enough to find such an $\bL$ with both $\chi_1$ and $\chi_2$ nontrivial.

Let $a_0 \in \bbF_q^{\times}$ be a generator of this cyclic group.  Since $q > 3$ and $A \cong \bbF_q^{\times}$ and $\bk$
is algebraically closed, we can find a nontrivial character
$\chi_1\colon A \rightarrow \bk^{\times}$ with
\[\chi_1(\diag(a_0,1)) \neq \lambda(\diag(a_0,a_0)).\]
It follows that the character $\chi_2\colon B \rightarrow \bk^{\times}$ defined
in \eqref{eqn:chi2formula}
satisfies $\chi_2(\diag(0,a_0)) \neq 1$, and in particular is also nontrivial.
It is therefore enough to prove that $\Res^G_A \bV$
contains a subrepresentation $\bL$ with $\bL \cong \bk_{\chi_1}$.  In
fact, we will prove the following:

\begin{unnumberedclaim}
For all irreducible representations $\bV$ of $G$ of dimension greater than $1$,
the representation $\Res^G_A \bV$ contains a copy of the left
regular representation $\bk[A]$ of $A$.  In particular, it contains a copy
of every irreducible representation of $A$.
\end{unnumberedclaim}

Let $U < G$ be the unipotent subgroup consisting of matrices
\[\text{$\left(\begin{matrix} 1 & u \\ 0 & 1 \end{matrix}\right)$ with $u \in \bbF_q$}\]
and let $M < G$ be the mirabolic (``miraculous parabolic'') subgroup consisting of matrices
\[\text{$\left(\begin{matrix} a & u \\ 0 & 1 \end{matrix}\right)$ with $a \in \bbF_q^{\times}$ and $u \in \bbF_q$}.\]
The group $U$ is isomorphic to the additive group of $\bbF_q$.
Let $\rho\colon U \rightarrow \bk^{\times}$ be an arbitrary nontrivial character and 
let $\bW = \Ind_U^M \bk_{\rho}$.  Then:
\begin{itemize}
\item $\bW$ is an irreducible representation of $M$ (see \cite[Theorem 6.1]{PiatetskiShapiro}); and
\item $\Ind_M^G \bW$ is the direct sum of one copy of every representation of $G$ of dimension greater than $1$ (see \cite[Theorem 16.1]{PiatetskiShapiro}).
\end{itemize}
In particular, since our representation $\bV$ is irreducible and has dimension greater than $1$ we can apply Frobenius reciprocity and see that
\[\bk = \Hom_G(\Ind_M^G \bW,\bV) = \Hom_M(\bW,\Res^G_M \bV).\]
Since $\bW$ is an irreducible representation of $M$, it follows that $\bW$ injects into 
$\Res^G_M \bV$.  Recalling that we are trying to prove that
$\Res^G_A \bV$ contains a copy of $\bk[A]$, 
it is thus enough to prove that $\Res^M_A \bW \cong \bk[A]$.  For this, we have $M = U \rtimes A$, so
\[\bW = \Ind_U^M \bk_{\rho} = \bigoplus_{a \in \bbF_q^{\times}} \diag(a,1) \Cdot \bk_{\rho}.\]
The action of $A$ on $\bW$ permutes these one-dimensional factors simply transitively.  It follows
that the restriction of $\bW$ to $A$ is isomorphic to $\bk[A]$, as desired.  
\end{proof}

\subsection{\texorpdfstring{GL\textsubscript{n}}{GLn} over finite fields}

The following generalizes Lemma~\ref{lemma:unmixedgl2}:

\begin{proposition}
\label{proposition:unmixedgl}
Let $\bbF_q$ be a finite field and $n_1,n_2 \geq 1$.  Set $n=n_1+n_2$.  Assume that $q>3$ if $n_1 = 1$ or
$n_2 = 1$.  Then 
$\GL_{n_1}(\bbF_q) \times \GL_{n_2}(\bbF_q)$ is universally mixed in $\GL_n(\bbF_q)$.
\end{proposition}

The proof of Proposition~\ref{proposition:unmixedgl} requires some preliminary lemmas.

\begin{lemma}
\label{lemma:pushup}
For $i=1,2$ let $\Gamma_i$ be a finite
group and let $G_i < \Gamma_i$ be a subgroup.  Let $\bV$ be a representation of $\Gamma_1 \times \Gamma_2$ over
a field of characteristic $0$ 
whose restriction to $G_1 \times G_2$ is mixed.  Then $\bV$ is mixed.
\end{lemma}
\begin{proof}
Let $\bW$ be a mixed irreducible subrepresentation of $\Res^{\Gamma_1 \times \Gamma_2}_{G_1 \times G_2} \bV$.
By definition, both $G_1$ and $G_2$ act nontrivially on $\bW$.  
Let $\bW'$ be the irreducible subrepresentation of $\bV$ such
that $\Res^{\Gamma_1 \times \Gamma_2}_{G_1 \times G_2} \bW'$ contains $\bW$.  Since $G_1$ and $G_2$
act nontrivially on $\bW$, both $\Gamma_1$ and $\Gamma_2$ act nontrivially on $\bW'$.  It follows
that $\bW'$ is a mixed irreducible representation of $\Gamma_1 \times \Gamma_2$, so $\bV$ is a mixed
representation of $\Gamma_1 \times \Gamma_2$.
\end{proof}

\begin{lemma}[Poison subgroup] 
\label{lemma:poison} 
Let $\Gamma$ be a finite group and let $\Gamma_1 \times \Gamma_2$ be a subgroup of $\Gamma$.  Let $G < \Gamma$
be a subgroup, and for $i=1,2$ let $G_i < \Gamma_i$ be a subgroup such that $G_1 \times G_2 < G$.  
Assume that $G_1 \times G_2$ is universally mixed in $G$ and that the $\Gamma$-normal closure of $G_1 \times G_2$
contains $\Gamma_1 \times \Gamma_2$.
Then $\Gamma_1 \times \Gamma_2$ is a universally mixed subgroup of $\Gamma$.
\end{lemma}
\begin{proof}
Let $\bV$ be a representation of $\Gamma$ over an algebraically closed field of characteristic $0$.
We must prove that $\Res^{\Gamma}_{\Gamma_1 \times \Gamma_2} \bV$ is either
trivial or mixed.  Since $G_1 \times G_2$ is universally mixed
in $G$, the restriction 
\[\Res^{G}_{G_1 \times G_2} \Res^{\Gamma}_G \bV = \Res^{\Gamma}_{G_1 \times G_2} \bV\]
is either trivial or mixed.  

If $\Res^{\Gamma}_{G_1 \times G_2} \bV$ is trivial, then the kernel of the action of $\Gamma$
on $\bV$ contains $G_1 \times G_2$.  Since the $\Gamma$-normal closure of $G_1 \times G_2$ contains
$\Gamma_1 \times \Gamma_2$, this implies that
the kernel of the action of $\Gamma$ on $\bV$ contains $\Gamma_1 \times \Gamma_2$,
i.e., that $\Res^{\Gamma}_{\Gamma_1 \times \Gamma_2} \bV$ is trivial.  If instead
$\Res^{\Gamma}_{G_1 \times G_2} \bV$ is mixed, then Lemma~\ref{lemma:pushup} implies
that $\Res^{\Gamma}_{\Gamma_1 \times \Gamma_2} \bV$ is also mixed.
\end{proof}

\begin{proof}[Proof of Proposition~\ref{proposition:unmixedgl}]
We divide the proof into two cases.

\begin{case}{1}
\label{case:unmixedglq2}
$q > 3$.
\end{case}

Lemma~\ref{lemma:unmixedgl2} says that $\GL_1(\bbF_q) \times \GL_1(\bbF_q)$ is
universally mixed in $\GL_2(\bbF_q)$.
Embed $\GL_2(\bbF_q)$ into $\GL_n(\bbF_q)$ such that
$\GL_1(\bbF_q) \times 1 \subset \GL_2(\bbF_q)$ maps to
$\GL_{n_1}(\bbF_q) \times 1$ and $1 \times \GL_1(\bbF_q) \subset \GL_2(\bbF_q)$ maps
to $1 \times \GL_{n_2}(\bbF_q)$.  Since $q \neq 2$, the normal closure of $\GL_1(\bbF_q) \times \GL_1(\bbF_q)$
in $\GL_{n}(\bbF_q)$ is $\GL_{n}(\bbF_q)$,
so this embedding satisfies
the conditions of the poison subgroup lemma (Lemma~\ref{lemma:poison}).
This lemma implies that $\GL_{n_1}(\bbF_q) \times \GL_{n_2}(\bbF_q)$ is
universally mixed in $\GL_n(\bbF_q)$, as desired.

\begin{case}{2}
$q \in \{2,3\}$.  Our hypotheses thus say that $n_1,n_2 \geq 2$.
\end{case}

Since $q^2 > 3$, Lemma~\ref{lemma:unmixedgl2} says that $\GL_1(\bbF_{q^2}) \times \GL_1(\bbF_{q^2})$ is
universally mixed in $\GL_{2}(\bbF_{q^2})$.
Embed the group $\GL_2(\bbF_{q^2})$ of $\bbF_{q^2}$-linear automorphisms of
$\bbF_{q^2}^2$ in the group $\GL_4(\bbF_{q})$ of $\bbF_{q}$-linear
automorphisms of $\bbF_q^4$ by regarding $\bbF_{q^2}$ as a $2$-dimensional vector space over $\bbF_q$.
This maps 
$\GL_1(\bbF_{q^2}) \times 1$ to $\GL_2(\bbF_q) \times 1$ and
$1 \times \GL_1(\bbF_{q^2})$ to $1 \times \GL_2(\bbF_q)$.  

We can now embed $\GL_4(\bbF_q)$ into $\GL_n(\bbF_q)$ such that $\GL_2(\bbF_q) \times 1$
maps to $\GL_{n_1}(\bbF_q) \times 1$ and $1 \times \GL_2(\bbF_q)$ maps
to $1 \times \GL_{n_2}(\bbF_q)$.  The normal closure of $\GL_1(\bbF_{q^2}) \times \GL_1(\bbF_{q^2})$ in
$\GL_n(\bbF_q)$ is $\GL_n(\bbF_q)$, so this embedding satisfies the conditions
of the poison subgroup lemma (Lemma~\ref{lemma:poison}).
We conclude that $\GL_{n_1}(\bbF_q) \times \GL_{n_2}(\bbF_q)$ is
universally mixed in $\GL_n(\bbF_q)$.
\end{proof}

\subsection{\texorpdfstring{GL\textsubscript{n}}{GLn} over integers mod \texorpdfstring{$\ell$}{l}}

Here is another example of a universally mixed subgroup:

\begin{proposition}
\label{proposition:congruencemixed}
Let $\ell \geq 2$ and $n_1,n_2 \geq 2$.  Set $n=n_1+n_2$.  Then
$\GL_{n_1}(\Z/\ell) \times \GL_{n_2}(\Z/\ell)$ is universally mixed in $\GL_n(\Z/\ell)$.
\end{proposition}
\begin{proof}
Write $\ell$ as a product of powers of distinct primes: $\ell = p_1^{d_1} \cdots p_m^{d_m}$.  By the
Chinese remainder theorem, we have
\[\GL_n(\Z/\ell) \cong \GL_n(\Z/p_1^{d_1}) \times \cdots \times \GL_n(\Z/p_m^{d_m}).\]
Lemma~\ref{lemma:congruenceprimemixed} below says that $\GL_{n_1}(\Z/p_i^{d_i}) \times \GL_{n_2}(\Z/p_i^{d_i})$ is
universally mixed in $\GL_n(\Z/p_i^{d_i})$.  Assuming this, let
$\bV$ be a representation of $\GL_n(\Z/\ell)$ over an algebraically closed 
field of characteristic $0$.  

Since $\GL_{n_1}(\Z/p_i^{d_i}) \times \GL_{n_2}(\Z/p_i^{d_i})$ is
universally mixed in $\GL_n(\Z/p_i^{d_i})$, it is also universally mixed in $\GL_n(\Z/\ell)$.  The restriction of
$\bV$ to each $\GL_{n_1}(\Z/p_i^{d_i}) \times \GL_{n_2}(\Z/p_i^{d_i})$ is thus either trivial or mixed.  If
all these restrictions are trivial, then the restriction of $\bV$ to $\GL_{n_1}(\Z/\ell) \times \GL_{n_2}(\Z/\ell)$ is
trivial.  If one of these restrictions is mixed, then Lemma~\ref{lemma:pushup} implies that 
the restriction of $\bV$ to $\GL_{n_1}(\Z/\ell) \times \GL_{n_2}(\Z/\ell)$ is mixed. 
\end{proof}

The above proof required:\footnote{The hypotheses of Lemma~\ref{lemma:congruenceprimemixed} are more
general than is needed for Proposition~\ref{proposition:congruencemixed}, but are exactly what
is needed for the proof of Lemma~\ref{lemma:congruenceprimemixed}.}

\begin{lemma}
\label{lemma:congruenceprimemixed}
Let $p^d$ be a prime power and $n_1,n_2 \geq 1$.  Set $n=n_1+n_2$.  If $n_1 = 1$ or $n_2 = 1$, then
assume that $p>3$.  Then
$\GL_{n_1}(\Z/p^d) \times \GL_{n_2}(\Z/p^d)$ is universally mixed in $\GL_n(\Z/p^d)$.
\end{lemma}
\begin{proof}
The proof is by induction on $d$.  The base case $d=1$ is Proposition~\ref{proposition:unmixedgl}.
For the inductive step, assume that $d \geq 2$ and that the lemma holds for $\GL_n(\Z/p^{d-1})$.  For
$m \geq 1$, define 
\begin{equation}
\label{eqn:definek}
\KL_m(\Z/p^d) = \ker(\GL_m(\Z/p^d) \rightarrow \GL_m(\Z/p^{d-1})).
\end{equation}
We will prove in Lemma~\ref{lemma:kuniversallmixed} below that $\KL_{n_1}(\Z/p^d) \times \KL_{n_2}(\Z/p^d)$ is universally mixed in $\GL_n(\Z/p^d)$.  
To see that this implies
the lemma, consider a representation $\bV$ of $\GL_n(\Z/p^d)$ over an algebraically closed field of characteristic $0$.  Since $\KL_{n_1}(\Z/p^d) \times \KL_{n_2}(\Z/p^d)$ is universally
mixed in $\GL_n(\Z/p^d)$, the restriction of $\bV$ to $\KL_{n_1}(\Z/p^d) \times \KL_{n_2}(\Z/p^d)$ is either trivial or mixed.

If it is mixed, then Lemma~\ref{lemma:pushup} implies that the restriction of $\bV$ to $\GL_{n_1}(\Z/p^{d}) \times \GL_{n_2}(\Z/p^{d})$
is mixed.  If it is trivial, then the kernel of the action of $\GL_n(\Z/p^d)$ on $\bV$ contains $\KL_{n_1}(\Z/p^d) \times \KL_{n_2}(\Z/p^d)$.  The
$\GL_n(\Z/p^d)$-normal closure of $\KL_{n_1}(\Z/p^d) \times \KL_{n_2}(\Z/p^d)$ is $\KL_n(\Z/p^d)$, so we deduce that the kernel of the action of $\GL_n(\Z/p^d)$ on $\bV$
contains $\KL_n(\Z/p^d)$.  This implies that the action of $\GL_n(\Z/p^d)$ on $\bV$ factors through $\GL_n(\Z/p^{d-1})$.  

Regarding
$\bV$ as a representation of $\GL_n(\Z/p^{d-1})$, our inductive hypothesis implies that the restriction of $\bV$ to
$\GL_{n_1}(\Z/p^{d-1}) \times \GL_{n_2}(\Z/p^{d-1})$ is either trivial or mixed.  This implies the analogous result
for $\GL_{n_1}(\Z/p^{d}) \times \GL_{n_2}(\Z/p^{d})$, finishing the proof.
\end{proof}

The following lemma was invoked during the proof of Lemma~\ref{lemma:congruenceprimemixed}.  It
uses the notation $\KL_m(\Z/p^d)$ from \eqref{eqn:definek}.

\begin{lemma} 
\label{lemma:kuniversallmixed}
Let $p^d$ be a prime power with $d \geq 2$ and let $n_1,n_2 \geq 1$.  Set $n=n_1+n_2$.  If $n_1 = 1$ or $n_2 = 1$, then
assume that\footnote{The condition $p \geq 3$ is not a typo.  Lemma~\ref{lemma:congruenceprimemixed}
required $p > 3$ here, but Lemma~\ref{lemma:kuniversallmixed} is true when $p \geq 3$ and the proof
naturally gives the more general statement.} $p \geq 3$.
Then $\KL_{n_1}(\Z/p^d) \times \KL_{n_2}(\Z/p^d)$ is universally mixed in $\GL_n(\Z/p^d)$.
\end{lemma}
\begin{proof}
Let $m \geq 1$.  We start by clarifying the nature of $\KL_m(\Z/p^d)$.
Elements of $\KL_m(\Z/p^d)$ can be written as $\bbI_m + p^{d-1} A$ with $A$ an $m \times m$ matrix over $\Z/p^d$.  The value
of $\bbI_m + p^{d-1} A$
only depends on the image of $A$ in the set of matrices over $\Z/p \cong \bbF_p$.  Since
\[(\bbI_n + p^{d-1} A)(\bbI_n + p^{d-1} B) = \bbI_n + p^{d-1}(A+B) + p^{2d-2} AB = \bbI_n + p^{d-1}(A+B),\]
we deduce that $\KL_m(\Z/p^d)$ is isomorphic to the additive group
$\Mat_m(\bbF_p)$ of $m \times m$ matrices over $\bbF_p$.  Under this isomorphism, the conjugation action
of $\GL_m(\Z/p^d)$ on its normal subgroup $\KL_m(\Z/p^d)$ is identified
with the conjugation action of $\GL_m(\Z/p)$ on $\Mat_m(\bbF_p)$ via the surjection
$\GL_m(\Z/p^d) \twoheadrightarrow \GL_m(\Z/p)$.

We return to proving that $\KL_{n_1}(\Z/p^d) \times \KL_{n_2}(\Z/p^d)$ is universally mixed in $\GL_n(\Z/p^d)$.
Arguing via the poison subgroup lemma (Lemma~\ref{lemma:poison}) as
in the proof of Case \ref{case:unmixedglq2} of the proof of Proposition~\ref{proposition:unmixedgl},
to prove that 
\[\KL_{n_1}(\Z/p^d) \times \KL_{n_2}(\Z/p^d) \cong \Mat_{n_1}(\bbF_p) \times \Mat_{n_2}(\bbF_p)\] 
is universally mixed in $\GL_n(\Z/p^d)$, it is enough handle the following small $n_i$ cases:

\begin{case}{1}
\label{case:k2pnot2}
$p \geq 3$ and $n_1 = n_2 = 1$, so $n=2$.
\end{case}

Consider a representation $\bW$ of $\GL_2(\Z/p^d)$ over an algebraically closed field $\bk$ of characteristic $0$.  
If the restriction of $\bW$ to $\KL_1(\Z/p^d) \times \KL_1(\Z/p^d)$ is trivial,
we are done.  Assume, therefore, that this restriction is nontrivial.  We must prove that it is mixed.  For this, we
will study it as a representation of the larger group $\KL_2(\Z/p^d)$.  Let $\bU$ be the restriction
of $\bW$ to $\KL_2(\Z/p^d)$.  We know that
the restriction of $\bU$ to $\KL_1(\Z/p^d) \times \KL_1(\Z/p^d)$ is nontrivial, and our goal
is to prove that this restriction is mixed.

Since $\KL_2(\Z/p^d) \cong \Mat_2(\bbF_p)$ is abelian and $\bk$ is algebraically closed, the irreducible representations of
$\KL_2(\Z/p^d)$ over $\bk$ are the one-dimensional representations associated to characters $\chi \in \Hom(\KL_2(\Z/p^d),\bk^{\times})$.
For $\chi \in \Hom(\KL_2(\Z/p^d),\bk^{\times})$, let $\bU_{\chi}$ denote the $\chi$-eigenspace:
\[\bU_{\chi} = \Set{$x \in \bU$}{$m \Cdot x = \chi(m) x$ for all $m \in \KL_2(\Z/p^d)$}.\]
We thus have
\[\bU = \bigoplus_{\chi \in \Hom(\KL_2(\Z/p^d),\bk^{\times})} \bU_{\chi}.\]
To prove that the restriction of $\bU$ to $\KL_1(\Z/p^d) \times \KL_1(\Z/p^d)$ 
is mixed, we must find some $\chi \in \Hom(\KL_2(\Z/p^d),\bk^{\times})$ such that:
\begin{itemize}
\item $\bU_{\chi} \neq 0$; and
\item $\chi$ restricts to a nontrivial character on both $\KL_1(\Z/p^d) \times 1$ and $1 \times \KL_1(\Z/p^d)$.
\end{itemize}
Since the restriction of $\bU$ to $\KL_1(\Z/p^d) \times \KL_1(\Z/p^d)$ is nontrivial, we can find some $\chi_0 \in \Hom(\KL_2(\Z/p^d),\bk^{\times})$ such that
$\bU_{\chi_0} \neq 0$ and $\chi_0$ restricts to a nontrivial character on either $\KL_1(\Z/p^d) \times 1$ or $1 \times \KL_1(\Z/p^d)$.
We will assume that $\chi_0$ restricts to a nontrivial character on $\KL_1(\Z/p^d) \times 1$; 
the other case is identical up to changes in notation.

The conjugation action of $\GL_2(\Z/p^d)$ on its normal subgroup $\KL_2(\Z/p^d)$ induces an action of $\GL_2(\Z/p^d)$
on $\Hom(\KL_2(\Z/p^d),\bk^{\times})$.  What is more, the action of $\GL_2(\Z/p^d)$ on 
\[\bU = \Res^{\GL_2(\Z/p^d)}_{\KL_2(\Z/p^d)} \bW\]
permutes the $\bU_{\chi}$ with $g \in \GL_2(\Z/p^d)$ taking $\bU_{\chi}$ to $\bU_{g \Cdot \chi}$.  Since
$\bU_{\chi_0} \neq 0$, we also have $\bU_{g \Cdot \chi_0} \neq 0$.  It follows that it is enough to find
some $g \in \GL_2(\Z/p^d)$ such that $g \Cdot \chi_0$ restricts to a nontrivial character on both $\KL_1(\Z/p^d) \times 1$ and $1 \times \KL_1(\Z/p^d)$.

As we said when describing $\KL_m(\Z/p^d)$ above, the action of $\GL_2(\Z/p^d)$ on $\KL_2(\Z/p^d) \cong \Mat_2(\bbF_p)$ comes from the conjugation action of $\GL_2(\bbF_p)$ on
$\Mat_2(\bbF_p)$ via the surjection
\[\begin{tikzcd}
\GL_2(\Z/p^d) \arrow[two heads]{r} & \GL_2(\Z/p) = \GL_2(\bbF_p).
\end{tikzcd}\]
Since
\[\Mat_2(\bbF_p) \cong \Hom(\bbF_p^2,\bbF_p^2) = (\bbF_p^2)^{\ast} \otimes \bbF_p^2,\]
we see that as a representation of $\GL_2(\bbF_p)$ the vector space $\Mat_2(\bbF_p)$ is self-dual.  Since
$\Char(\bk)=0$, this
implies that there is an $\GL_2(\bbF_p)$-equivariant isomorphism
\[\Hom(\KL_2(\Z/p^d),\bk^{\times}) = \Hom(\Mat_2(\bbF_p),\bk^{\times}) \cong \Mat_2(\bbF_p).\]
Let $X_0 \in \Mat_2(\bbF_p)$ be the image of $\chi_0 \in \Hom(\KL_2(\Z/p^d),\bk^{\times})$ under this isomorphism.
Since $\chi_0$ restricts to a nontrivial character on $\KL_1(\Z/p^d) \times 1$, the $(1,1)$-entry of the $2 \times 2$ matrix $X_0$ is nonzero.
Our goal is to find some $g \in \GL_2(\bbF_p)$ such that the both the $(1,1)$- and the $(2,2)$-entries of $g X_0 g^{-1}$ are nonzero.

If the $(2,2)$-entry of $X_0$ is already nonzero, there is nothing to prove, so we can assume it is zero.  Write
\[\text{$X_0 = \left(\begin{matrix} a & b \\ c & 0 \end{matrix}\right)$ with $a,b,c \in \bbF_p$ and $a \neq 0$}.\]
It is enough to deal with the following three cases:
\begin{itemize}
\item $b=c=0$.  Since $p \geq 3$, what we want follows from
\[\left(\begin{matrix} 2 & 1 \\ 1 & 1 \end{matrix}\right) \left(\begin{matrix} a & 0 \\ 0 & 0 \end{matrix}\right) \left(\begin{matrix} 2 & 1 \\ 1 & 1 \end{matrix}\right)^{-1} 
= \left(\begin{matrix} 2a & 0 \\ a & 0 \end{matrix}\right) \left(\begin{matrix} 1 & -1 \\ -1 & 2 \end{matrix}\right)
= \left(\begin{matrix} 2a & -2a \\ a & -a \end{matrix}\right).\]
\item $b \neq 0$.  Since $p \geq 3$, we can find some $x \in \bbF_p$ with $x \neq 0$ and $a-bx \neq 0$ and what
we want follows from
\[\left(\begin{matrix} 1 & 0 \\ x & 1 \end{matrix}\right) \left(\begin{matrix} a & b \\ c & 0 \end{matrix}\right) \left(\begin{matrix} 1 & 0 \\ x & 1 \end{matrix}\right)^{-1} 
= \left(\begin{matrix} a & b \\ ax+c & bx \end{matrix}\right) \left(\begin{matrix} 1 & 0 \\ -x & 1 \end{matrix}\right)
= \left(\begin{matrix} a-bx & b \\ ax+c-bx^2 & bx \end{matrix}\right).\]
\item $c \neq 0$.  Since $p \geq 3$, we can find some $x \in \bbF_p$ with $x \neq 0$ and $a+cx \neq 0$ and what
we want follows from
\[\left(\begin{matrix} 1 & x \\ 0 & 1 \end{matrix}\right) \left(\begin{matrix} a & b \\ c & 0 \end{matrix}\right) \left(\begin{matrix} 1 & x \\ 0 & 1 \end{matrix}\right)^{-1} 
= \left(\begin{matrix} a+cx & b \\ c & 0 \end{matrix}\right) \left(\begin{matrix} 1 & -x \\ 0 & 1 \end{matrix}\right)
= \left(\begin{matrix} a+cx & -ax-cx^2+b \\ c & -cx \end{matrix}\right).\]
\end{itemize}

\begin{case}{2}
$p = 2$ and either $(n_1,n_2) = (1,2)$ or $(n_1,n_2) = (2,1)$, so $n=3$.
\end{case}

For $\bbF_2$, the argument in Case \ref{case:k2pnot2} fails only at the last step, and in fact keeping in mind that $1$ is the only nonzero element of $\bbF_2$ one can check for instance that there does not exist
\[\text{$\left(\begin{matrix} x & y \\ z & w \end{matrix}\right) \in \GL_2(\bbF_2)$ such that $\left(\begin{matrix} x & y \\ z & w \end{matrix}\right) \left(\begin{matrix} 1 & 0 \\ 0 & 0 \end{matrix}\right) \left(\begin{matrix} x & y \\ z & w \end{matrix}\right)^{-1} = \left(\begin{matrix} 1 & r \\ s & 1 \end{matrix}\right)$ with $r,s \in \bbF_2$}.\]
This is why we must go up to $3 \times 3$ matrices.

The cases $(n_1,n_2) = (1,2)$ and $(n_1,n_2) = (2,1)$ are identical up to changes in notation, so we will explain what to do
for $(n_1,n_2) = (1,2)$.  An argument identical to the one in in Case \ref{case:k2pnot2} shows that it is enough to prove the
following:
\begin{itemize}
\item Consider
\[\text{$X_0 = \left(\begin{matrix} a_{11} & a_{12} & a_{13} \\ a_{21} & a_{22} & a_{23} \\ a_{31} & a_{32} & a_{33} \end{matrix}\right) \in \Mat_3(\bbF_2)$ with either $a_{11} \neq 0$ or $\left(\begin{matrix} a_{22} & a_{23} \\ a_{32} & a_{33} \end{matrix}\right) \neq 0$.}\]
Then there exists some $g \in \GL_3(\bbF_2)$ such that
\[\text{$g X_0 g^{-1} = \left(\begin{matrix} a'_{11} & a'_{12} & a'_{13} \\ a'_{21} & a'_{22} & a'_{23} \\ a'_{31} & a'_{32} & a'_{33} \end{matrix}\right) \in \Mat_3(\bbF_2)$ with both $a'_{11} \neq 0$ and $\left(\begin{matrix} a'_{22} & a'_{23} \\ a'_{32} & a'_{33} \end{matrix}\right) \neq 0$.}\]
\end{itemize}
Since $\Mat_3(\bbF_2)$ has $2^9 = 512$ elements there are finitely many cases to check, and since $\GL_3(\bbF_2)$ has
$168$ elements this is easily done with a computer.  We omit the details.
\end{proof}

\subsection{\texorpdfstring{Sp\textsubscript{2g}}{Sp2g} over integers mod \texorpdfstring{$\ell$}{l}}

Our final example of a universally mixed subgroup is:

\begin{proposition}
\label{proposition:spuniversal}
Let $\ell \geq 2$ and $g_1,g_2 \geq 2$.  Set $g=g_1+g_2$.  Then
$\Sp_{2g_1}(\Z/\ell) \times \Sp_{2g_2}(\Z/\ell)$ is universally mixed in $\Sp_{2g}(\Z/\ell)$.
\end{proposition}
\begin{proof}
Proposition~\ref{proposition:congruencemixed} says that $\GL_{g_1}(\Z/\ell) \times \GL_{g_2}(\Z/\ell)$ is
universally mixed in $\GL_g(\Z/\ell)$.  We will prove the proposition by embedding $\GL_{g}(\Z/\ell)$ into
$\Sp_{2g}(\Z/\ell)$ and applying the poison subgroup lemma (Lemma~\ref{lemma:poison}).

Let $\{a_1,b_1,\ldots,a_g,b_g\}$ be a symplectic basis for $(\Z/\ell)^{2g}$.  Let
$I \cong (\Z/\ell)^{g}$ be the span of $\{a_1,\ldots,a_g\}$ and let
$J \cong (\Z/\ell)^{g}$ be the span of $\{b_1,\ldots,b_g\}$, so $(\Z/\ell)^{2g} = I \oplus J$.
Embed $\GL_g(\Z/\ell)$ into $\Sp_{2g}(\Z/\ell)$ by letting a matrix $M \in \GL_g(\Z/\ell)$ act
on $I \cong (\Z/\ell)^g$ via the action of $M \in \GL_g(\Z/\ell)$ on $(\Z/\ell)^g$ and on $J$
via the action of $(M^t)^{-1} \in \GL_g(\Z/\ell)$ on $(\Z/\ell)^g$.

Under this embedding, $\GL_{g_1}(\Z/\ell) \times 1$ maps to $\Sp_{2g_1}(\Z/\ell) \times 1$ and
$1 \times \GL_{g_2}(\Z/\ell)$ maps to $1 \times \Sp_{2g_2}(\Z/\ell)$.  Moreover, the normal closure
of $\GL_{g_1}(\Z/\ell) \times \GL_{g_2}(\Z/\ell)$ in $\Sp_{2g}(\Z/\ell)$ is $\Sp_{2g}(\Z/\ell)$.  The
conditions of the poison subgroup lemma (Lemma~\ref{lemma:poison}) are thus satisfied, so
using it we conclude that $\Sp_{2g_1}(\Z/\ell) \times \Sp_{2g_2}(\Z/\ell)$ is universally mixed in $\Sp_{2g}(\Z/\ell)$.
\end{proof}

\section{A criterion for algebraicity III: proof of algebraicity criterion}
\label{section:algebraicityproof}

We finally use our results to prove our algebraicity criterion (Theorem~\ref{theorem:algebraicity}):

\newtheorem*{theorem:algebraicity}{Theorem~\ref{theorem:algebraicity}}
\begin{theorem:algebraicity}[Algebraicity criterion]
Let $\bk$ be a field of characteristic $0$ and let $g \geq 3$.  Let $\bV_{g-2}$ be an $\Sp_{2(g-2)}(\Z)$-representation over $\bk$
and $\bV_g$ be an $\Sp_{2g}(\Z)$-representation over $\bk$.
Let $f\colon \bV_{g-2} \boxtimes \bk \rightarrow \bV_g$ be an $\Sp_{2(g-2)}(\Z) \times \Sp_4(\Z)$-equivariant map.
Assume:
\begin{itemize}
\item[(a)] the representation $\bV_g$ is finite dimensional; and
\item[(b)] the cokernel of $f$ is an algebraic representation of $\Sp_{2(g-2)}(\Z)$.
\end{itemize}
Then $\bV_g$ is an algebraic representation of $\Sp_{2g}(\Z)$.
\end{theorem:algebraicity}
\begin{proof}
Whether an $\Sp_{2g}(\Z)$-representation is algebraic is unchanged under field extensions, so we can assume
that $\bk$ is algebraically closed.  We now appeal to a theorem of Lubotzky \cite{LubotzkyProAlgebraic}
that says the following:\footnote{This uses the fact that $g \geq 2$.  See \cite{PutmanRepSLZ} for an expository account of the related case of $\SL_n(\Z)$.}
\begin{itemize}
\item the finite-dimensional representations of $\Sp_{2g}(\Z)$ over $\bk$ are semisimple, i.e., they decompose as
direct sums of irreducible representations; and
\item the finite-dimensional irreducible representations of $\Sp_{2g}(\Z)$ over $\bk$ are precisely those
of the form $\bU \otimes \bW$, where $\bU$ is an irreducible algebraic representation of $\Sp_{2g}(\Z)$ and
$\bW$ is an irreducible representation of\footnote{For $g \geq 2$, the congruence subgroup
property \cite{MennickeCSPSp} says that all finite quotients of $\Sp_{2g}(\Z)$ factor through $\Sp_{2g}(\Z/\ell)$ for some $\ell \geq 2$.}  $\Sp_{2g}(\Z/\ell)$ for some $\ell \geq 2$.
Here $\Sp_{2g}(\Z)$ acts on $\bW$ via the surjection $\Sp_{2g}(\Z) \twoheadrightarrow \Sp_{2g}(\Z/\ell)$.
\end{itemize}
Applying this to $\bV_g$, we can decompose it as a direct sum of irreducible representations.  By projecting
everything onto an irreducible subrepresentation of $\bV_g$, we can assume that $\bV_g$ is an irreducible
representation, and thus is of the form $\bU \otimes \bW$ with $\bU$ an irreducible algebraic representation of $\Sp_{2g}(\Z)$
and $\bW$ an irreducible representation of $\Sp_{2g}(\Z/\ell)$ for some $\ell \geq 2$.
Our goal is to prove that $\bW$ is a trivial representation.

We will prove below that the restriction
of $\bW$ to $\Sp_{2(g-2)}(\Z/\ell) \times \Sp_{4}(\Z/\ell)$ is unmixed.  Assuming this,
since by Proposition~\ref{proposition:spuniversal} the subgroup
$\Sp_{2(g-2)}(\Z/\ell) \times \Sp_4(\Z/\ell)$ is universally mixed in $\Sp_{2g}(\Z/\ell)$ we deduce
that  the restriction of $\bW$ to $\Sp_{2(g-2)}(\Z/\ell) \times \Sp_{4}(\Z/\ell)$ trivial.
Since the normal closure of $\Sp_{2(g-2)}(\Z/\ell) \times \Sp_4(\Z/\ell)$ in
$\Sp_{2g}(\Z/\ell)$ is $\Sp_{2g}(\Z/\ell)$, this implies that $\bW$ is a trivial
representation of $\Sp_{2g}(\Z/\ell)$, as desired.

It remains to prove that the restriction of $\bW$ to $\Sp_{2(g-2)}(\Z/\ell) \times \Sp_{4}(\Z/\ell)$ is unmixed.
Let
\[\Res^{\Sp_{2g}(\Z/\ell)}_{\Sp_{2(g-2)}(\Z/\ell) \times \Sp_4(\Z/\ell)} \bW = \bigoplus_{i=1}^n \bW_i\]
be a decomposition into irreducible representations of $\Sp_{2(g-2)}(\Z/\ell) \times \Sp_4(\Z/\ell)$ and let
\[\Res^{\Sp_{2g}(\Z)}_{\Sp_{2(g-2)}(\Z) \times \Sp_4(\Z)} \bU = \bigoplus_{j=1}^m \bU_j\]
be a decomposition into irreducible algebraic representations of $\Sp_{2(g-2)}(\Z) \times \Sp_4(\Z)$.  We thus have
\[\Res^{\Sp_{2g}(\Z)}_{\Sp_{2(g-2)}(\Z) \times \Sp_4(\Z)} \bV_g = \bigoplus_{i=1}^n \bigoplus_{j=1}^m \bU_j \otimes \bW_i.\]
Each $\bU_j \otimes \bW_i$ is an irreducible representation\footnote{This follows from 
the Jacobson density theorem (cf.\ the proof of the first Claim of \cite[Theorem C]{PutmanRepSLZ}).} 
of $\Sp_{2(g-2)}(\Z) \times \Sp_4(\Z/\ell)$.  

For each $1 \leq i \leq n$ and $1 \leq j \leq m$, assumption (b) implies that one of the following two things happens:
\begin{itemize}
\item $\bU_j \otimes \bW_i$ is in the image of $f\colon \bV_{g-2} \boxtimes \bk \rightarrow \bV_g$, so
$1 \times \Sp_4(\Z)$ acts trivially on $\bU_j \otimes \bW_i$ and hence on $\bW_i$; or
\item $\bU_j \otimes \bW_i$ survives in $\coker(f)$, so its restriction to $\Sp_{2(g-2)}(\Z)$ is an algebraic
representation of $\Sp_{2(g-2)}(\Z)$.  This implies that $\Sp_{2(g-2)}(\Z) \times 1$ acts trivially on $\bW_i$.
\end{itemize}
Let $\bW'$ be the direct sum of the $\bW_i$ such that $1 \times \Sp_4(\Z)$ acts trivially on $\bW_i$ and let
$\bW''$ be the direct sum of the $\bW_i$ such that $\Sp_{2(g-2)}(\Z) \times 1$ acts trivially on $\bW_i$.  The
restriction of $\bW$ to $\Sp_{2(g-2)}(\Z/\ell) \times \Sp_{4}(\Z/\ell)$ decomposes as
$\bW' \oplus \bW''$, showing that it is unmixed.
\end{proof}

\part{Homology of Torelli, step 1: reduction to curve stabilizers}
\label{part:step1}

Recall from \S \ref{section:introduction} that the Torelli group $\Torelli_{g,p}^b$ on a genus $g$ surface $\Sigma_{g,p}^b$ with
$p$ marked points and $b$ boundary components is the kernel of the action of the mapping class
group $\Mod_{g,p}^b$ on $\HH_1(\Sigma_g)$.
Each $\HH_d(\Torelli_{g,p}^b;\Q)$ is a representation
of $\Sp_{2g}(\Z)$.  Our goal is to prove Theorem~\ref{maintheorem:h2torelli}, which says
that $\HH_2(\Torelli_{g,p}^b;\Q)$ is finite dimensional for $g \geq 5$
and an algebraic representation of $\Sp_{2g}(\Z)$ for $g \geq 6$.  In this
part of the paper, we reduce this to a theorem about curve stabilizers.  This
reduction occurs in \S \ref{section:step1main}, which is preceded by the preliminary
\S \ref{section:step1prelim}.

\section{Step 1.1: preliminary results about the homology of Torelli}
\label{section:step1prelim}

This section contains some preliminary results about the homology of the Torelli group.

\subsection{Deleting boundary components and marked points}
\label{section:deleteboundary}

If $f\colon \bV \rightarrow \bW$ is an
equivariant map between $\Sp_{2g}(\Z)$-representations $\bV$ and $\bW$ over $\Q$, then say that
$f$ is an {\em isomorphism mod fin dim alg reps} if both $\ker(f)$ and
$\coker(f)$ are finite-dimensional algebraic representations of $\Sp_{2g}(\Z)$.  If this holds, then
$\bV$ is a finite-dimensional algebraic representation of $\Sp_{2g}(\Z)$ if and only if $\bW$ is.

The following two lemmas imply that for a fixed $g \geq 3$,
to prove that $\HH_2(\Torelli_{g,p}^b;\Q)$ is
either finite-dimensional or an algebraic representation of $\Sp_{2g}(\Z)$
for all $p,b \geq 0$ it is enough to prove this for any single choice
of $p$ or $b$.  This will reduce us to only considering
$\HH_2(\Torelli_g^1;\Q)$.

\begin{lemma}[Cap boundary]
\label{lemma:capboundary}
Let $b,p \geq 0$ and $g \geq 3$.  Let $\partial$ be a component of
$\partial \Sigma_{g,p}^{b+1}$ and let $\Torelli_{g,p}^{b+1} \rightarrow \Torelli_{g,p+1}^b$
be the map that glues a disc containing a marked point to $\partial$ and extends mapping
classes in $\Torelli_{g,p}^{b+1}$ over it by the identity.
Then the induced map $\HH_2(\Torelli_{g,p}^{b+1};\Q) \rightarrow \HH_2(\Torelli_{g,p+1}^b;\Q)$
is an isomorphism mod fin dim alg representations.
\end{lemma}

The proofs of this and many other results will use the following fact:
\begin{align*}
\tag{$\spadesuit$}\label{eqn:algclosed} &\text{the collection of finite-dimensional algebraic representations of $\Sp_{2g}(\Z)$ is a Serre}\\ &\text{class, i.e., it is closed under subquotients and extensions.}
\end{align*}

\begin{proof}
By \cite[Proposition 3.19]{FarbMargalitPrimer}, there is a central extension
\begin{equation}
\label{eqn:cap}
\begin{tikzcd}
1 \arrow{r} & \Z \arrow{r} & \Mod_{g,p}^{b+1} \arrow{r}{f} & \Mod_{g,p+1}^b \arrow{r} & 1,
\end{tikzcd}
\end{equation}
where the central $\Z$ is generated by the Dehn twist $T_{\partial}$ and $f$
glues a disc containing a marked point to $\partial$ and extends mapping classes
over it by the identity.  Since the action of $\Mod_{g,p}^{b+1}$ on $\HH_1(\Sigma_g)$ factors through
$\Mod_{g,p+1}^b$, an element $\phi \in \Mod_{g,p}^{b+1}$ acts trivially on $\HH_1(\Sigma_g)$ if and only
if $f(\phi)$ does.  It follows that \eqref{eqn:cap} restricts to a similar central extension
\[\begin{tikzcd}
1 \arrow{r} & \Z \arrow{r} & \Torelli_{g,p}^{b+1} \arrow{r} & \Torelli_{g,p+1}^b \arrow{r} & 1
\end{tikzcd}\]
of Torelli groups.  The Hochschild--Serre spectral sequence of this extension induces a long exact Gysin sequence that contains the segment
\[\begin{tikzcd}
\HH_1(\Torelli_{g,p+1}^b;\Q) \arrow{r} & \HH_2(\Torelli_{g,p}^{b+1};\Q) \arrow{r} & \HH_2(\Torelli_{g,p+1}^b;\Q) \arrow{r} & \HH_0(\Torelli_{g,p+1}^b;\Q)
\end{tikzcd}\]
It follows from Johnson's work \cite{JohnsonAbel} that $\HH_1(\Torelli_{g,p+1}^b;\Q)$ is a finite-dimensional
algebraic representation of $\Sp_{2g}(\Z)$ for $g \geq 3$.  Also, $\HH_0(\Torelli_{g,p+1}^n;\Q) = \Q$ is an
algebraic representation of $\Sp_{2g}(\Z)$.  The lemma now follows from \eqref{eqn:algclosed}.
\end{proof}

\begin{lemma}[Delete marked point]
\label{lemma:deletepuncture}
Let $b,p \geq 0$ and $g \geq 3$.  Let $p_0$ be a marked point of $\Sigma_{g,p+1}^b$ and let
$\Torelli_{g,p+1}^b \rightarrow \Torelli_{g,p}^b$ be the map that deletes $p_0$.  Then
the induced map $\HH_2(\Torelli_{g,p+1}^b;\Q) \rightarrow \HH_2(\Torelli_{g,p}^b;\Q)$
is an isomorphism mod fin dim alg reps.
\end{lemma}
\begin{proof}
There is a Birman exact sequence \cite[Theorem 4.6]{FarbMargalitPrimer}
\begin{equation}
\label{eqn:birman}
\begin{tikzcd}
1 \arrow{r} & \pi_1(\Sigma_{g,p}^b,p_0) \arrow{r} & \Mod_{g,p+1}^b \arrow{r}{f} & \Mod_{g,p}^b \arrow{r} & 1,
\end{tikzcd}
\end{equation}
where\footnote{Here by $\pi_1(\Sigma_{g,p}^b,p_0)$ we mean the fundamental group of a genus $g$ surface with $p$ punctures (not marked points)
and $b$ boundary components.  Throughout this proof, we will continue to let context indicate whether $p$
means ``punctures'' or ``marked points''.} $\pi_1(\Sigma_{g,p}^b,p_0)$ is the point-pushing subgroup of
$\Mod_{g,p+1}^b$ and $f$ deletes $p_0$.
Since the action of $\Mod_{g,p+1}^b$ on $\HH_1(\Sigma_g)$ factors through $\Mod_{g,p}^b$, an element
$\phi \in \Mod_{g,p+1}^b$ acts trivially on $\HH_1(\Sigma_g)$ if and only if $f(\phi)$ does.  It follows that
\eqref{eqn:birman} restricts to an exact sequence
\[\begin{tikzcd}
1 \arrow{r} & \pi_1(\Sigma_{g,p}^b,p_0) \arrow{r} & \Torelli_{g,p+1}^b \arrow{r}{f} & \Torelli_{g,p}^b \arrow{r} & 1.
\end{tikzcd}\]
The associated Hochschild--Serre spectral sequence takes the form
\begin{equation}
\label{eqn:birmanss}
\ssE^2_{pq} = \HH_p(\Torelli_{g,p}^b;\HH_q(\pi_1(\Sigma_{g,p}^b);\Q)) \Rightarrow \HH_{p+q}(\Torelli_{g,p+1}^b;\Q).
\end{equation}
The terms of this spectral sequence are representations of $\Sp_{2g}(\Z)$, and the differentials are
$\Sp_{2g}(\Z)$-equivariant.
We will use this spectral sequence to prove that the kernel and cokernel of $\HH_2(\Torelli_{g,p+1}^b;\Q) \rightarrow \HH_2(\Torelli_{g,p}^b;\Q)$
are finite-dimensional algebraic representations of $\Sp_{2g}(\Z)$.

We start with the cokernel.  We have $\ssE^2_{20} = \HH_2(\Torelli_{g,p}^b;\Q)$, and 
the image of the map $\HH_2(\Torelli_{g,p+1}^b;\Q) \rightarrow \HH_2(\Torelli_{g,p}^b;\Q)$ is
\[\ssE^{\infty}_{20} = \ssE^3_{20} = \ker(\ssE^2_{20} \rightarrow \ssE^2_{01}) = \ker(\HH_2(\Torelli_{g,p}^b;\Q) \rightarrow \ssE^2_{01}).\]
This implies that cokernel of $\HH_2(\Torelli_{g,p+1}^b;\Q) \rightarrow \HH_2(\Torelli_{g,p}^b;\Q)$ embeds
into $\ssE^2_{01}$.  Using \eqref{eqn:algclosed}, to prove
that the cokernel of $\HH_2(\Torelli_{g,p+1}^b;\Q) \rightarrow \HH_2(\Torelli_{g,p}^b;\Q)$ is
a finite dimensional algebraic representation of $\Sp_{2g}(\Z)$ it is enough to prove that
\[\ssE^2_{01} = \HH_0(\Torelli_{g,p}^b;\HH_1(\Sigma_{g,p}^b;\Q)) = \HH_1(\Sigma_{g,p}^b;\Q)_{\Torelli_{g,p}^b}\]
is a finite-dimensional algebraic representation of $\Sp_{2g}(\Z)$, where the subscript indicates that we are
taking coinvariants.  
Though $\Torelli_{g,p}^b$ acts trivially on $\HH_1(\Sigma_g;\Q)$, it might
not act trivially on $\HH_1(\Sigma_{g,p}^b;\Q)$.  However, letting
$m = \max(p+q-1,0)$ we do have an extension
\begin{equation}
\label{eqn:punctureextension}
\begin{tikzcd}
0 \arrow{r} & \Q^{m} \arrow{r} & \HH_1(\Sigma_{g,p}^b;\Q) \arrow{r} & \HH_1(\Sigma_g;\Q) \arrow{r} & 0
\end{tikzcd}
\end{equation}
with $\Torelli_{g,p}^b$ acting trivially on the kernel and cokernel.  This induces a right-exact sequence
\[\begin{tikzcd}
\Q^{m} \arrow{r} & \HH_1(\Sigma_{g,p}^b;\Q)_{\Torelli_{g,p}^b} \arrow{r} & \HH_1(\Sigma_g;\Q) \arrow{r} & 0.
\end{tikzcd}\]
Since $\Q^{m}$ is a trivial representation of $\Sp_{2g}(\Z)$ and $\HH_1(\Sigma_g;\Q)$ is a finite-dimensional algebraic representation of $\Sp_{2g}(\Z)$,
using \eqref{eqn:algclosed}
it follows that $\ssE^2_{01}=\HH_1(\Sigma_{g,p}^b;\Q)_{\Torelli_{g,p}^b}$ is a finite-dimensional
algebraic representation of $\Sp_{2g}(\Z)$, as desired.

We now handle the kernel.  In terms of the spectral sequence \eqref{eqn:birmanss}, the kernel
of the map $\HH_2(\Torelli_{g,p+1}^b;\Q) \rightarrow \HH_2(\Torelli_{g,p}^b;\Q)$ has a filtration
whose associated graded terms are $\ssE^{\infty}_{02}$ and $\ssE^{\infty}_{11}$.  These are subquotients
of $\ssE^2_{02}$ and $\ssE^2_{11}$.
Using \eqref{eqn:algclosed}, it
is enough to prove that $\ssE^2_{02}$ and $\ssE^2_{11}$ are finite-dimensional algebraic representations
of $\Sp_{2g}(\Z)$.

The term $\ssE^2_{02}$ is a finite-dimensional algebraic representation of $\Sp_{2g}(\Z)$ since
\[\ssE^2_{02} = \HH_0(\Torelli_{g,p}^b;\HH_2(\Sigma_{g,p}^b;\Q)) = \HH_2(\Sigma_{g,p}^b;\Q)_{\Torelli_{g,p}^b} = \begin{cases} \Q & \text{if $p=b=0$}, \\ 0 & \text{otherwise}. \end{cases}\]
For $\ssE^2_{11} = \HH_1(\Torelli_{g,p}^b;\HH_1(\Sigma_{g,p}^b;\Q))$, note that
the long exact sequence in homology associated to the extension \eqref{eqn:punctureextension} of $\Torelli_{g,p}^b$-representations
contains the segment
\[\begin{tikzcd}
\HH_1(\Torelli_{g,p}^b;\Q^{m}) \arrow{r} & \HH_1(\Torelli_{g,p}^b;\HH_1(\Sigma_{g,p}^b;\Q)) \arrow{r} & \HH_1(\Torelli_{g,p}^b;\HH_1(\Sigma_{g};\Q)).
\end{tikzcd}\]
This can be rewritten as
\[\begin{tikzcd}
\HH_1(\Torelli_{g,p}^b;\Q)^{m} \arrow{r} & \ssE^2_{11} \arrow{r} & \HH_1(\Torelli_{g,p}^b;\Q) \otimes \HH_1(\Sigma_g;\Q).
\end{tikzcd}\]
It follows from Johnson's work \cite{JohnsonAbel} that $\HH_1(\Torelli_{g,p+1}^b;\Q)$ is a finite-dimensional
algebraic representation of $\Sp_{2g}(\Z)$ for $g \geq 3$.  Using \eqref{eqn:algclosed}
it follows that $\ssE^2_{11}$ is a finite-dimensional
algebraic representation of $\Sp_{2g}(\Z)$, as desired.
\end{proof}

\subsection{Curve stabilizers}
\label{section:curvestabilizers}

Embed $\Sigma_{g-1}^1$ in $\Sigma_g$ and let $\gamma$ be an oriented nonseparating simple
closed curve on $\Sigma_g$ that is disjoint from $\Sigma_{g-1}^1$:\\
\Figure{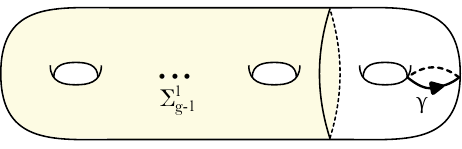}
Let $(\Torelli_g)_{\gamma}$ be the $\Torelli_g$-stabilizer of the isotopy class of $\gamma$.  The image of the
map $\Torelli_{g-1}^1 \hookrightarrow \Torelli_g$ induced by our embedding lies
in $(\Torelli_g)_{\gamma}$.  Then:

\begin{lemma}
\label{lemma:curvestabilizer}
Let $g \geq 5$.  Embed $\Sigma_{g-1}^1$ in $\Sigma_g$ and let $\gamma$ be a nonseparating simple
closed curve on $\Sigma_g$ that is disjoint from $\Sigma_{g-1}^1$.  Let
$h\colon \HH_2(\Torelli_{g-1}^1;\Q) \rightarrow \HH_2((\Torelli_g)_{\gamma};\Q)$ be the map
induced by the inclusion $\Torelli_{g-1}^1 \hookrightarrow (\Torelli_g)_{\gamma}$.  Regard
$h$ as a map of $\Sp_{2(g-1)}(\Z)$-representations.  Then $h$ is an isomorphism mod fin dim alg reps.
\end{lemma}

The proof uses a result that we will also use several other times.
For a subsurface $S$ of $\Sigma_g$, let $\Torelli_g(S)$ be the subgroup
of $\Torelli_g$ consisting of mapping classes supported on $S$.  For instance, if
$\gamma_1,\ldots,\gamma_k$ are simple closed curves on $\Sigma_g$ that intersect
transversely, then
$(\Torelli_g)_{\gamma_1,\ldots,\gamma_k} = \Torelli_g(S)$ 
for $S$ the complement of an open regular neighborhood of
$\gamma_1 \cup \cdots \cup \gamma_k$.  Then:

\begin{theorem}[{Putman, \cite[Theorem B]{PutmanJohnson}}]
\label{theorem:putmaninjective}
Let $S \subset \Sigma_g$ be a connected subsurface of genus at least $3$.
Then the map $\HH_1(\Torelli_g(S);\Q) \rightarrow \HH_1(\Torelli_g;\Q)$
is injective.
\end{theorem}

\begin{proof}[{Proof of Lemma~\ref{lemma:curvestabilizer}}]
Let $\partial_1$ and $\partial_2$ be the components of $\partial \Sigma_{g-1}^2$.  Embed $\Sigma_{g-1}^1$ in $\Sigma_{g-1}^2$:\\
\Figure{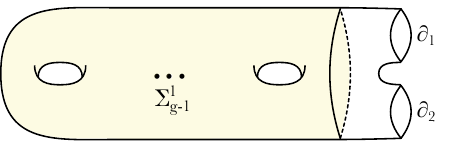}
Let $\Sigma_{g-1}^2 \rightarrow \Sigma_g$
be the map that glues $\partial_1$ to $\partial_2$ and maps those components to $\gamma$
and $\Sigma_{g-1}^1$ to $\Sigma_{g-1}^1$.  Let $\phi\colon \Mod_{g-1}^2 \rightarrow \Mod_g$ be the induced
map on mapping class groups, so $\phi(T_{\partial_1}) = \phi(T_{\partial_2}) = T_{\gamma}$.
The image of $\phi$ is $(\Mod_g)_{\gamma}$, and $\phi$ takes $\Mod_{g-1}^1 < \Mod_{g-1}^2$
isomorphically onto $\Mod_{g-1}^1 < \Mod_g$.  Define
$\tTorelli_{g-1}^2 = \phi^{-1}(\Torelli_g)$.
Be warned that this is a proper subgroup of $\Torelli_{g-1}^2$; see \cite{PutmanCutPaste}.  The
map $h$ factors as
\[\begin{tikzcd}
\HH_2(\Torelli_{g-1}^1;\Q) \arrow{r}{f} & \HH_2(\tTorelli_{g-1}^2;\Q) \arrow{r}{\tphi} & \HH_2((\Torelli_g)_{\gamma};\Q),
\end{tikzcd}\]
where $f$ is induced by the inclusion $\Torelli_{g-1}^1 \hookrightarrow \tTorelli_{g-1}^2$ and $\tphi$ is induced by
the restriction of $\phi\colon \Mod_{g-1}^2 \rightarrow (\Mod_g)_{\gamma}$ to $\tTorelli_{g-1}^2$.  To prove that
$h$ is an isomorphism mod fin dim alg reps, it is enough to prove this for $f$ and $\tphi$:

\begin{step}{1}
The map $\tphi\colon \HH_2(\tTorelli_{g-1}^2;\Q) \longrightarrow \HH_2((\Torelli_g)_{\gamma};\Q)$ is
an isomorphism mod fin dim alg reps.
\end{step}

Since $\phi(T_{\partial_1}) = \phi(T_{\partial_2}) = T_{\gamma}$, we have $T_{\partial_1} T_{\partial_2}^{-1} \in \ker(\phi)$.
In fact, we have a central extension
\[\begin{tikzcd}
1 \arrow{r} & \Z \arrow{r} & \Mod_{g-1}^2 \arrow{r}{\phi} & (\Mod_g)_{\gamma} \arrow{r} & 1
\end{tikzcd}\]
whose central $\Z$ is generated by $T_{\partial_1} T_{\partial_2}^{-1}$.  We have
$T_{\partial_1} T_{\partial_2}^{-1} \in \tTorelli_{g-1}^2$, so this restricts to a central extension
\[\begin{tikzcd}
1 \arrow{r} & \Z \arrow{r} & \tTorelli_{g-1}^2 \arrow{r} & (\Torelli_g)_{\gamma} \arrow{r} & 1.
\end{tikzcd}\]
The Hochschild--Serre spectral sequence of this extension induces a long exact Gysin sequence that contains the segment
\[\begin{tikzcd}
\HH_1((\Torelli_g)_{\gamma};\Q) \arrow{r} & \HH_2(\tTorelli_{g-1}^2;\Q) \arrow{r}{\tphi} & \HH_2((\Torelli_g)_{\gamma};\Q) \arrow{r} & \HH_0((\Torelli_g)_{\gamma};\Q).
\end{tikzcd}\]
Since $\HH_0((\Torelli_g)_{\gamma};\Q) = \Q$ is an
algebraic representation of $\Sp_{2(g-1)}(\Z)$, by \eqref{eqn:algclosed}
it is enough to prove that $\HH_1((\Torelli_g)_{\gamma};\Q)$ is a finite-dimensional
algebraic representation of $\Sp_{2(g-1)}(\Z)$.

Theorem~\ref{theorem:putmaninjective} implies that the map
$\HH_1((\Torelli_g)_{\gamma};\Q) \rightarrow \HH_1(\Torelli_g;\Q)$
is injective.  Johnson \cite{JohnsonAbel} proved that $\HH_1(\Torelli_g;\Q)$
is a finite-dimensional algebraic representation of $\Sp_{2g}(\Z)$.  Its restriction
to $\Sp_{2(g-1)}(\Z)$ is also algebraic, so by \eqref{eqn:algclosed} its
subrepresentation $\HH_1((\Torelli_g)_{\gamma};\Q)$ is a finite-dimensional
algebraic representation of $\Sp_{2(g-1)}(\Z)$, as desired.

\begin{step}{2}
\label{step:curvestabilizer2}
The map $f\colon \HH_2(\Torelli_{g-1}^1;\Q) \rightarrow \HH_2(\tTorelli_{g-1}^2;\Q)$ is an isomorphism
mod fin dim alg reps.
\end{step}

Putman \cite{PutmanCutPaste} constructed a version of the Birman exact sequence for\footnote{In the notation
of \cite{PutmanCutPaste}, the group $\tTorelli_{g-1}^2$ is $\Torelli(\Sigma_{g-1}^2,\{\{\partial_1,\partial_2\}\})$.} $\tTorelli_{g-1}^2$.  
Letting $\pi = \pi_1(\Sigma_{g-1}^1)$, it takes the form
\begin{equation}
\label{eqn:birmantorelli}
\begin{tikzcd}
1 \arrow{r} & {[\pi,\pi]} \arrow{r} & \tTorelli_{g-1}^2 \arrow{r} & \Torelli_{g-1}^1 \arrow{r} & 1.
\end{tikzcd}
\end{equation}
Here $\tTorelli_{g-1}^2 \rightarrow \Torelli_{g-1}^1$ is the map induced by gluing a disc to $\partial_1$ and
extending mapping classes by the identity, and $[\pi,\pi]$ is an appropriate subgroup of the ``disc pushing group''.  The
extension \eqref{eqn:birmantorelli} splits via the map $\Torelli_{g-1}^1 \rightarrow \tTorelli_{g-1}^2$ induced by the embedding $\Sigma_{g-1}^1 \hookrightarrow \Sigma_{g-1}^2$
discussed above.

Since $[\pi,\pi]$ is a free group, the Hochschild--Serre spectral sequence of \eqref{eqn:birmantorelli} has two rows, and
since this extension is split all the differentials coming out of the bottom row of this spectral sequence vanish.  We
deduce that this spectral sequence degenerates to give an extension
\[\begin{tikzcd}
0 \arrow{r} & \HH_1(\Torelli_{g-1}^1;\HH_1([\pi,\pi];\Q)) \arrow{r} & \HH_2(\tTorelli_{g-1}^2;\Q) \arrow{r} & \HH_2(\Torelli_{g-1}^1;\Q) \arrow{r} & 0.
\end{tikzcd}\]
The $f\colon \HH_2(\Torelli_{g-1}^1;\Q) \rightarrow \HH_2(\tTorelli_{g-1}^2;\Q)$ we are trying to prove is an isomorphism
mod fin dim alg reps is the splitting of this exact sequence coming from the splitting of \eqref{eqn:birmantorelli}.
We conclude that $\ker(f) = 1$ and that $\coker(f) \cong \HH_1(\Torelli_{g-1}^1;\HH_1([\pi,\pi];\Q))$.
For $g \geq 5$, the authors proved in \cite{MinahanPutmanAbelian}\footnote{Earlier 
Putman \cite{PutmanAbelian}
proved an analogous result for the level-$\ell$ subgroup of $\Mod_{g-1}^1$.  This played an important role in
Putman's computation of the second homology of the level-$\ell$ subgroup in \cite{PutmanH2Level, PutmanPicard}.}  that $\HH_1(\Torelli_{g-1}^1;\HH_1([\pi,\pi];\Q))$ is
a finite-dimensional algebraic representation of $\Sp_{2(g-1)}(\Z)$.  The step follows.
\end{proof}

\section{Step 1.1: reduction to curve stabilizers}
\label{section:step1main}

Let $\alpha$ and $\beta$ be the following curves on $\Sigma_{g}$:\\
\Figure{AandBcurves}
Let 
\[\lambda\colon \HH_2((\Torelli_g)_{\alpha};\Q) \oplus \HH_2((\Torelli_g)_{\beta};\Q) \longrightarrow \HH_2(\Torelli_g;\Q)\]
be the sum of the maps induced by the inclusions $(\Torelli_g)_{\alpha} \hookrightarrow \Torelli_g$
and $(\Torelli_g)_{\beta} \hookrightarrow \Torelli_g$ and let $\Lambda_g$ be the cokernel of $\lambda$.  The group
$\Mod_g$ does not act on $\Lambda_g$ since it does not fix $\alpha$ and $\beta$, but the subgroup $\Mod_{g-1}^1$ of $\Mod_g$ does
act on $\Lambda_g$.  This factors through $\Sp_{2(g-1)}(\Z)$, making $\Lambda_g$ into a representation of $\Sp_{2(g-1)}(\Z)$.
The rest of this paper is devoted to the proof of:

\setcounter{maintheoremprime}{1}
\begin{maintheoremprime}
\label{maintheorem:cokernel}
The representation $\Lambda_g$ is a finite-dimensional algebraic representation of $\Sp_{2(g-1)}(\Z)$
for $g \geq 5$.
\end{maintheoremprime}

Here we will assume Theorem~\ref{maintheorem:cokernel} and
use it to prove Theorem~\ref{maintheorem:h2torelli}.

\newtheorem*{maintheorem:h2torelli}{Theorem~\ref{maintheorem:h2torelli}}
\begin{maintheorem:h2torelli}
Let $b,p \geq 0$.  Then $\HH_2(\Torelli_{g,p}^b;\Q)$ is finite dimensional for $g \geq 5$
and an algebraic representation of $\Sp_{2g}(\Z)$ for $g \geq 6$.
\end{maintheorem:h2torelli}
\begin{proof}[Proof of Theorem~\ref{maintheorem:h2torelli}, assuming Theorem~\ref{maintheorem:cokernel}]
By Lemmas \ref{lemma:capboundary} and \ref{lemma:deletepuncture}, it is enough to prove Theorem
\ref{maintheorem:h2torelli} for $\HH_2(\Torelli_g^1;\Q)$.
Embed $\Sigma_{g-1}^1$ into $\Sigma_{g}^1$ as in the following figure:\\
\Figure{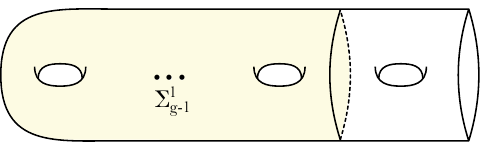}
Extending mapping classes in $\Torelli_{g-1}^1$ to $\Sigma_{g}^1$ by the identity,
we get an induced map $\Torelli_{g-1}^1 \rightarrow \Torelli_{g}^1$.  Passing
to $\HH_2$ gives a coherent sequence\footnote{See \S \ref{section:coherent} for the definition of a coherent sequence.}
\[\begin{tikzcd}
\HH_2(\Torelli_1^1;\Q) \arrow{r}{f_1} & \HH_2(\Torelli_2^1;\Q) \arrow{r}{f_2} & \HH_2(\Torelli_3^1;\Q) \arrow{r}{f_3} & \cdots
\end{tikzcd}\]
of representations of $\Sp_{2g}(\Z)$.

We would like to apply our stability theorem (Theorem~\ref{maintheorem:stability})
to this coherent sequence with $g_0=5$.  If we can do this, that theorem
will imply that $\HH_2(\Torelli_g^1;\Q)$ is finite dimensional for $g \geq g_0$
and an algebraic representation of $\Sp_{2g}(\Z)$ for $g \geq g_0+1$, just like we want.
Theorem~\ref{maintheorem:stability} has two hypotheses (i) and (ii) we must verify:

\begin{step}{1}
Hypothesis (i) holds: $\coker(f_{g-1})$ is a finite-dimensional algebraic representation
of $\Sp_{2(g-1)}(\Z)$ for $g \geq g_0=5$.
\end{step}

Let $f'_{g-1}\colon \HH_2(\Torelli_{g-1}^1;\Q) \rightarrow \HH_2(\Torelli_g;\Q)$ be the composition
\[\begin{tikzcd}
\HH_2(\Torelli_{g-1}^1;\Q) \arrow{r}{f_{g-1}} & \HH_2(\Torelli_g^1;\Q) \arrow{r}{\pi} & \HH_2(\Torelli_g;\Q),
\end{tikzcd}\]
where $\pi$ is the map induced by gluing a disc to $\partial \Sigma_g^1$ 
and extending mapping classes by the identity.  We can factor $\pi$ as
\[\begin{tikzcd}
\HH_2(\Torelli_g^1;\Q) \arrow{r} & \HH_2(\Torelli_{g,1};\Q) \arrow{r} & \HH_2(\Torelli_g;\Q).
\end{tikzcd}\]
Lemmas \ref{lemma:capboundary} and \ref{lemma:deletepuncture} say that these two maps are isomorphisms
mod fin dim alg reps for $g \geq 3$, so using \eqref{eqn:algclosed} the map $\pi$ is as well.  Again using
\eqref{eqn:algclosed}, this implies that to prove that
$\coker(f_{g-1})$ is a finite-dimensional algebraic representation of $\Sp_{2(g-1)}(\Z)$, it is
enough to prove the same result for $\coker(f'_{g-1})$.

Let $\alpha$ and $\beta$ be the curves on $\Sigma_g$ from Theorem~\ref{maintheorem:cokernel}.  Let
\[h_{\alpha}\colon \HH_2(\Torelli_{g-1}^1;\Q) \longrightarrow \HH_2((\Torelli_g)_{\alpha};\Q) \quad \text{and} \quad h_{\beta}\colon \HH_2(\Torelli_{g-1}^1;\Q) \longrightarrow \HH_2((\Torelli_g)_{\beta};\Q)\]
be the maps induced by the inclusions $\Torelli_{g-1}^1 \hookrightarrow (\Torelli_g)_{\alpha}$
and $\Torelli_{g-1}^1 \hookrightarrow (\Torelli_g)_{\beta}$ and let
\[\lambda_{\alpha}\colon \HH_2((\Torelli_g)_{\alpha};\Q) \longrightarrow \HH_2(\Torelli_g;\Q) \quad \text{and} \quad \lambda_{\beta}\colon \HH_2((\Torelli_g)_{\beta};\Q) \longrightarrow \HH_2(\Torelli_g;\Q)\]
be the maps induced by the inclusions $(\Torelli_g)_{\alpha} \hookrightarrow \Torelli_g$ and
$(\Torelli_g)_{\beta} \hookrightarrow \Torelli_g$.  Consider the sum of two copies of $f'_{g-1}$:
\[f'_{g-1} + f'_{g-1} \colon \HH_2(\Torelli_{g-1}^1;\Q) \oplus \HH_2(\Torelli_{g-1}^1;\Q) \longrightarrow \HH_2(\Torelli_g;\Q).\]
We can factor $f'_{g-1}+f'_{g-1}$ as
\begin{equation}
\label{eqn:factorf}
\begin{tikzcd}
\HH_2(\Torelli_{g-1}^1;\Q) \oplus \HH_2(\Torelli_{g-1}^1;\Q) \arrow{r}{h_{\alpha} \oplus h_{\beta}} & \HH_2((\Torelli_g)_{\alpha};\Q) \oplus \HH_2((\Torelli_g)_{\beta};\Q) \arrow{r}{\lambda_{\alpha} + \lambda_{\beta}} & \HH_2(\Torelli_g;\Q).
\end{tikzcd}
\end{equation}
Lemma~\ref{lemma:curvestabilizer} says that $h_{\alpha}$ and $h_{\beta}$ are both isomorphisms mod fin dim alg reps for $g \geq 5$, so
by \eqref{eqn:algclosed} the map
$h_{\alpha} \oplus h_{\beta}$ is as well.  Theorem~\ref{maintheorem:cokernel} 
says that $\Lambda_g = \coker(\lambda_{\alpha} + \lambda_{\beta})$ is a finite-dimensional algebraic
representation of $\Sp_{2(g-1)}(Z)$ for $g \geq 5$.  Using the factorization \eqref{eqn:factorf} and \eqref{eqn:algclosed}, 
we deduce that $\coker(f'_{g-1}+f'_{g-1})$ is also
a finite-dimensional algebraic representation of $\Sp_{2(g-1)}(\Z)$ for $g \geq 5$.  Since
$f'_{g-1}$ and $f'_{g-1}+f'_{g-1}$ have the same image, the same is true for $\coker(f'_{g-1})$,
as desired.

\begin{step}{2}
Hypothesis (ii) holds: the coinvariants $\HH_2(\Torelli_g^1;\Q)_{\Sp_{2g}(\Z)}$ are finite-dimensional
for $g \geq g_0=5$.  In fact, this holds for $g \geq 3$.
\end{step}

One quick way to see this is to appeal to \cite{KassabovPutman}, which says that $\HH_2(\Torelli_g^1;\Q)$ is
finitely generated as a module over the group ring $\Q[\Sp_{2g}(\Z)]$.  Another more elementary
approach is to use the Hochschild--Serre spectral sequence of the extension
\[\begin{tikzcd}
1 \arrow{r} & \Torelli_g^1 \arrow{r} & \Mod_g^1 \arrow{r} & \Sp_{2g}(\Z) \arrow{r} & 1,
\end{tikzcd}\]
which takes the form
\[\ssE^2_{pq} = \HH_p(\Sp_{2g}(\Z);\HH_q(\Torelli_g^1;\Q)) \Rightarrow \HH_{p+q}(\Mod_g^1;\Q).\]
We want to show that
\[\ssE^2_{02} = \HH_0(\Sp_{2g}(\Z);\HH_2(\Torelli_g^1;\Q)) = \HH_2(\Torelli_g^1;\Q)_{\Sp_{2g}(\Z)}\]
is finite-dimensional.
Since the homology groups of $\Mod_g^1$ are all finite-dimensional, we know that
$\ssE^{\infty}_{02}$ is finite dimensional.  To go from $\ssE^2_{02}$ to $\ssE^{\infty}_{02}$,
we must kill the images of differentials coming from
\begin{align*}
\ssE^2_{21} &= \HH_2(\Sp_{2g}(\Z);\HH_1(\Torelli_g^1;\Q)) \quad \text{and} \\
\ssE^3_{30} &\subset \ssE^2_{30} = \HH_3(\Sp_{2g}(\Z);\Q).
\end{align*}
In light of Johnson's theorem \cite{JohnsonAbel} saying that $\HH_1(\Torelli_g^1;\Q)$ is a finite-dimensional
algebraic representation of $\Sp_{2g}(\Z)$ and Borel--Serre's theorem \cite{BorelSerreCorners} saying
that $\Sp_{2g}(\Z)$ has a finite-index subgroup with a compact classifying space, these are both
finite-dimensional.  It follows that $\ssE^2_{02}$ is finite-dimensional, as desired.
\end{proof}

\part{Homology of Torelli, step 2: generators for cokernel}
\label{part:step2}

It remains to prove Theorem~\ref{maintheorem:cokernel}, which
says that $\Lambda_g$ is a finite-dimensional algebraic representation of $\Sp_{2(g-1)}(\Z)$ for $g \geq 5$.
In this part of the paper, we find generators for $\Lambda_g$.  
In \S \ref{section:equivarianthomology} we introduce equivariant homology,
in \S \ref{section:handlecomplex} we introduce the handle complex $\cC_{ab}(\Sigma_g)$,
and then in \S \ref{section:generators} we find our generators.

\section{Step 2.1: Preliminaries on equivariant homology}
\label{section:equivarianthomology}

We start with some preliminaries on (Borel) equivariant homology.
See \cite[\S VII.7]{BrownCohomology} for a textbook reference.  Fix
a group $G$ and a commutative ring $\bk$.

\subsection{Basic definitions}
\label{section:basicdef}

A {\em $G$-CW complex} is a CW complex equipped with a cellular action of $G$.  Fix a contractible $G$-CW complex $EG$
on which $G$ acts freely.  For a $G$-CW complex $X$, the group $G$ acts freely on $EG \times X$.  Define
\[EG \times_{G} X = (EG \times X)/G.\]
This is known as the {\em Borel construction}.  The {\em $G$-equivariant homology groups} of $X$, denoted $\HH^G_{\bullet}(X;\bk)$, are $\HH_{\bullet}(EG \times_G X;\bk)$.
They do not depend on the choice of $EG$ and are functorial under $G$-equivariant cellular maps.  They
are also functorial under group homomorphisms in the following sense: if
$f\colon G_1 \rightarrow G_2$ is a group homomorphism and $\phi\colon X_1 \rightarrow X_2$ is a cellular map from a
$G_1$-CW complex $X_1$ to a $G_2$-CW complex $X_2$ such that
\[\phi(g \Cdot x) = f(g) \Cdot \phi(x) \quad \text{for all $g \in G_1$ and $x \in X_1$},\]
then there is an induced map
$(f,\phi)_{\ast}\colon \HH^{G_1}_{\bullet}(X_1;\bk) \rightarrow \HH^{G_2}_{\bullet}(X_2;\bk)$.

\subsection{Relation to group homology}
If $X$ is a contractible $G$-CW complex, then $EG \times X$ is also contractible.  Since $G$ acts
freely on $EG \times X$, we deduce that $EG \times_G X$ is a $K(G,1)$, so by definition\footnote{We use $=$ rather than
$\cong$ to indicate that this isomorphism is canonical.}
$\HH^G_{\bullet}(X;\bk) = \HH_{\bullet}(G;\bk)$.
A basic example is $X = \pt$ equipped with the trivial $G$-action, so
\[\HH^G_{\bullet}(\pt;\bk) = \HH_{\bullet}(G;\bk).\]
For an arbitrary $G$-CW complex $X$, the map $X \rightarrow \pt$ induces a canonical map
\begin{equation}
\label{eqn:maptopt}
\HH^G_{\bullet}(X;\bk) \longrightarrow \HH^G_{\bullet}(\pt;\bk) = \HH_{\bullet}(G;\bk).
\end{equation}
We then have the following:

\begin{lemma}
\label{lemma:connected}
Let $G$ be a group and let $X$ be an $n$-connected $G$-CW complex.  For all commutative rings $\bk$, the map
$\HH^G_{d}(X;\bk) \rightarrow \HH_{d}(G;\bk)$ from \eqref{eqn:maptopt} is an isomorphism for
$d \leq n$ and a surjection for $d = n+1$.
\end{lemma}
\begin{proof}
We can build a contractible $G$-CW complex $Y$ from $X$ by equivariantly attaching cells of dimension at least $n+2$.  The
CW complex $EG \times_G Y$ is then built from $EG \times_G X$ by attaching cells of dimension at least $n+2$.
Letting $\iota\colon X \rightarrow Y$ be the inclusion, the map we are studying factors as
\[\begin{tikzcd}
\HH^G_d(X;\bk) \arrow{r}{\iota_{\ast}} & \HH^G_d(Y;\bk) \arrow{r} & \HH_d(G;\bk).
\end{tikzcd}\]
Since $Y$ is contractible the canonical map $\HH^G_d(Y;\bk) \rightarrow \HH_d(G;\bk)$ 
is an isomorphism for all $d$, and by construction $\iota_{\ast}$ is an isomorphism for
$d \leq n$ and a surjection for $d = n+1$.  The lemma follows.
\end{proof}

\subsection{Spectral sequence}
\label{section:spectralsequence}

Let $X$ be a $G$-CW complex.  Assume that $G$ acts on $X$ {\em without rotations}, i.e., for all cells $\sigma$ of $X$
the stabilizer $G_{\sigma}$ fixes $c$ pointwise.  This ensures
 that $X/G$ is a CW complex whose $p$-cells are in bijection with the $G$-orbits of $p$-cells of $X$.  

For a $p$-cell $\sigma \in (X/G)^{(p)}$, let $\tsigma$ be a lift of $\sigma$ to $X$.  Consider
$\HH_q(G_{\tsigma};\bk)$.  This appears to depend on the choice of $\tsigma$.  However, if $\tsigma'$ is another lift
of $\sigma$ to $X$ then there exists some $g \in G$ with $g \Cdot \tsigma = \tsigma'$ and hence $g G_{\tsigma} g^{-1} = G_{\tsigma'}$.
Conjugation by $g$ thus gives an isomorphism
\begin{equation}
\label{eqn:isostab}
\HH_q(G_{\tsigma};\bk) \cong \HH_q(G_{\tsigma'};\bk).
\end{equation}
This isomorphism does not depend on the choice of $g$; indeed, any other choice of $g \in G$ with $g \Cdot \tsigma = \tsigma'$ is of the form $g h$ with 
$h \in G_{\tsigma}$, and conjugation by $h$ induces the trivial automorphism of $\HH_q(G_{\tsigma};\bk)$.
Since the isomorphism \eqref{eqn:isostab} is canonical, we can 
unambiguously write $\HH_q(G_{\tsigma};\bk)$ for $\sigma \in (X/G)^{(p)}$.  With this convention, we have:

\begin{proposition}[{\cite[VII.7.7]{BrownCohomology}}]
\label{proposition:spectralsequence}
Let $G$ be a group, let $X$ be a $G$-CW complex, and let $\bk$ be a commutative ring.  Assume that $G$ acts
on $X$ without rotations.  There is then a functorial first quadrant spectral sequence
\[\ssE^1_{pq} = \bigoplus_{\sigma \in (X/G)^{(p)}} \HH_q(G_{\tsigma};\bk) \Rightarrow \HH^G_{p+q}(X;\bk).\]
\end{proposition}

In the rest of this section, we will fix a $G$-CW complex $X$ on which $G$ acts without rotations and
discuss properties of the spectral sequence $\ssE$ given by Proposition~\ref{proposition:spectralsequence}.

\subsection{Left column}
\label{section:leftcol}
Since $0$-cells are vertices, we will denote them with $v$ instead of $\sigma$.  For a fixed $q$, consider the composition
\begin{equation}
\label{eqn:leftcol}
\begin{tikzcd}
\bigoplus\limits_{v \in (X/G)^{(0)}} \HH_q(G_{\tv};\bk) = \ssE^1_{0q} \arrow[two heads]{r} & \ssE^{\infty}_{0q} \arrow[hook]{r} & \HH^G_q(X;\bk) \arrow{r} & \HH_q(G;\bk)
\end{tikzcd}
\end{equation}
whose maps are as follows:
\begin{itemize}
\item the surjection $\ssE^1_{0q} \twoheadrightarrow \ssE^{\infty}_{0q}$ comes from the fact that there are no nonzero differentials coming out of $\ssE^1_{0q}$; and
\item the inclusion $\ssE^{\infty}_{0q} \hookrightarrow \HH^G_q(X;\bk)$ comes from the fact that $\ssE^{\infty}_{0q}$ is the first term in
the filtration of $\HH^G_q(X;\bk)$ coming from our spectral sequence; and
\item the map $\HH^G_q(X;\bk) \rightarrow \HH_q(G;\bk)$ is the canonical map.
\end{itemize}
We claim that \eqref{eqn:leftcol} equals the sum of the maps induced by the inclusions $G_{\tv} \hookrightarrow G$ of vertex stabilizers.
This can be proved easily from the construction of the spectral sequence in \cite[VII.7.7]{BrownCohomology}, but we
prefer the following less computational proof.

Consider the $G$-equivariant map $X \rightarrow \pt$.  Letting
$\ssF$ be the spectral sequence obtained by applying Proposition~\ref{proposition:spectralsequence} to $\pt$, we get a map
$\ssE \rightarrow \ssF$ of spectral sequences converging to the canonical map
\[\HH^G_{\bullet}(X;\bk) \longrightarrow \HH^G_{\bullet}(\pt;\bk) = \HH_{\bullet}(G;\bk).\]
The spectral sequence $\ssF$ degenerates at $\ssF^1$, which is of the form
\[\ssF^1_{pq} = \begin{cases}
\HH_q(G;\bk) & \text{if $p=0$},\\
0            & \text{if $p \neq 0$}.
\end{cases}\]
Identifying $\ssE^1_{0q}$ with
\[\bigoplus_{v \in (X/G)^{(0)}} \HH_q(G_{\tv};\bk),\]
the map $\ssE^1_{0q} \rightarrow \ssF^1_{0q} = \HH_q(G;\bk)$ is exactly the sum of the maps
induced by the inclusions $G_{\tv} \hookrightarrow G$ of vertex stabilizers.  The claim follows.

\subsection{Bottom row}
\label{section:bottomrow}

We will next need a description of the differentials of our spectral sequence in two special case.  The
first is when $q=0$.  Observe that
\[\ssE^1_{p0} = \bigoplus_{\sigma \in (X/G)^{(p)}} \HH_0(G_{\tsigma};\bk) = \bigoplus_{\sigma \in (X/G)^{(p)}} \bk = \CC_p(X/G;\bk),\]
where $\CC_p(X/G;\bk)$ is the $p^{\text{th}}$ term of the cellular chain complex for $X/G$.  The $\ssE^1$-differentials
$\ssE^1_{p0} \rightarrow \ssE^1_{p-1,0}$ thus fit into a chain complex of the form
\[\begin{tikzcd}
\CC_0(X/G;\bk) & \CC_1(X/G;\bk) \arrow{l} & \CC_2(X/G;\bk) \arrow{l} & \cdots \arrow{l}.
\end{tikzcd}\]
This is exactly the cellular chain complex of $X/G$; see \cite[\S VII.8]{BrownCohomology}.  It follows that
$\ssE^2_{p0} = \HH_p(X/G;\bk)$.

\subsection{1-cell differentials}
\label{section:differentials}

We also need a description of the differentials when $p=1$.  By our description of the $\ssE^1$-page, these differentials
$\partial\colon \ssE^1_{1q} \rightarrow \ssE^1_{0q}$ are of the form
\[\begin{tikzcd}
\bigoplus_{e \in (X/G)^{(1)}} \HH_q(G_{\te};\bk) \arrow{r}{\partial} & \bigoplus_{v \in (X/G)^{(0)}} \HH_q(G_{\tv};\bk).
\end{tikzcd}\]
Consider some $e \in (X/G)^{(1)}$.  For our differential, we must fix some (arbitrary) orientation on $e$.  
Let $\te \in X^{(1)}$ be our lift to $X$, which has an orientation coming from the orientation on $e$.  
Let $w_0 \in X^{(0)}$ and $w_1 \in X^{(0)}$ be the initial and terminal vertices of $\te$, respectively.  We
have inclusions $G_{\te} \hookrightarrow G_{w_0}$ and $G_{\te} \hookrightarrow G_{w_1}$.  On
the summand $\HH_q(G_{\te};\bk)$ of $\ssE^1_{1q}$, the differential $\partial$ is then the difference between the
two induced maps
\begin{equation}
\label{eqn:1celldiff1}
\begin{tikzcd}
\HH_q(G_{\te};\bk) \arrow{r} & \HH_q(G_{w_1};\bk) \arrow[hook]{r} & \bigoplus_{v \in (X/G)^{(0)}} \HH_q(G_{\tv};\bk)
\end{tikzcd}
\end{equation}
and
\begin{equation}
\label{eqn:1celldiff2}
\begin{tikzcd}
\HH_q(G_{\te};\bk) \arrow{r} & \HH_q(G_{w_0};\bk) \arrow[hook]{r} & \bigoplus_{v \in (X/G)^{(0)}} \HH_q(G_{\tv};\bk).
\end{tikzcd}
\end{equation}
See \cite[\S VII.8]{BrownCohomology} for a proof.  

If $e$ is a loop, then $w_0$ and $w_1$
are in the same $G$-orbit, so there exists some $s \in G$ with $s(w_0) = w_1$.  The
terms $\HH_q(G_{w_1};\bk)$ and $\HH_q(G_{w_0};\bk)$ in \eqref{eqn:1celldiff1} and
\eqref{eqn:1celldiff2} go to the same term in the indicated direct sum, and
$\HH_q(G_{w_1};\Q)$ is identified with
$\HH_q(G_{w_0};\bk)$ via conjugation by $s$.  

Consider the special case $q=1$.  For $g \in G_{w_0}$, let
$\overline{g} \in \HH_1(G_{w_0};\bk)$ be its homology class.  
Our differential is the composition 
\[\begin{tikzcd}
\HH_1(G_{\te};\bk) \arrow{r} & \HH_1(G_{w_0};\bk) \arrow[hook]{r} & \bigoplus_{v \in (X/G)^{(0)}} \HH_1(G_{\tv};\bk),
\end{tikzcd}\]
where the first map take the homology class of $g \in G_{\te}$ to
\[\overline{s^{-1} g s} - \overline{g} = \overline{g^{-1} s^{-1} g s} = \overline{[g,s]} \in \HH_1(G_{w_0};\bk).\]

\section{Step 2.2: the handle complex}
\label{section:handlecomplex}

We now introduce a space on which $\Torelli_g$ acts.

\subsection{Complex of homologous curves}

Let $v \in \HH_1(\Sigma_g)$ be a primitive element, i.e., one that is only divisible by $\pm 1$.  
Let $\cC_v(\Sigma_g)$ be the following simplicial complex:
\begin{itemize}
\item {\bf vertices}: isotopy classes
of oriented simple closed curves $\gamma$ on $\Sigma_g$ with $[\gamma] = v$.
\item {\bf $\mathbf{p}$-simplices}: sets $\sigma = \{\gamma_0,\ldots,\gamma_p\}$ of 
distinct vertices such that the $\gamma_i$ can
be isotoped to be pairwise disjoint.
\end{itemize}
For instance, the following is a $2$-simplex of $\cC_v(\Sigma_g)$ for $v = [\gamma_0]$:\footnote{To avoid cluttering our figures, they will often not indicate the orientations on the curves.}\\
\Figure{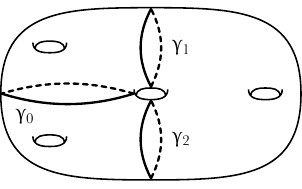}
Putman \cite{PutmanTrick} introduced $\cC_v(\Sigma_g)$ and proved it was connected for $g \geq 3$.
Hatcher--Margalit \cite{HatcherMargalitGenerators} gave an alternate proof of this that Minahan \cite{MinahanComplex}
generalized to show:\footnote{It is not known if $\pi_1(\cC_v(\Sigma_g))=1$ for $g \geq 4$, which would let us conclude it is $(g-3)$-connected.}

\begin{theorem}[{Minahan, \cite{MinahanComplex}}]
\label{theorem:cvacyclic}
Let $g \geq 2$ and let $v \in \HH_1(\Sigma_g)$ be a primitive element.  Then
$\cC_v(\Sigma_g)$ is $(g-3)$-acyclic, i.e., $\RH_k(\cC_v(\Sigma_g)) = 0$ for $k \leq g-3$.
\end{theorem}

\subsection{Handle complex}
Let $a,b \in \HH_1(\Sigma_g)$ be primitive elements with algebraic intersection number $1$.  An {\em $a$-curve}
(resp.\ a {\em $b$-curve}) is an oriented simple closed curve $\gamma$ with $[\gamma] = a$ (resp.\ $[\gamma] = b$).
The {\em handle complex}, denoted $\cC_{ab}(\Sigma_g)$, is the following simplicial complex:
\begin{itemize}
\item {\bf vertices}: isotopy classes of $a$-curves and $b$-curves.
\item {\bf $\mathbf{p}$-simplices}: sets $\sigma = \{\gamma_0,\ldots,\gamma_p\}$ of distinct vertices such that either:
\begin{itemize}
\item $\sigma$ is a $p$-simplex of $\cC_a(\Sigma_g)$ or $\cC_b(\Sigma_g)$; or
\item for some $\gamma_{i_0} \in \sigma$, the set $\sigma \setminus \{\gamma_{i_0}\}$ is a $(p-1)$-simplex of $\cC_a(\Sigma_g)$ and
$\gamma_{i_0}$ is a $b$-curve that can be isotoped to intersect each curve in $\sigma \setminus \{\gamma_{i_0}\}$ once; or
\item for some $\gamma_{i_0} \in \sigma$, the set $\sigma \setminus \{\gamma_{i_0}\}$ is a $(p-1)$-simplex of $\cC_b(\Sigma_g)$ and
$\gamma_{i_0}$ is an $a$-curve that can be isotoped to intersect each curve in $\sigma \setminus \{\gamma_{i_0}\}$ once.
\end{itemize}
\end{itemize}
We will call the simplices of $\cC_a(\Sigma_g)$ and $\cC_b(\Sigma_g)$ the {\em pure simplices} of $\cC_{ab}(\Sigma_g)$ and
the other simplices the {\em mixed simplices}.  If in the following
figure the orange curves are $a$-curves
and the blue curves are $b$-curves,\footnote{We will use this coloring convention in the rest
of the paper.} then the indicated curves form mixed simplices of $\cC_{ab}(\Sigma_g)$:\\
\Figure{MixedCells}
The $1$-skeleton of $\cC_{ab}(\Sigma_g)$ is the {\em handle graph} defined by Putman \cite{PutmanSmallGenset}, who
proved it is connected for $g \geq 3$.  Proposition~\ref{proposition:mixedacyclic} below says that $\cC_{ab}(\Sigma_g)$ is $1$-acyclic\footnote{Presumably it could
be made more highly acyclic by allowing mixed simplices that contain more $a$- and $b$-curves.  We
defined $\cC_{ab}(\Sigma_g)$ like we did to ensure that $\cC_{ab}(\Sigma_g) / \Torelli_g$ is contractible; see
Proposition~\ref{proposition:quotientcontractible}.} for $g \geq 4$.

\subsection{Rotations}

The group $\Torelli_g$ acts on $\cC_{ab}(\Sigma_g)$.  This action is without rotations:

\begin{lemma}
\label{lemma:withoutrotations}
Let $g \geq 1$ and let $a,b \in \HH_1(\Sigma_g)$ be primitive elements with 
algebraic intersection number $1$.  Then $\Torelli_g$ acts on $\cC_{ab}(\Sigma_g)$
without rotations.
\end{lemma}
\begin{proof}
This is immediate from the fact that every simplex of $\cC_{ab}(\Sigma_g)$ is the join
of a simplex of $\cC_a(\Sigma_g)$ and a simplex of $\cC_{b}(\Sigma_g)$ along
with the fact that $\Torelli_g$ acts on $\cC_a(\Sigma_g)$ and $\cC_b(\Sigma_g)$
without rotations (see \cite[Theorem 1.2]{IvanovBook} for a more general result).
\end{proof}

\subsection{Mixing pure simplices}

The following shows that every pure simplex of $\cC_{ab}(\Sigma_g)$ can be extended to a mixed simplex, and this
extension is unique up to the action of $\Torelli_g$:

\begin{lemma}
\label{lemma:mixpure}
Let $g \geq 1$ and let $a,b \in \HH_1(\Sigma_g)$ be primitive elements with 
algebraic intersection number $1$.  Let $\sigma$ be a pure simplex of
$\cC_{ab}(\Sigma_g)$.  Then there is a vertex $\delta$ of $\cC_{ab}(\Sigma_g)$
such that $\sigma \cup \{\delta\}$ is a mixed simplex.  Moreover, if
$\delta'$ is another vertex of $\cC_{ab}(\Sigma_g)$ such that $\sigma \cup \{\delta'\}$
is a mixed simplex, then there exists $f \in \Torelli_g$ with
$f(\sigma) = \sigma$ and $f(\delta) = \delta'$.
\end{lemma}
\begin{proof}
We will give the proof for $\sigma$ a simplex of $\cC_{a}(\Sigma_g)$.  The case
where it is a simplex of $\cC_b(\Sigma)$ is identical.  Let
$\sigma = \{\gamma_0,\ldots,\gamma_p\}$, so the $\gamma_i$ are disjoint $a$-curves.
They divide $\Sigma_g$ into $(p+1)$ 
components $T_0,\ldots,T_{p}$ with $T_i \cong \Sigma_{g_i}^2$
for some $g_i \geq 1$ with $1+\sum g_i = g$:\\
\Figure{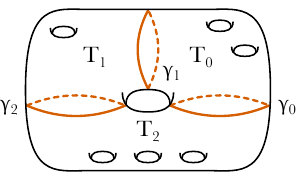}
As in this figure, we order the $\gamma_i$ and $T_i$ such that 
$\partial T_i = \gamma_{i} \sqcup \gamma_{i+1}$,
where the indices are taken modulo $p$.  We now divide the proof into two steps.

\begin{step}{1}
There exists a $b$-curve $\delta$ that intersects each $\gamma_i$ once, so
$\sigma \cup \{\delta\}$ is a mixed simplex.
\end{step}

Let $\zeta$ be an arbitrary oriented simple closed curve that intersects each $\gamma_i$ once
such that the intersection number of $\gamma_i$ with $\zeta$ is $+1$.
Let $S_0,\ldots,S_{p}$ be the components of the complement
of a regular neighborhood of $\zeta \cup \gamma_0 \cup \cdots \cup \gamma_p$, ordered
such that $S_i \subset T_i$:\\
\Figure{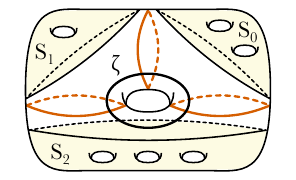}
We then have
\[\HH_1(\Sigma_g) = \Span{[\gamma_0],[\zeta]} \oplus \bigoplus_{i=0}^{p} \HH_1(S_i).\]
Write
\[b = c [\gamma_0] + d [\zeta] + \sum_{i=0}^{p} x_i \quad \text{with $c,d \in \Z$ and $x_i \in \HH_1(S_i)$}.\]
Since the algebraic intersection numbers of $b$ and $[\zeta]$ with $a = [\gamma_0]$ are $1$,
we have $d = 1$.  Replacing $\zeta$ with $T^{-c}_{\gamma_0}(\zeta)$, 
we can also assume that $c = 0$, so $b = [\zeta] + \sum_{i=0}^p x_i$.
For $0 \leq i \leq p$, let $q_i$ be the intersection point of $\zeta$ with $\gamma_i$
and let $\zeta_i$ be the subarc of $\zeta$ lying in $T_i$, so $\zeta_i$ goes
from $q_i$ to $q_{i+1}$.  By
\cite[Lemma 3.2]{PutmanSmallGenset}, there exist a properly embedded arc
$\delta_i$ in $T_i$ going from $q_i$ to $q_{i+1}$ such that in the relative
homology group $\HH_1(T_i,\{q_i,q_{i+1}\})$, we have $[\delta_i] = [\zeta_i]+x_i$.
We can then take $\delta$ to be the loop made up of the $\delta_i$.

\begin{step}{2}
Let $\delta$ and $\delta'$ be $b$-curves such that $\sigma \cup \{\delta\}$
and $\sigma \cup \{\delta'\}$ are mixed simplices.  Then there exists
$f \in \Torelli_g$ with
$f(\sigma) = \sigma$ and $f(\delta) = \delta'$.
\end{step}

Let $S_0,\ldots,S_{p}$ 
be the components of the complement
of a regular neighborhood of $\delta \cup \gamma_0 \cup \cdots \cup \gamma_p$,
ordered such that $S_i \subset T_i$.  
The span of $\HH_1(S_i)$ and $a = [\gamma_i]$ equals $\HH_1(T_i)$, so
$\HH_1(S_i)$ is the intersection of $\HH_1(T_i)$ with the orthogonal complement
of $b = [\delta]$.

Similarly, let $S'_1,\ldots,S'_p$
be the components of the complement
of a regular neighborhood of $\delta' \cup \gamma_0 \cup \cdots \cup \gamma_p$,
ordered such that $S'_i \subset T_i$.  Just like above, $\HH_1(S'_i)$ is
the intersection of $\HH_1(T_i)$ with the orthogonal complement of $b = [\delta']$.
In other words, as subgroups of $\HH_1(\Sigma_g)$ we have $\HH_1(S_i) = \HH_1(S'_i)$
for $0 \leq i \leq p$.  Let $V_i = \HH_1(S_i) = \HH_1(S'_i)$.

By the change of coordinates principle from \cite{FarbMargalitPrimer}, we
can find $\phi \in \Mod_g$ with $\phi(\delta) = \delta'$ and $\phi(\gamma_i) = \gamma_i$
for $0 \leq i \leq p$.  By construction $\phi$ fixes $a = [\gamma_0]$ and $b = [\delta]$, and
it takes $V_i = \HH_1(S_i)$ to $V_i = \HH_1(S'_i)$ for $0 \leq i \leq p$.  Since mapping
classes on the $1$-holed surface $S_i$ can realize any symplectic automorphism of $V_i = \HH_1(S_i)$,
we can find some $\psi_i \in \Mod_g$ supported on $S_i$ such that $\psi_i$ induces
$\phi_{\ast}|_{V_i} \colon V_i \rightarrow V_i$ on $V_i = \HH_1(S_i)$.  Define
\[f = \phi \psi_0^{-1} \cdots \psi_p^{-1} \in \Mod_g.\]
We have $f(\sigma) = \sigma$ and $f(\delta) = \delta'$.  By 
construction $f$ fixes $a = [\gamma_0]$ and $b = [\delta]$ as well as each $V_i$.
Since these span $\HH_1(\Sigma_g)$, we conclude that $f$ acts trivially on $\HH_1(\Sigma_g)$,
i.e., $f \in \Torelli_g$.
\end{proof}

\subsection{Description of action}
\label{section:quotient}

We now prove several results about the action of $\Torelli_g$ on $\cC_{ab}(\Sigma)$.

\begin{lemma}
\label{lemma:transitive}
Let $g \geq 1$ and let $a,b \in \HH_1(\Sigma_g)$ be primitive elements with algebraic intersection number $1$.
Then $\Torelli_g$ acts transitively on:
\begin{itemize}
\item[(i)] vertices of $\cC_{ab}(\Sigma_g)$ that are $a$-curves; and
\item[(ii)] vertices of $\cC_{ab}(\Sigma_g)$ that are $b$-curves; and
\item[(iii)] mixed $1$-simplices of $\cC_{ab}(\Sigma_g)$.
\end{itemize}
\end{lemma}
\begin{proof}
Johnson \cite[Lemma 5]{JohnsonConjugacy} proved that for any oriented nonseparating simple closed curves
$\gamma$ and $\gamma'$ on $\Sigma_g$ with $[\gamma] = [\gamma']$,
there exists $f \in \Torelli_g$ with $f(\gamma) = \gamma'$.  This implies that $\Torelli_g$ acts
transitively on $a$-curves and $b$-curves, as in (i) and (ii).

To prove (iii), for $i=1,2$ let $e_i$ be a mixed $1$-simplex joining an $a$-curve $\alpha_i$ to a $b$-curve $\beta_i$.
By the previous paragraph, there exists $f \in \Torelli_g$ with $f(\alpha_1) = \alpha_2$.  Both
$f(\beta_1)$ and $\beta_2$ are $b$-curves intersecting $f(\alpha_1) = \alpha_2$ once, so by 
Lemma~\ref{lemma:mixpure} there exists $f' \in \Torelli_g$ with 
$f'(f(\alpha_1)) = f(\alpha_1) = \alpha_2$
and
$f'(f(\beta_1)) = \beta_2$.
It follows that $f' f$ takes $e_1$ to $e_2$, as desired.
\end{proof}

This immediately implies:

\begin{corollary}
\label{corollary:quotient}
Let $g \geq 1$ and let $a,b \in \HH_1(\Sigma_g)$ be primitive elements with algebraic intersection number $1$.
Then the $1$-skeleton of $\cC_{ab}(\Sigma_g)/\Torelli_g$ consists of:
\begin{itemize}
\item two vertices $v_a$ and $v_b$, with $v_a$ (resp.\ $v_b$) the image of any $a$-curve (resp.\ $b$-curve).
\item loops based at $v_a$ and $v_b$, each the image of a pure $1$-simplex. 
\item a single $1$-simplex $e_{ab}$ joining $v_a$ and $v_b$, with $e_{ab}$ the image of any mixed $1$-simplex.
\end{itemize}
\end{corollary}

The following puts edges of $\cC_{ab}(\Sigma_g)$ into a normal form up
to the action of $\Torelli_g$:

\begin{lemma}
\label{lemma:normalform}
Let $g \geq 1$, let $a,b \in \HH_1(\Sigma_g)$ be primitive elements with algebraic intersection number $1$, and
let $\{\alpha,\beta\}$ be a mixed $1$-simplex of $\cC_{ab}(\Sigma_g)$ with $[\alpha]=a$ and
$[\beta]=b$.  Then:
\begin{itemize}
\item[(i)] for all oriented $1$-simplices $e$ of $\cC_a(\Sigma_g)$, there exists some
$f,B \in \Torelli_g$ with $B(\beta) = \beta$ such that
$f(e)$ goes from $\alpha$ to $B(\alpha)$; and
\item[(ii)] for all oriented $1$-simplices $e$ of $\cC_b(\Sigma_g)$, there exists some
$f,A \in \Torelli_g$ with $A(\alpha) = \alpha$ such that
$f(e)$ goes from $\beta$ to $A(\beta)$.
\end{itemize}
\end{lemma}
\Figure{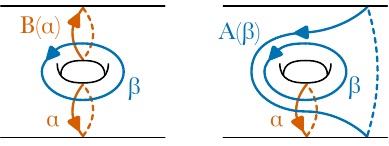}
\begin{proof}
The proofs of (i) and (ii) are similar, so we prove (i) and leave (ii) to the reader.  Write $e = \{\alpha_1,\alpha_2\}$.  Our
goal is to find $f,B \in \Torelli_g$ such that $B(\beta) = \beta$ and
$f(e) = \{f(\alpha_1),f(\alpha_2)\} = \{\alpha,B(\alpha)\}$.

By Lemma~\ref{lemma:transitive}, there is an $f_1 \in \Torelli_g$ with
$f_1(\alpha_1) = \alpha$, so $f_1(e) = \{\alpha,f_1(\alpha_2)\}$.  By Lemma~\ref{lemma:mixpure},
there is a $b$-curve $\beta'$ with $\{\alpha,f_1(\alpha_2),\beta'\}$ a mixed simplex.
Since $\{\alpha,\beta\}$ and $\{\alpha,\beta'\}$ are mixed $1$-simplices, 
by Lemma~\ref{lemma:transitive} there is an $f_2 \in \Torelli_g$ with $f_2(\alpha) = \alpha$ and
$f_2(\beta') = \beta$.  Let $f = f_2 f_1$, so $f(e) = \{\alpha,f(\alpha_2)\}$
and $\{\alpha,f(\alpha_2),\beta\}$ is a mixed simplex.
Since $\{\alpha,\beta\}$ and $\{f(\alpha_2),\beta\}$ are mixed $1$-simplices, by Lemma~\ref{lemma:mixpure}
there is a $B \in \Torelli_g$ with $B(\alpha) = f(\alpha_2)$ and $B(\beta) = \beta$, so $f(e) = \{\alpha,B(\alpha)\}$,
as desired.
\end{proof}

\subsection{Contractability of quotient}

Below we will prove that $\cC_{ab}(\Sigma_g)$ is $1$-acyclic.  First, however, we note that our results
quickly imply that $\cC_{ab}(\Sigma_g) / \Torelli_g$ is contractible:

\begin{proposition}
\label{proposition:quotientcontractible}
Let $g \geq 1$ and let $a,b \in \HH_1(\Sigma_g)$ be primitive elements with algebraic intersection number $1$.
Then $\cC_{ab}(\Sigma_g) / \Torelli_g$ is contractible.
\end{proposition}
\begin{proof}
Let $X = \cC_{ab}(\Sigma_g) / \Torelli_g$.  Let $A$ and $B$ be the images in $X$ of
$\cC_a(\Sigma_g)$ and $\cC_b(\Sigma_g)$, respectively.  Let
$v_a$ and $v_b$ be the vertices of $\cC_{ab}(\Sigma_g)/\Torelli_g$ discussed in Corollary \ref{corollary:quotient},
so $v_a \in A$ and $v_b \in B$.  Let $e_{ab}$ be the $1$-simplex of $X$ connecting
$v_a$ to $v_b$ from Corollary \ref{corollary:quotient}.
Lemma~\ref{lemma:mixpure} implies that for every cell $\sigma$ of $A$, there is a unique
cell of $X$ obtained by coning $\sigma$ off with $v_b$.  It follows that $X$ contains the cone $\cone_{v_b}(A)$ of
$A$ with cone point $v_b$.  Similarly, $X$ also contains the cone $\cone_{v_a}(B)$ of $B$ with cone point $v_a$.  Thus 
\[X = \cone_{v_a}(B) \cup \cone_{v_b}(A) \quad \text{and} \quad \cone_{v_a}(B) \cap \cone_{v_b}(A) = e_{ab}.\]
Since $\cone_{v_a}(B)$ and $\cone_{v_b}(A)$ and $e_{ab}$ are contractible, $X$ is also contractible.
\end{proof}

\subsection{Strips}

To understand the topology of $\cC_{ab}(\Sigma_g)$ itself, we need some preliminaries.  For mixed $1$-simplices $e_0$ and $e_1$ of $\cC_{ab}(\Sigma_g)$,
an {\em $ab$-strip} connecting $e_0$ to $e_1$ is a triangulation $S$ of $[0,1]^2$ equipped with a
simplicial map $f\colon S \rightarrow \cC_{ab}(\Sigma_g)$ such that:
\begin{itemize}
\item for $i=0,1$, the subspace $i \times [0,1]$ of $S$ is a $1$-simplex mapping to $e_i$; and
\item $f$ maps each vertex in $[0,1] \times 0$ to an $a$-curve and each vertex in $[0,1] \times 1$ to a $b$-curve.
\end{itemize}
\Figure{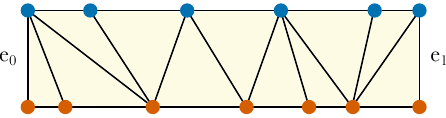}

Every two mixed $1$-simplices can be connected by an $ab$-strip:

\begin{lemma}
\label{lemma:strip}
Let $g \geq 3$ and let $a,b \in \HH_1(\Sigma_g)$ be primitive elements with algebraic intersection number $1$.
Let $e_0$ and $e_1$ be mixed $1$-simplices of $\cC_{ab}(\Sigma_g)$.  Then there exists an $ab$-strip
$f\colon S \rightarrow \cC_{ab}(\Sigma_g)$ connecting $e_0$ to $e_1$.
\end{lemma}
\begin{proof}~\hspace{-5pt}\footnote{This proof is similar to the proof in \cite{PutmanTrick} that
$\cC_v(\Sigma_g)$ is connected for $g \geq 3$.}
For mixed $1$-simplices $e'_0$ and $e'_1$ of $\cC_{ab}(\Sigma_g)$, write $e'_0 \sim e'_1$ if there exists an $ab$-strip
$f\colon S \rightarrow \cC_{ab}(\Sigma_g)$ connecting $e'_0$ and $e'_1$.  This is
an equivalence relation, and our goal is to prove that $e_0 \sim e_1$.
Using the change of coordinates principle from \cite{FarbMargalitPrimer}, we can assume that
$e_0$ is the edge connecting the curves $\alpha$ and $\beta$ in the following:\\
\Figure{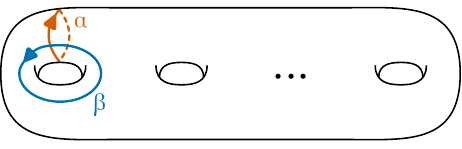}
Lemma~\ref{lemma:transitive} says that the group $\Torelli_g$ acts transitively on the set of mixed $1$-simplices of $\cC_{ab}(\Sigma_g)$,
so we can find $f \in \Torelli_g$ with $f(e_0) = e_1$.  

Let $\Delta \subset \Torelli_g$ be the finite generating set for $\Torelli_g$ constructed by Johnson \cite{JohnsonFinite}.
Using the notation from \cite{JohnsonFinite}, this generating set is defined using 
a set of curves\footnote{Be warned that $c_{\beta}$ has nothing to do with our curve $\beta$.} $\{c_1,\ldots,c_{2g},c_{\beta}\}$, and matching up our figures these satisfy
$c_1 = \alpha$ and $c_2 = \beta$.  All the other curves in $\{c_1,\ldots,c_{2g},c_{\beta}\}$ are
disjoint from $\alpha$ and $\beta$.  We will say more about $\Delta$ during the proof of the claim below.

Below we will prove that for $s \in \Delta^{\pm 1}$ we have $e_0 \sim s(e_0)$.  This implies the lemma via
the trick from the second author's paper \cite{PutmanTrick}.  Here are more details:
writing $f = s_1 \cdots s_k$ with $s_i \in \Delta^{\pm 1}$, the fact that $e_0 \sim s_i(e_0)$ for all
$1 \leq i \leq k$ implies that
\[e_0 \sim s_1(e_0) \sim s_1 s_2(e_0) \sim s_1 s_2 s_3(e_0) \sim \cdots \sim s_1 \cdots s_k(e_0) = e_1.\]
Here we are using the fact that $\sim$ is invariant under $\Torelli_g$, and thus for instance
if $e_0 \sim s_2(e_0)$ then $s_1(e_0) \sim s_1 s_2(e_0)$.
It remains to prove:

\begin{unnumberedclaim}
For $s \in \Delta^{\pm 1}$, we have $e_0 \sim s(e_0)$.
\end{unnumberedclaim}

The generators for $\Torelli_g$ from \cite{JohnsonFinite} are all bounding pair maps, i.e., products
$T_x T_y^{-1}$ with $x$ and $y$ disjoint simple closed curves on $\Sigma_g$ such that $x \cup y$
separates $\Sigma_g$.  Since the inverse of a bounding pair map is a bounding pair map, we can therefore
write $s = T_x T_y^{-1}$ with $x$ and $y$ as above.  If $x \cup y$ is disjoint from $\alpha \cup \beta$,
then $s(e_0) = e_0$ and there is nothing to prove.  We can therefore assume that $x \cup y$ intersects
$\alpha \cup \beta$.

Let $T$ be a regular neighborhood of $\alpha \cup \beta$, so $T$ is a torus with one boundary component.
Since $\alpha$ and $\beta$ are the curves pictured above and $x \cup y$ intersects $\alpha \cup \beta$,
it is immediate from the construction in \cite{JohnsonFinite} that:
\begin{itemize}
\item both $x$ and $y$ intersect $T$ in single arcs $x_1$ and $y_1$; and
\item both $x_1$ and $y_1$ intersect $\alpha$ at most once and $\beta$ at most once.
\end{itemize}
Since $x$ is homologous to $y$, it follows\footnote{This uses the fact that $\partial T$ is a separating curve,
and in particular is null-homologous.} that $x_1$ is homologous to $y_1$ in $\HH_1(T,\partial T)$.
This implies that $x_1$ and $y_1$ are parallel properly embedded arcs in $T$.  There thus
exists an embedded $U \hookrightarrow T$ with $U \cong [0,1]^2$ such that $\partial U$ consists of
$x_1 \cup y_1$ along with two subarcs of $\partial T$.

There are now two possibilities.  The first is that $x$ and $y$ intersect one of $\alpha$ or $\beta$
and are disjoint from the other.  For concreteness, assume that they intersect $\alpha$.  We can then
find a small ball around these intersections that looks like the following figure:\\
\Figure{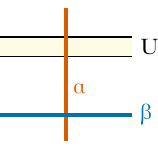}
Depending on whether $x$ is the top or bottom arc of the depicted portion of $U$, the loop $s(\alpha) = T_x T_y^{-1}(\alpha)$
looks like the following figure:\\
\Figure{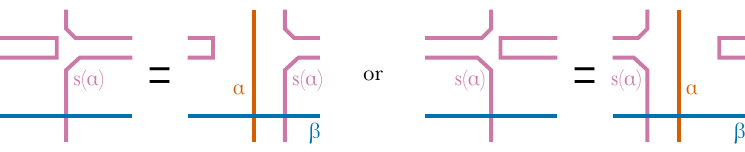}
In either case, $s(\alpha)$ is disjoint from $\alpha$ and intersects $\beta$ once.
Using this, we can then use the $ab$-strip\\
\Figure{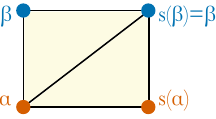}
to show that $e_0 \sim s(e_0)$.

The second possibility is that $x$ and $y$ intersect both $\alpha$ and $\beta$.  We can then find
a small open ball in $T$ that looks like the following figure:\footnote{This picture might not preserve orientations, but
this does not matter for our argument.}\\
\Figure{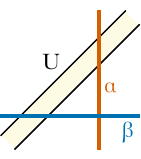}
Depending on whether $x$ is the top or bottom arc of the depicted portion of $U$, the loop $s(\alpha) = T_x T_y^{-1}(\alpha)$
looks like the following figure:\\
\Figure{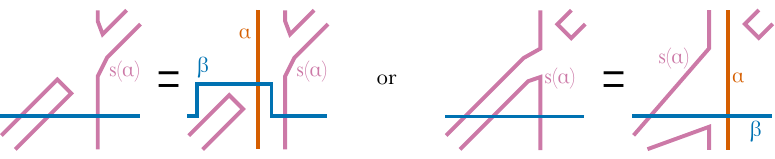}
In either case, $s(\alpha)$ is disjoint from
$\alpha$ and intersects $\beta$ once.  Similarly, $s(\beta)$ is disjoint from $\beta$, and
since $\alpha$ and $\beta$ intersect once, $s(\alpha)$ intersects $s(\beta)$ once.
The $ab$-strip\\
\Figure{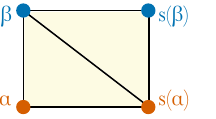}
then witnesses the fact that $e_0 \sim s(e_0)$.
\end{proof}

\subsection{Connectivity}
We can now prove:

\begin{proposition}
\label{proposition:mixedacyclic}
Let $g \geq 4$ and let $a,b \in \HH_1(\Sigma_g)$ be primitive elements with algebraic intersection number $1$.
Then $\cC_{ab}(\Sigma_g)$ is $1$-acyclic.
\end{proposition}
\begin{proof}
Let $\gamma$ be a loop in $\cC_{ab}(\Sigma_g)$.  We must prove that $\gamma$ is homologous
to a constant loop.  Homotoping $\gamma$, we can assume that it is a simplicial loop in the
$1$-skeleton.  Theorem~\ref{theorem:cvacyclic} says that $\cC_a(\Sigma_g)$ is $1$-acyclic,
so it is enough to prove that $\gamma$ is homologous to a loop lying in $\cC_a(\Sigma_g)$, i.e.,
to a loop taking all vertices to $a$-curves.

Assume that some vertices of $\gamma$ are mapped to $b$-curves.  If two adjacent vertices
are mapped to $b$-curves $\beta$ and $\beta'$, then by Lemma~\ref{lemma:mixpure} we can find
an $a$-curve $\alpha$ that intersects $\beta$ and $\beta'$ once.  This can be used to homotope
$\gamma$ to make $\beta$ and $\beta'$ the images of non-adjacent vertices:\\
\Figure{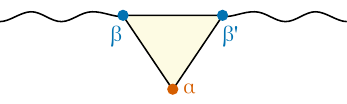}
Repeating this, we can can ensure that no adjacent vertices of $\gamma$ are mapped to $b$-curves.

To complete the proof, we must show how to eliminate vertices mapping to $b$-curves.  Consider
such a vertex, and let $e_0$ and $e_1$ be the images of the edges on either side of it.  Both
$e_0$ and $e_1$ are mixed $1$-simplices, so by Lemma~\ref{lemma:strip} there exists an $ab$-strip
connecting $e_0$ to $e_1$.  Attach this strip to $\gamma$ as follows:\\
\Figure{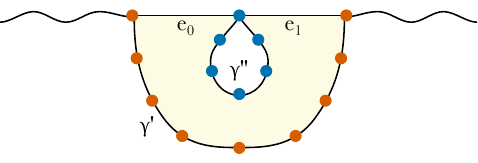}
As is shown in this figure, the result is that $\gamma$ is homologous to a sum of two loops: a
loop $\gamma'$ with one fewer vertex mapping to a $b$-curve, and another loop $\gamma''$ mapping
entirely to $\cC_b(\Sigma_g)$.  Another application of Theorem~\ref{theorem:cvacyclic} shows
that $\gamma''$ is null-homologous, so we deduce that $\gamma$ is homologous to $\gamma'$, as
desired.
\end{proof}

\section{Step 2.3: generators for cokernel}
\label{section:generators}

We first recall some notation.  Let $\alpha$ and $\beta$ be these curves on $\Sigma_g$:\\
\Figure{AandBcurvesNoCut}
Let $\lambda\colon \HH_2((\Torelli_g)_{\alpha};\Q) \oplus \HH_2((\Torelli_g)_{\beta};\Q) \rightarrow \HH_2(\Torelli_g;\Q)$
be the sum of the maps induced by the inclusions $(\Torelli_g)_{\alpha} \hookrightarrow \Torelli_g$
and $(\Torelli_g)_{\beta} \hookrightarrow \Torelli_g$ and let $\Lambda_g = \coker(\lambda)$.  Recall
that our goal is to prove Theorem~\ref{maintheorem:cokernel}, which says
that $\Lambda_g$ is a finite-dimensional algebraic representation of $\Sp_{2(g-1)}(\Z)$ for $g \geq 5$.
This section constructs generators for $\Lambda_g$.

\subsection{Commutator and Dehn twist conventions}
For a group $G$ and $x,y \in G$, our conventions are
$[x,y] = x^{-1} y^{-1} x y$ and $x^y = y^{-1} x y$.  
For a simple closed curve $\eta$ on a surface, $T_{\eta}$ denotes the right Dehn twist about $\eta$.
 
\subsection{Surface relations}
\label{section:surfacerelations}
To construct generators for $\Lambda_g$, we need a formalism for describing
elements of $\Lambda_g$.  Let $G$ be a group.  A {\em surface relation} in $G$ is a relation of the form
\[[x_1,y_1] \cdots [x_k,y_k] = 1 \quad \text{with $x_1,y_1,\ldots,x_k,y_k \in G$}.\]
Write this $r = [x_1,y_1] \cdots [x_k,y_k]$.  We emphasize that $r$ is a formal
product of commutators, not an element of $G$.
Let $\phi_r\colon \pi_1(\Sigma_k) \rightarrow G$ be the map
taking the standard generators of $\pi_1(\Sigma_k)$ to the $x_i$ and $y_i$.
Define $\fh(r) \in \HH_2(G;\Q)$ to be the image under
$(\phi_r)_{\ast}\colon \HH_2(\pi_1(\Sigma_k);\Q) \rightarrow \HH_2(G;\Q)$ of the fundamental
class of $\HH_2(\pi_1(\Sigma_k);\Z) = \Z$.

Writing $G = F/R$ for a free group $F$ and $R \lhd F$, Hopf's formula says that
\begin{equation}
\label{eqn:hopf}
\HH_2(G;\Z) = \frac{[F,F] \cap R}{[R,F]}.
\end{equation}
Each element in the numerator of this gives a surface relation $r$ in $G$, and
$\fh(r) \in \HH_2(G;\Q)$ is the associated element of homology.  See \cite{PutmanHopf}
for a discussion of Hopf's formula in these terms.  From
\eqref{eqn:hopf}, we see that the $\fh(r)$ satisfy several basic identities.
These will be used repeatedly throughout the remainder of the paper without citation.

First, consider surface relations $r = [x_1,y_1] \cdots [x_k,y_k]$ and $r' = [x'_1,y'_1]\cdots[x'_{k'},y'_{k'}]$.
Define
\[r r' = [x_1,y_1] \cdots [x_k,y_k] [x'_1,y'_1] \cdots [x'_{k'},y'_{k'}].\]
We then have $\fh(r r') = \fh(r) + \fh(r')$.

Since surface relations are formal product of commutators, it does not make literal sense to include terms
in them like $[x,y]^{-1}$.  In the context of surface relations, we therefore let $[x,y]^{-1}$ denote
the commutator\footnote{The reason for this is that $[x,y] [y,x] = 1$, so $[y,x]$ is the inverse to $[x,y]$.} 
$[y,x]$.  If $r$ is a surface relation and $r'$ is the surface relation
obtained from $r$ by deleting an adjacent pair of inverse commutators $[x,y] [x,y]^{-1}$,
then $\fh(r) = \fh(r')$.
Consequently, if
$r = [x_1,y_1] \cdots [x_k,y_k]$
is a surface relation and we define
$r^{-1} = [x_k,y_k]^{-1} \cdots [x_1,y_1]^{-1}$,
then $\fh(r^{-1}) = - \fh(r)$.

For $x,y,z \in G$, we let $[x,y]^z$ denote the commutator $[x^z,y^z]$.  If
$r$ is a surface relation containing two adjacent terms $[x_i,y_i] [x_{i+1},y_{i+1}]$ and
$r'$ is the surface relation obtained by replacing these with $[x_{i+1},y_{i+1}] [x_i,y_i]^{[x_{i+1},y_{i+1}]}$,
then $\fh(r) = \fh(r')$.
Also, if $r$ is a surface relation containing a commutator $[x z, y]$ and $r'$
is obtained by expanding this to $[x,y]^z [z,y]$, then $\fh(r) = \fh(r')$, and similarly
if $r$ contains $[x,zy]$.

\subsection{Elements in cokernel}
\label{section:llcomm}

How might elements in $\Lambda_g$ arise?
Assume that we have
\begin{itemize}
\item $A \in (\Torelli_g)_{\alpha}$ and
$B \in (\Torelli_g)_{\beta}$ such that $[A,B]$ fixes either $\alpha$ or $\beta$.
\end{itemize}

\begin{example}
\label{example:commex}
Let $\alpha,\alpha',\beta,\beta'$ be as follows, where $X$ and $Y$ are connected subsurfaces of $\Sigma_g$:\\
\Figure{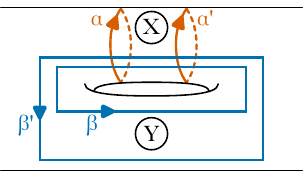}
Letting $A = T_{\alpha} T_{\alpha'}^{-1}$ and $B = T_{\beta'} T_{\beta}^{-1}$, the commutator $[A,B]$ fixes both $\alpha$ and $\beta$.
To see this, note that $A(\beta)$ is disjoint from $\beta \cup \beta'$ and $B(\alpha)$ is disjoint from $\alpha \cup \alpha'$:\\
\Figure{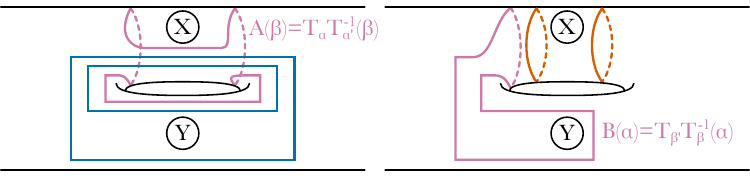}
This implies that
\[[A,B](\alpha) = A^{-1} B^{-1} A B(\alpha) = A^{-1} B^{-1} B(\alpha) = A^{-1}(\alpha) = \alpha,\]
and similarly that $[A,B](\beta) = \beta$.
\end{example}

We define $\LLComm{A,B} \in \Lambda_g$ as follows.
To simplify our notation, we will assume that $[A,B]$ fixes $\alpha$.  The case where
it fixes $\beta$ is similar.
Let $\pi\colon \HH_2(\Torelli_g;\Q) \rightarrow \Lambda_g$ be the projection.
The element $[A,B] \in (\Torelli_g)_{\alpha}$ vanishes in the abelianization
of $\Torelli_g$.  Theorem~\ref{theorem:putmaninjective} says\footnote{In unpublished work, Putman has
also proved that the map $\HH_1((\Torelli_g)_{\alpha};\Z) \rightarrow \HH_1(\Torelli_g;\Z)$ is
injective for $g \geq 4$.  Using this would allow us to take $n=1$ in the argument below,
simplifying several parts of our proof.  To avoid a dependence on unpublished work, we do
not use this integral statement.} that the map $\HH_1((\Torelli_g)_{\alpha};\Q) \rightarrow \HH_1(\Torelli_g;\Q)$
is injective for $g \geq 4$, so $[A,B]$ vanishes in
$\HH_1((\Torelli_g)_{\alpha};\Q)$.  This implies that there exists some $n \geq 1$ such that
$[A,B]^n$ vanishes in the abelianization of $(\Torelli_g)_{\alpha}$, so we can write
$[A,B]^n \fc = 1$
for some product $\fc$ of commutators in $(\Torelli_g)_{\alpha}$.  Define
\[\LLComm{A,B} = \pi\left(\frac{1}{n} \fh\left([A,B]^n \fc\right)\right) \in \Lambda_g.\]
This appears to depend on the choice of $n$ and $\fc$, but the following claim shows that it
is well-defined:

\begin{unnumberedclaim}
This does not depend the choice of $n$ and $\fc$.
\end{unnumberedclaim}
\begin{proof}[Proof of claim]
If $m \geq 1$ and $\fd$ is another product of commutators in $(\Torelli_g)_{\alpha}$ such that
$[A,B]^m \fd = 1$,
then
\begin{align*}
\frac{1}{n} \fh\left([A,B]^n \fc\right) - \frac{1}{m} \fh\left([A,B]^m \fd\right)
&= \frac{1}{nm} \left(\fh\left(\left([A,B]^n \fc\right)^m\right) + \fh\left(\left([A,B]^m \fd\right)^{-n}\right)\right) \\
&= \frac{1}{nm} \fh\left([A,B]^{nm} \fc' \fd' [A,B]^{-nm}\right) \\
&= \frac{1}{nm} \fh\left([A,B]^{nm} [A,B]^{-nm} \fc'' \fd''\right) \\
&= \frac{1}{nm} \fh\left(\fc'' \fd''\right) \in \Image\left(\HH_2\left((\Torelli_g)_{\alpha};\Q\right) \rightarrow \HH_2\left(\Torelli_g;\Q\right)\right).
\end{align*}
Here $\fc'$ and $\fc''$ are products of commutators in $(\Torelli_g)_{\alpha}$ obtained by commuting
terms of the form $[A,B]^{\pm 1}$ past terms in $\fc$, and similarly for $\fd'$ and $\fd''$.
The claim follows.
\end{proof}

There is one remaining ambiguity: if $[A,B]$ fixes $\alpha$ and $\beta$, then our
recipe gives two potentially different definitions of $\LLComm{A,B} \in \Lambda_g$.  However, if $[A,B]$
fixes $\alpha$ and $\beta$, then $[A,B] \in (\Torelli_g)_{\alpha, \beta} \cong \Torelli_{g-1}^1$.
In the above procedure, we can therefore choose $\fc$ to be a product of commutators lying in $(\Torelli_g)_{\alpha, \beta} \cong \Torelli_{g-1}^1$
whether
we are considering $[A,B]$ as an element of $(\Torelli_g)_{\alpha}$ or of $(\Torelli_g)_{\beta}$.

\subsection{Interlude: twisted surface groups}
\label{section:twistedsurface}

For $n,k \geq 1$, let $\Gamma_{n,k}$ be the following group:
\[\Gamma_{n,k} = \GroupPres{$z_0,w_0,\ldots,z_k,w_k$}{$[z_0,w_0]^n [z_1,w_1] \cdots [z_k,w_k] = 1$}.\]
These arise naturally in the construction of the elements $\LLComm{A,B}$ above, which
are multiples of the homology classes associated to surface relators
\[[A,B]^n [c_1,d_1] \cdots [c_k,d_k] = 1 \quad \text{with $A,B,c_i,d_i \in \Torelli_g$}.\]
The associated maps $\pi_1(\Sigma_{n+k}) \rightarrow \Torelli_g$ factor through
$\Gamma_{n,k}$.
The homology of $\Gamma_{n,k}$ is given by:

\begin{lemma}
\label{lemma:homologygamma}
For $n,k \geq 1$, we have
$\HH_1(\Gamma_{n,k}) = \Z^{2k+2}$ and $\HH_2(\Gamma_{n,k}) = \Z$ and $\HH_d(\Gamma_{n,k}) = 0$
for $d \geq 3$.
\end{lemma}
\begin{proof}
The single relation in $\Gamma_{n,k}$ can be rewritten as
\[[z_0,w_0]^{-n} = [z_1,w_1] \cdots [z_k,w_k].\]
Letting $F(S)$ denote the free group on a set $S$, this implies that
$\Gamma_{n,k}$ can be decomposed as an amalgamated free product\footnote{The
Freiheitsatz \cite{MagnusFrei, PutmanOneRelator} for
one-relator groups implies
that $\{z_0,w_0\}$ and $\{z_1,w_1,\ldots,z_k,w_k\}$ generate free subgroups
of $\Gamma_{n,k}$.  This general result is unnecessary since our relation can be written as
$r_1 = r_2$ with $r_1 \in F(z_0,w_0)$ and $r_2 \in F(z_1,w_1,\ldots,z_k,w_k)$, giving this decomposition as a free product with amalgamation.}
\[\Gamma_{n,k} = F(z_0,w_0) \ast_{\Z} F(z_1,w_1,\ldots,z_k,w_k),\]
where the infinite cyclic group $\Z$ is identified with the cyclic subgroups generated by
$[z_0,w_0]^{-n} \in F(z_0,w_0)$ and $[z_1,w_1] \cdots [z_k,w_k] \in F(z_1,w_1,\ldots,z_k,w_k)$.
Since $\HH_1(\Z) = \Z$, the lemma follows from the associated Mayer-Vietoris sequence in group homology whose nonzero terms are
\begin{small}
\[\begin{tikzcd}[column sep=small]
0 \arrow{r} & \HH_2(\Gamma_{n,k}) \arrow{r} & \HH_1(\Z) \arrow{r}{0} & \HH_1(F(z_0,w_0)) \oplus \HH_1(F(z_1,w_1,\ldots,z_k,w_k)) \arrow{r} & \HH_1(\Gamma_{n,k}) \arrow{r} & 0.
\end{tikzcd}\qedhere\]
\end{small}
\end{proof}

To identify our elements $\LLComm{A,B} \in \HH_2(\Torelli_g;\Q)$ with
terms we will construct using equivariant homology, we need a space for $\Gamma_{n,k}$
to act on.  For this, define
\[\Gamma'_{n,k} = \GroupPres{$z_0,u_0,z_1,w_1,\ldots,z_k,w_k$}{$(z_0^{-1} u_0)^n [z_1,w_1] \cdots [z_k,w_k] = 1$}.\]
There is a homomorphism $\Gamma'_{n,k} \rightarrow \Gamma_{n,k}$ taking $u_0$ to $z_0^{w_0} = w_0^{-1} z_0 w_0$.  The following implies that this
is injective, so henceforth we can identify $\Gamma'_{n,k}$ with a subgroup of $\Gamma_{n,k}$:

\begin{lemma}
\label{lemma:hnn}
The group $\Gamma_{n,k}$ is an HNN extension of $\Gamma'_{n,k}$ over an infinite cyclic subgroup
with stable letter $w_0$ conjugating $z_0$ to $u_0$.
\end{lemma}
\begin{proof}
This HNN extension can be written
\[\GroupPres{$w_0,z_0,u_0,z_1,w_1,\ldots,z_k,w_k$}{$(z_0^{-1} u_0)^n [z_1,w_1] \cdots [z_k,w_k] = 1$, $w_0^{-1} z_0 w_0 = u_0$}.\]
A Tietze transformation now eliminates $u_0$ and turns this into
the presentation for $\Gamma_{n,k}$.
\end{proof}

By Bass--Serre theory \cite{SerreTrees}, the following is an immediate corollary
of Lemma~\ref{lemma:hnn}:

\begin{corollary}
\label{corollary:tree}
For $n,k \geq 1$, the group $\Gamma_{n,k}$ acts without rotations on a tree $T$ such that:
\begin{itemize}
\item $\Gamma_{n,k}$ acts transitively on the vertices and edges of $T$, so
$T/\Gamma_{n,k} \cong S^1$; and
\item there is an edge $\tau$ of $T$ going from a vertex $\tau_0$ to a vertex $\tau_1$ such
that
$(\Gamma_{n,k})_{\tau_0} = \Gamma'_{n,k}$ and $(\Gamma_{n,k})_{\tau} = \Span{z_0} \cong \Z$
and $w_0 \Cdot \tau_0 = \tau_1$.
\end{itemize}
\end{corollary}

We will also need the homology of $\Gamma'_{n,k}$:

\begin{lemma}
\label{lemma:homologygammaprime}
For $n,k \geq 1$, we have $\HH_1(\Gamma'_{n,k}) \cong \Z^{2k+1} \oplus \Z/n$ and
$\HH_d(\Gamma'_{n,k}) = 0$ for $d \geq 2$.
\end{lemma}
\begin{proof}
Just like in the proof of Lemma~\ref{lemma:homologygamma}, write $\Gamma'_{n,k}$ as an amalgamated free product
\[\Gamma'_{n,k} = F(z_0,u_0) \ast_{\Z} F(z_1,w_1,\ldots,z_k,w_k),\]
where the infinite cyclic group $\Z$ is identified with the cyclic subgroups generated by
$(z_0^{-1} u_0)^{-n} \in F(z_0,u_0)$ and $[z_1,w_1] \cdots [z_k,w_k] \in F(z_1,w_1,\ldots,z_k,w_k)$.
Since $\HH_1(\Z) = \Z$, the lemma follows from the associated Mayer-Vietoris sequence in group homology whose nonzero terms are
\begin{small}
\[\begin{tikzcd}[column sep=small]
0 \arrow{r} & \HH_2(\Gamma'_{n,k}) \arrow{r} & \HH_1(\Z) \arrow{r}{f} & \HH_1(F(u_0,w_0)) \oplus \HH_1(F(z_1,w_1,\ldots,z_k,w_k)) \arrow{r} & \HH_1(\Gamma_{n,k}) \arrow{r} & 0.
\end{tikzcd}\]
\end{small}
Here $f$ is the injective map taking the generator of $\HH_1(\Z) = \Z$ to
\[\left(n [u_0]-n [z_0],0\right) \in \HH_1(F(u_0,w_0)) \oplus \HH_1(F(z_1,w_1,\ldots,z_k,w_k)). \qedhere\]
\end{proof}

\subsection{Generators}

We now return to the $\LLComm{A,B}$.  
How can we find $A \in (\Torelli_g)_{\alpha}$ and $B \in (\Torelli_g)_{\beta}$ such that $[A,B]$ fixes either
$\alpha$ or $\beta$?  One way for this to hold is for $A$ to fix $B(\alpha)$, in which case
$[A,B]$ will fix $\alpha$.  Similarly, if $B$ fixes $A(\beta)$ then $[A,B]$ will fix $\beta$.
For instance, both of these hold in Example \ref{example:commex}.
The following says that the $\LLComm{A,B}$ coming from certain $A$ and $B$ of this form generate $\Lambda_g$:

\begin{proposition}
\label{proposition:generators}
Let $g \geq 5$, and let $\alpha$ and $\beta$ be the curves discussed above.  Then $\Lambda_g$ is spanned
by elements of the form $\LLComm{A,B}$, where $A \in (\Torelli_g)_{\alpha}$ and $B \in (\Torelli_g)_{\beta}$
satisfy either:\footnote{The condition that $A(\beta)$ is disjoint from $\beta$ is equivalent to saying that $\{\beta,A(\beta)\}$
is a simplex of $\cC_{ab}(\Sigma_g)$, and similarly when $B(\alpha)$ is disjoint from $\alpha$.}
\begin{itemize}
\item $A(\beta)$ is disjoint from $\beta$, and $B$ fixes $A(\beta)$; or
\item $B(\alpha)$ is disjoint from $\alpha$, and $A$ fixes $B(\alpha)$.
\end{itemize}
\end{proposition}
\begin{proof}
In this proof, all homology is taken with $\Q$ coefficients.\footnote{This is just to simplify
our notation in this proof.  It does not represent a change in our general conventions for writing
homology groups.}
Let $\Lambda'_g < \Lambda_g$ be the span of the purported generators.  Our goal is to prove
that $\Lambda'_g = \Lambda_g$.
Let $a = [\alpha]$ and $b = [\beta]$, and consider the $\Torelli_g$-equivariant homology of the handle complex
$\cC_{ab}(\Sigma_g)$ from \S \ref{section:handlecomplex}.  
Proposition~\ref{proposition:mixedacyclic} says that $\cC_{ab}(\Sigma_g)$ is $1$-acyclic, so by Lemma~\ref{lemma:connected} the canonical map
\[\HH^{\Torelli_g}_2(\cC_{ab}(\Sigma_g)) \longrightarrow \HH_2(\Torelli_g)\]
is surjective.  
Lemma~\ref{lemma:withoutrotations} says that $\Torelli_g$ acts on $\cC_{ab}(\Sigma_g)$ without rotations, so we
can study $\HH^{\Torelli_g}_2(\cC_{ab}(\Sigma_g))$ using the spectral
sequence from Proposition~\ref{proposition:spectralsequence}:
\[\ssE^1_{pq} = \hspace{-15pt}\bigoplus_{\sigma \in (\cC_{ab}(\Sigma_g)/\Torelli_g)^{(p)}} \hspace{-15pt}\HH_q((\Torelli_g)_{\tsigma}) \Rightarrow \HH^{\Torelli_g}_{p+q}(\cC_{ab}(\Sigma_g)).\]
As was discussed in \S \ref{section:bottomrow}, the terms $\ssE^1_{p0}$ are exactly the cellular chain complex
of $\cC_{ab}(\Sigma_g)/\Torelli_g$.  As for the other entries, Corollary \ref{corollary:quotient} says that
the $1$-skeleton of $\cC_{ab}(\Sigma_g)/\Torelli_g$ consists of:
\begin{itemize}
\item[(i)] vertices $v_a$ and $v_b$, with $v_a$ and $v_b$ the images of $\alpha$ and $\beta$, respectively.
\item[(ii)] loops based at $v_a$ and $v_b$, each the image of a pure $1$-simplex.
\item[(iii)] a $1$-simplex $e_{ab}$ joining $v_a$ and $v_b$, with $e_{ab}$ the image of the edge $\{\alpha,\beta\}$.
\end{itemize}
It follows that the $\ssE^1$-page of our spectral sequence is
\SetTblrInner{colsep=2pt}
\[\text{\footnotesize\begin{tblr}{|ccccc}
$\HH_2((\Torelli_g)_{\alpha}) \oplus \HH_2((\Torelli_g)_{\beta})$ & $\leftarrow$ & $\ast$ & $\leftarrow$ & $\ast$ &\\
$\HH_1((\Torelli_g)_{\alpha}) \oplus \HH_1((\Torelli_g)_{\beta})$ & $\leftarrow$ & $\HH_1((\Torelli_g)_{\alpha,\beta}) \oplus\hspace{-15pt}\bigoplus_{e \in (\cC_a(\Sigma_g)/\Torelli_g)^{(1)}} \hspace{-15pt}\HH_1((\Torelli_g)_{\te}) \oplus\hspace{-15pt}\bigoplus_{e \in (\cC_b(\Sigma_g)/\Torelli_g)^{(1)}}\hspace{-15pt}\HH_1((\Torelli_g)_{\te})$ & $\leftarrow$ & $\ast$ \\
$C_0(\cC_{ab}(\Sigma_g)/\Torelli_g)$ & $\leftarrow$ & $C_1(\cC_{ab}(\Sigma_g)/\Torelli_g)$ & $\leftarrow$ & $C_2(\cC_{ab}(\Sigma_g)/\Torelli_g)$ \\
\hline
\end{tblr}}\]
Proposition~\ref{proposition:quotientcontractible} says that $\cC_{ab}(\Sigma_g)/\Torelli_g$ is contractible, so 
\begin{equation}
\label{eqn:killbottomrow}
\ssE^{\infty}_{p0} = \ssE^{2}_{p0} = 0 \quad \text{for $p \geq 1$}.
\end{equation}
Moreover, as we discussed in \S \ref{section:leftcol} the composition
\begin{equation}
\label{eqn:constructcokernel}
\begin{tikzcd}
\HH_2((\Torelli_g)_{\alpha}) \oplus \HH_2((\Torelli_g)_{\beta}) = \ssE^1_{02} \arrow[two heads]{r} & \ssE^{\infty}_{02} \arrow[hook]{r} & \HH^{\Torelli_g}_2(\cC_{ab}(\Sigma_g)) \arrow[two heads]{r} & \HH_2(\Torelli_g)
\end{tikzcd}
\end{equation}
is the sum of the maps induced by the inclusions $(\Torelli_g)_{\alpha} \hookrightarrow \Torelli_g$ and
$(\Torelli_g)_{\beta} \hookrightarrow \Torelli_g$.  It follows that the cokernel of \eqref{eqn:constructcokernel}
is $\Lambda_g$.  Combining this with \eqref{eqn:killbottomrow}, we deduce that the surjection
$\HH^{\Torelli_g}_2(\cC_{ab}(\Sigma_g)) \twoheadrightarrow \HH_2(\Torelli_g)$
induces a surjection
\begin{equation}
\label{eqn:toquotient}
\begin{tikzcd}
\ssE^2_{11} = \ssE^{\infty}_{11} \arrow[two heads]{r} & \Lambda_g.
\end{tikzcd}
\end{equation}
To prove that $\Lambda'_g = \Lambda_g$, it is enough to prove that the image of 
\eqref{eqn:toquotient} is contained in $\Lambda'_g$.
The vector space $\ssE^2_{11}$ is a quotient of the kernel of the following differential $\partial\colon \ssE^1_{11} \rightarrow \ssE^1_{01}$:
\[\partial\colon \HH_1((\Torelli_g)_{\alpha,\beta}) \oplus\hspace{-10pt}\bigoplus_{e \in (\cC_a(\Sigma_g)/\Torelli_g)^{(1)}} \hspace{-10pt}\HH_1((\Torelli_g)_{\te}) \oplus\hspace{-10pt}\bigoplus_{e \in (\cC_b(\Sigma_g)/\Torelli_g)^{(1)}}\hspace{-10pt}\HH_1((\Torelli_g)_{\te}) \longrightarrow \HH_1((\Torelli_g)_{\alpha}) \oplus \HH_1((\Torelli_g)_{\beta}).\]
To prove that the image of \eqref{eqn:toquotient} lies in $\Lambda'_g$, it is enough to prove the following two claims:

\begin{claim}{1}
The kernel of the differential $\partial\colon \ssE^1_{11} \rightarrow \ssE^1_{01}$ is
\begin{equation}
\label{eqn:claimedkernel}
\bigoplus_{e \in (\cC_a(\Sigma_g)/\Torelli_g)^{(1)}} \hspace{-10pt}\HH_1((\Torelli_g)_{\te}) \oplus\hspace{-10pt}\bigoplus_{e \in (\cC_b(\Sigma_g)/\Torelli_g)^{(1)}}\hspace{-10pt}\HH_1((\Torelli_g)_{\te}).
\end{equation}
\end{claim}

We must prove $\partial$ vanishes on \eqref{eqn:claimedkernel} and that
the restriction of $\partial$ to $\HH_1((\Torelli_g)_{\alpha,\beta})$ is an injection.
We start with the latter fact.  From the description of the differentials in \S \ref{section:differentials},
we see that on $\HH_1((\Torelli_g)_{\alpha,\beta})$ our differential is the difference between the maps
\[\begin{tikzcd}
\HH_1((\Torelli_g)_{\alpha,\beta}) \arrow{r} & \HH_1((\Torelli_g)_{\alpha}) \arrow[hook]{r} & \HH_1((\Torelli_g)_{\alpha}) \oplus \HH_1((\Torelli_g)_{\beta})
\end{tikzcd}\]
and
\[\begin{tikzcd}
\HH_1((\Torelli_g)_{\alpha,\beta}) \arrow{r} & \HH_1((\Torelli_g)_{\beta}) \arrow[hook]{r} & \HH_1((\Torelli_g)_{\alpha}) \oplus \HH_1((\Torelli_g)_{\beta})
\end{tikzcd}\]
The fact that 
this difference is injective follows from Theorem~\ref{theorem:putmaninjective}, which
implies that 
$\HH_1((\Torelli_g)_{\alpha,\beta}) \rightarrow \HH_1((\Torelli_g)_{\alpha})$ and $\HH_1((\Torelli_g)_{\alpha,\beta}) \rightarrow \HH_1((\Torelli_g)_{\beta})$
are both injective.\footnote{In fact, we only need that one of them is injective.}

We now prove that $\partial$ vanishes on \eqref{eqn:claimedkernel}.  We will handle
the terms coming from $\cC_a(\Sigma_g)$; the other case is similar.  Consider a $1$-cell $e$ of $\cC_a(\Sigma_g)/\Torelli_g$.  
Corollary \ref{corollary:quotient} says that $e$ is a loop.  Using Lemma~\ref{lemma:normalform}, we can lift $e$ to an
edge $\te$ of $\cC_a(\Sigma_g)$ going from $\alpha$ to $B(\alpha)$ for some $B \in (\Torelli_g)_{\beta}$.  As we discussed in \S \ref{section:differentials}, on 
$\HH_1((\Torelli_g)_{\alpha,B(\alpha)})$ our differential is the composition
\begin{equation}
\label{eqn:loopvanish}
\begin{tikzcd}
\HH_1((\Torelli_g)_{\alpha,B(\alpha)}) \arrow{r} & \HH_1((\Torelli_g)_{\alpha}) \arrow[hook]{r} & \HH_1((\Torelli_g)_{\alpha}) \oplus \HH_1((\Torelli_g)_{\beta})
\end{tikzcd}
\end{equation}
where the first map takes the homology class of $x \in (\Torelli_g)_{\alpha,B(\alpha)}$
to the homology class of $[x,B] \in (\Torelli_g)_{\alpha}$.  Theorem 
\ref{theorem:putmaninjective} says that the inclusion $(\Torelli_g)_{\alpha} \hookrightarrow \Torelli_g$
induces an injection $\HH_1((\Torelli_g)_{\alpha}) \hookrightarrow \HH_1(\Torelli_g)$.  Since
in $\HH_1(\Torelli_g)$ the homology class of the commutator $[x,B]$ vanishes, the same is true
in $\HH_1((\Torelli_g)_{\alpha})$.  We conclude that \eqref{eqn:loopvanish} is zero, as desired.

\begin{claim}{2}
\label{claim:identify}
Let $e$ be a $1$-simplex of either $\cC_a(\Sigma_g) / \Torelli_g$ or $\cC_b(\Sigma_g) / \Torelli_g$.  Then the image of the composition
$\HH_1((\Torelli_g)_{\te}) \rightarrow \ssE^2_{11} \twoheadrightarrow \Lambda_g$
is contained in $\Lambda'_g$.
\end{claim}

We will assume that $e$ is a $1$-simplex of $\cC_a(\Sigma_g) / \Torelli_g$.  The other case is similar.
By Lemma~\ref{lemma:normalform}, we can construct 
a lift $\te$ of $e$ to $\cC_{ab}(\Sigma_g)$ going from $\alpha$ to $\alpha' = B(\alpha)$ for some $B \in (\Torelli_g)_{\beta}$.  Consider $A \in (\Torelli_g)_{\alpha,B(\alpha)}$.
The element $\LLComm{A,B} \in \Lambda_g$ is one of our generators for $\Lambda'_g$,
so it is enough to prove that the composition
\[\begin{tikzcd}
\HH_1((\Torelli_g)_{\alpha,B(\alpha)}) \arrow{r} & \ssE^2_{11} \arrow[two heads]{r} & \Lambda_g
\end{tikzcd}\]
takes the homology class of $A$ to a multiple of $\LLComm{A,B}$.
We divide the proof of this into 3 steps.  We reiterate that during this proof all homology has $\Q$-coefficients.

\begin{step}{2.1}
\label{step:identify2}
For some $n,k \geq 1$, we construct a homomorphism\footnote{Here
$\Gamma_{n,k} = \GroupPres{$z_0,w_0,\ldots,z_k,w_k$}{$[z_0,w_0]^n [z_1,w_1] \cdots [z_k,w_k] = 1$}$
is the group defined in \S \ref{section:twistedsurface}.}
$f\colon \Gamma_{n,k} \rightarrow \Torelli_g$
such that the image of the composition
\begin{equation}
\label{eqn:22map}
\begin{tikzcd}
\HH_2(\Gamma_{n,k}) \arrow{r}{f_{\ast}} & \HH_2(\Torelli_g) \arrow{r} & \Lambda_g
\end{tikzcd}
\end{equation}
is the $\Q$-span of $\LLComm{A,B}$.
\end{step}

Recall that $\LLComm{A,B}$ is
the image in $\Lambda_g$ of the element of $\HH_2(\Torelli_g)$ constructed as follows.
We have $[A,B] \in (\Torelli_g)_{\alpha}$, and we can find $n \geq 1$ and
$x_1,y_2,\ldots,x_k,y_k \in (\Torelli_g)_{\alpha}$ such that
\[[A,B]^n [x_1,y_1] \cdots [x_k,y_k] = 1.\]
Then $\LLComm{A,B}$ is the image in $\Lambda_g$ of the homology class
\[\frac{1}{n} \fh([A,B]^n [x_1,y_1] \cdots [x_k,y_k]) \in \HH_2(\Torelli_g).\]
Let $f\colon \Gamma_{n,k} \rightarrow \Torelli_g$ be the map defined by
\[\text{$f(z_0) = A$ and $f(w_0) = B$ and $f(z_i) = x_i$ and $f(w_i) = y_i$ for $1 \leq i \leq k$.}\]
Lemma~\ref{lemma:homologygamma} says that $\HH_2(\Gamma_{n,k}) = \Q$, and the image
of $f_{\ast}\colon \HH_2(\Gamma_{n,k}) \rightarrow \HH_2(\Torelli_g)$ is the $\Q$-span 
of $\fh([A,B]^n [x_1,y_1] \cdots [x_k,y_k])$.  It follows that the image of \eqref{eqn:22map}
is the $\Q$-span of $\LLComm{A,B}$.

\begin{step}{2.2}
We construct a tree $T$ on which $\Gamma_{n,k}$ acts without rotations
and a commutative diagram in equivariant homology
\[\begin{tikzcd}
\HH^{\Gamma_{n,k}}_{2}(T) \arrow{r} \arrow{d}{\cong} & \HH^{\Torelli_g}_{2}(\cC_{ab}(\Sigma_g)) \arrow{d} \\
\HH_{2}(\Gamma_{n,k})     \arrow{r}{f_{\ast}}        & \HH_{2}(\Torelli_g).
\end{tikzcd}\]
\end{step}

Recall that $\Gamma'_{n,k}$ is the subgroup of $\Gamma_{n,k}$ generated by
$\{z_0,w_0^{-1} z_0 w_0, z_1, w_1,\ldots,z_k,w_k\}$.  Corollary \ref{corollary:tree}
gives a tree $T$ on which $\Gamma_{n,k}$ acts without rotations such that:
\begin{itemize}
\item $\Gamma_{n,k}$ acts transitively on the vertices and edges of $T$, so
$T/\Gamma_{n,k} \cong S^1$; and
\item there is an edge $\tau$ of $T$ going from a vertex $\tau_0$ to a vertex $\tau_1$ such
that
\begin{equation}
\label{eqn:stabstuff}
(\Gamma_{n,k})_{\tau_0} = \Gamma'_{n,k} \quad \text{and} \quad (\Gamma_{n,k})_{\tau} = \Span{z_0} \cong \Z
\quad \text{and} \quad w_0 \Cdot \tau_0 = \tau_1.
\end{equation}
\end{itemize}
Since $T$ is contractible, the canonical map
$\HH^{\Gamma_{n,k}}_d(T) \rightarrow \HH_d(\Gamma_{n,k})$ is
an isomorphism for $d \geq 0$.  

We claim that $f(\Gamma'_{n,k}) \subset (\Torelli_g)_{\alpha}$.  This can be checked on generators:
\begin{itemize}
\item $f(z_0) = A$ fixes the initial point $\alpha$ and the endpoint $\alpha' = B(\alpha)$ of $\te$; and
\item $f(w_0^{-1} z_0 w_0) = B^{-1} A B$ fixes $\alpha$ since $A$ fixes $\alpha$ and $B(\alpha)$; and
\item for $1 \leq i \leq k$, the elements $f(z_i) = x_i$ and $f(w_i) = y_i$ fix $\alpha$ by construction.
\end{itemize}
In light of \eqref{eqn:stabstuff} and the fact that $\Gamma_{n,k}$ acts transitively on the vertices
and edges of $T$, we can therefore define a cellular map $\phi\colon T \rightarrow \cC_{ab}(\Sigma_g)$ via
the formula
\[\phi(v \tau) = f(v) \te \quad \text{and} \quad \phi(v \tau_0) = f(v) \alpha \quad \text{for all $v \in \Gamma_{n,k}$}.\]
By the functoriality of equivariant homology discussed in \S \ref{section:basicdef}, the maps $f$ and $\phi$
induce a map $\HH^{\Gamma_{n,k}}_{2}(T) \rightarrow \HH^{\Torelli_g}_{2}(\cC_{ab}(\Sigma_g))$ fitting
into the claimed commutative diagram.

\begin{step}{2.3}
We use these constructions to prove our goal: that under the map
\[\begin{tikzcd}
\HH_1((\Torelli_g)_{\alpha,B(\alpha)}) \arrow{r} & \ssE^2_{11} \arrow[two heads]{r} & \Lambda_g
\end{tikzcd}\]
the homology class of $A$ maps to a multiple of $\LLComm{A,B} \in \Lambda_g$.
\end{step}

Proposition~\ref{proposition:spectralsequence} gives a spectral sequence
\[\ssF^1_{pq} = \hspace{-15pt}\bigoplus_{\sigma \in (T/\Gamma_{n,k})^{(p)}} \hspace{-15pt}\HH_q((\Gamma_{n,k})_{\tsigma}) \Rightarrow \HH^{\Gamma_{n,k}}_{p+q}(T).\]
Using Lemma~\ref{lemma:homologygammaprime} to identify $\HH_{\bullet}(\Gamma'_{n,k})$, the nonzero terms of the $\ssF^1$-page are
\[\text{\begin{tblr}{|ccc}
$\HH_1((\Gamma_{n,k})_{\tau_0})$ & $\leftarrow$ & $\HH_1((\Gamma_{n,k})_{\tau})$ \\
$\HH_0((\Gamma_{n,k})_{\tau_0})$ & $\leftarrow$ & $\HH_0((\Gamma_{n,k})_{\tau})$ \\
\hline
\end{tblr}}
\cong
\text{\begin{tblr}{|ccc}
$\HH_1(\Gamma'_{n,k})$ & $\leftarrow$ & $\HH_1(\Span{z_0})$ \\
$\HH_0(\Gamma'_{n,k})$ & $\leftarrow$ & $\HH_0(\Span{z_0})$\\
\hline
\end{tblr}}
\cong
\text{\begin{tblr}{|ccc}
$\Q^{2k+1}$ & $\leftarrow$ & $\Q$ \\
$\Q$        & $\leftarrow$ & $\Q$ \\
\hline
\end{tblr}}\]
This converges to the homology of $\Gamma_{n,k}$, which by Lemma~\ref{lemma:homologygamma} has
$\HH_1(\Gamma_{n,k}) \cong \Q^{2k+2}$ and $\HH_2(\Gamma_{n,k}) = \Q$.
The differentials on $\ssF^1$ are thus all $0$, and $\ssF$ degenerates at $\ssF^1$.  In particular,
\[\HH^{\Gamma_{n,k}}_2(T) = \ssF^{\infty}_{11} = \ssF^{1}_{11}.\]
The spectral sequence from Proposition~\ref{proposition:spectralsequence} is functorial, so there
is a map of spectral sequences $\ssF \rightarrow \ssE$ converging to the map
$\HH^{\Gamma_{n,k}}_{\bullet}(T) \rightarrow \HH^{\Torelli_g}_{\bullet}(\cC_{ab}(\Sigma_g))$
induced by $f$ and $\phi$.  Since $f(z_0) = A \in (\Torelli_g)_{\alpha,B(\alpha)}$, we have a commutative
diagram
\[\begin{tikzcd}
\HH_1(\Span{z_0}) \arrow{d} \arrow{r}{\cong} & \ssF^2_{11} \arrow{d} \\
\HH_1((\Torelli_g)_{\alpha,B(\alpha)}) \arrow{r} & \ssE^2_{11}
\end{tikzcd}\]
and the image of $\ssF^2_{11} \rightarrow \ssE^2_{11}$ is the $\Q$-span of the image of the
homology class of $A \in (\Torelli_g)_{\alpha,B(\alpha)}$.  As we discussed in Step \ref{step:identify2},
the image of 
\[\ssF^2_{11} = \HH^{\Gamma_{n,k}}_2(T) = \HH_2(\Gamma_{n,k}) \cong \Q\] 
in $\Lambda_g$ is the line spanned by $\LLComm{A,B}$.  The step follows.
\end{proof}

\part{Homology of Torelli, step 3: algebraicity of cokernel} 
\label{part:step3}

We close by proving Theorem~\ref{maintheorem:cokernel}, whose statement we recall in \S \ref{section:step3intro}.
To avoid having to constantly impose genus hypotheses, we make the following standing assumption:

\begin{assumption}
\label{assumption:step3genus}
Throughout Part \ref{part:step3}, we fix some $g \geq 5$.
\end{assumption}

\section{Introduction to step 3}
\label{section:step3intro}

After fixing some notation, we outline the rest of the paper.

\subsection{Recollection of goal}
Let $\alpha$ and $\beta$ be the following curves on $\Sigma_g$:\\
\Figure{AandBcurvesNoCut}
Recall that we are trying to prove Theorem~\ref{maintheorem:cokernel}, which says that the cokernel
$\Lambda_g$ of the map
$\lambda\colon \HH_2((\Torelli_g)_{\alpha};\Q) \oplus \HH_2((\Torelli_g)_{\beta};\Q) \rightarrow \HH_2(\Torelli_g;\Q)$
is a finite-dimensional algebraic representation of $\Sp_{2(g-1)}(\Z)$.

\subsection{Homology}
\label{section:homology}

The action of $\Sp_{2(g-1)}(\Z)$
on $\Lambda_g$ is induced by the conjugation action of
$(\Mod_g)_{\alpha,\beta}$ on $\Torelli_g$.  Set
$a = [\alpha]$ and $b = [\beta]$.  Let $H_{\Z}$ be the
orthogonal complement in $\HH_1(\Sigma_g;\Z)$ of $a$ and $b$ with respect to
the algebraic intersection form.  We can
identify $\Sp_{2(g-1)}(\Z)$ with $\Sp(H_{\Z})$.  Similarly,
let $H$ be the orthogonal complement in $\HH_1(\Sigma_g;\Q)$
of $a$ and $b$.  

\subsection{Outline}
\label{section:outlinestep3}

Proposition~\ref{proposition:generators} gives generators $\LLComm{A,B}$ for $\Lambda_g$.
Roughly speaking, we will find enough relations between the $\LLComm{A,B}$ to force
$\Lambda_g$ to be a subquotient of $(\wedge^2 H)^{\otimes 2}$.  There are six steps:
\begin{itemize}
\item \S \ref{section:step3notation} establishes some terminology and notation for the $\LLComm{A,B}$.
\item \S \ref{section:step3compatible} and \S \ref{section:step3shifters} show how to interpret
the $A$ and $B$ in $\LLComm{A,B}$ in terms of $H$.
\item \S \ref{section:step3genset} refines our generating set for $\Lambda_g$.
\item \S \ref{section:step3identifygen} identifies some redundancies in our refined generating set.
\item \S \ref{section:step3proof} introduces a subquotient of $(\wedge^2 H)^{\otimes 2}$ called
the {\em symmetric kernel}, and discusses a presentation of it from the authors' recent paper
\cite{MinahanPutmanRepPresentations}.  We then use this presentation to prove that $\Lambda_g$ is a quotient
of the symmetric kernel.  This implies that $\Lambda_g$ is a finite-dimensional algebraic representation
of $\Sp(H_{\Z})$, as claimed by Theorem~\ref{maintheorem:cokernel}
\end{itemize}

\subsection{Standing notation}
\label{section:standingnotation}

We already fixed $g \geq 5$ in Assumption \ref{assumption:step3genus}.
To avoid having to re-introduce notation in each section, we fix the following notation once and for all
as above: $\alpha$, $\beta$, $a$, $b$, $H$, $H_{\Z}$.  

\section{Step 3.1: notation for generators}
\label{section:step3notation}

We introduce some terminology for
the generators of $\Lambda_g$ given by Proposition~\ref{proposition:generators}.

\subsection{Shifters and compatibility}
The generators come in two families.
Say that $A \in \Torelli_g$ is a {\em $\beta$-shifter} if $A$ fixes $\alpha$ and $A(\beta)$ is disjoint from $\beta$.  In that
case, say that $B \in \Torelli_g$ is {\em compatible} with $A$ if $B$ fixes $\beta$ and $A(\beta)$.  We
then have a generator $\LLComm{A,B} \in \Lambda_g$, which for clarity we will denote $\LLBComm{A,B}$.
Similarly, say that $B \in \Torelli_g$ is an {\em $\alpha$-shifter} if $B$ fixes $\beta$ and $B(\alpha)$ is disjoint from $\alpha$.
In that case, say that $A \in \Torelli_g$ is {\em compatible} with $B$ if $A$ fixes $\alpha$ and $B(\alpha)$.  We
then have a generator $\LLComm{A,B} \in \Lambda_g$, which for clarity we will denote $\LLAComm{A,B}$.

\begin{remark}
\label{remark:distinct}
It is possible for both $\LLAComm{A,B}$ and $\LLBComm{A,B}$ to be defined, in which case they are identical elements
of $\Lambda_g$.
\end{remark}

\subsection{Left and right sides}

Let $A$ be a $\beta$-shifter, so $\beta$ and $A(\beta)$ are disjoint.
Since we always consider curves up to isotopy, it is possible
for $A(\beta) = \beta$.
Assume that $A(\beta) \neq \beta$,
in which case we will call $A$ a {\em nontrivial $\beta$-shifter}.  The curves $\beta \cup A(\beta)$
divide $\Sigma_g$ into two subsurfaces $T$ and $T'$.  Order these such that $T$ lies to the left
of $\beta$ and $T'$ lies to the right:\footnote{If $A(\beta) = \beta$, then you can slide $A(\beta)$ over $\beta$
and exchange which side is the left and the right.}\\
\Figure{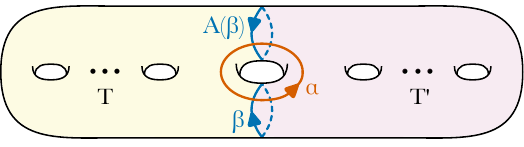}
We call $T$ the {\em left side} of $\beta \cup A(\beta)$ and $T'$ the {\em right side}.
An element $B \in \Torelli_g$ is {\em left-compatible} with $A$
if $B$ is supported on $T$ and is {\em right-compatible} with $A$ if $B$ is supported on $T'$.

\begin{remark}
If $B$ is compatible with a $\beta$-shifter $A$, then we can write $B = B_1 B_2$ with
$B_1$ left-compatible with $A$ and $B_2$ right-compatible with $A$.
\end{remark}

Similarly, let $B$ be an $\alpha$-shifter, so $\alpha$ and $B(\alpha)$ are disjoint.  Assume
that $B(\alpha) \neq \alpha$, in which case we call $B$ a {\em nontrivial $\alpha$-shifter}.  In that
case, $\alpha \cup B(\alpha)$ divides $\Sigma_g$ into two subsurfaces $T$ and $T'$, ordered such that
$T$ is to the left of $\alpha$ and $T'$ is to the right.
We call $T$ the {\em left side} of $\alpha \cup B(\alpha)$ and $T'$ the {\em right side}.
We say that $A \in \Torelli_g$ is {\em left-compatible} with $B$ if $A$ is supported on $T$ and is {\em right-compatible}
with $B$ if $A$ is supported on $T'$.

\subsection{Symplectic summands and splittings}

Recall that $a = [\alpha] \in \HH_1(\Sigma_g)$ and $b = [\beta] \in \HH_1(\Sigma_g)$, and
$H_{\Z} = \Span{a,b}^{\perp}$.
A {\em symplectic summand} of $H_{\Z}$ is a subgroup $V < H_{\Z}$ such that
$H_{\Z} = V \oplus V^{\perp}$, in which case $V \cong \Z^{2h}$ for some $h$ called
the {\em genus} of $V$.  If $V$ is a symplectic summand of $H_{\Z}$, then $V^{\perp}$ is too.
A {\em symplectic splitting} of $H_{\Z}$ is a splitting of the form $H_{\Z} = V \oplus V^{\perp}$.  
The terms $V$ and $V^{\perp}$ are ordered, so $H_{\Z} = V^{\perp} \oplus V$ is a different symplectic splitting.
We call $H_{\Z} = V \oplus V^{\perp}$ a
{\em nontrivial} symplectic splitting if $V,V^{\perp} \neq 0$.

\subsection{Induced splittings}

Let $A$ be a nontrivial $\beta$-shifter.  Let $T$ and $T'$ be the left and right sides of $\beta \cup A(\beta)$, respectively.
Let $S$ and $S'$ be the components of the complement of a regular neighborhood of
$\alpha \cup \beta \cup A(\beta)$, ordered such that $S \subset T$ and $S' \subset T'$.  This is all depicted
in the following figure:\\
\Figure{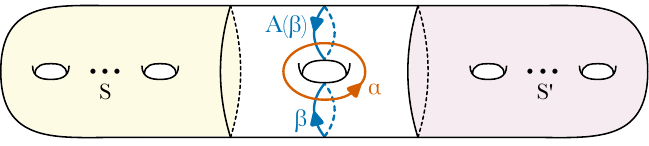}
We have $\HH_1(S) \subset H_{\Z}$ and $\HH_1(S') \subset H_{\Z}$, and
$H_{\Z} = \HH_1(S) \oplus \HH_1(S')$.
This is a nontrivial symplectic splitting of $H_{\Z}$ that we call the symplectic splitting
{\em induced} by $A$.  
We call $\HH_1(S)$ and $\HH_1(S')$ the {\em left-summand} and {\em right-summand} of $A$, respectively.

Similarly, let $B$ be a nontrivial $\alpha$-shifter.  Let $T$ and $T'$ be the left and right sides of $\alpha \cup B(\alpha)$, respectively.
Let $S$ and $S'$ be the components of the complement of a regular neighborhood of
$\beta \cup \alpha \cup B(\alpha)$, ordered such that $S \subset T$ and $S' \subset T'$.
We then have $H_{\Z} = \HH_1(S) \oplus \HH_1(S')$, and we call this the symplectic splitting {\em induced} by $B$.
We call $\HH_1(S)$ and $\HH_1(S')$ the {\em left-summand} and {\em right-summand} of $B$, respectively.

\section{Step 3.2: homological interpretation of compatible elements}
\label{section:step3compatible}

Our next goal is to give a homological interpretation of the compatible elements
$B$ in $\LLBComm{A,B}$ and $A$ in $\LLAComm{A,B}$.  This requires first proving
some relations in $\Lambda_g$.

\begin{remark}
In our proofs, we will freely use the identities between homology classes of surface
relations from \S \ref{section:surfacerelations}
\end{remark}

\subsection{Linearity}
Our first relation is:

\begin{lemma}
\label{lemma:linearity}
The following hold:
\begin{itemize}
\item Let $A$ be a $\beta$-shifter and let $B_1$ and $B_2$
be compatible with $A$.  Then $\LLBComm{A,B_1 B_2} = \LLBComm{A,B_1} + \LLBComm{A,B_2}$.
\item Let $B$ be an $\alpha$-shifter and let $A_1$ and $A_2$
be compatible with $B$.  Then $\LLAComm{A_1 A_2,B} = \LLAComm{A_1,B} + \LLAComm{A_2,B}$.
\end{itemize}
\end{lemma}
\begin{proof}~\hspace{-10pt}\footnote{This almost follows from the proof of Claim \ref{claim:identify} of Proposition~\ref{proposition:generators},
but it takes extra work to see that this does not e.g.\ give an identity of the form $\LLBComm{A,B_1 B_2} = \lambda_1 \LLBComm{A,B_1} + \lambda_2 \LLBComm{B_2}$ for some $\lambda_1,\lambda_2 \in \Q$, so we give a direct proof.  Similar comments apply to many of our other relations.  The most important relations 
that cannot be derived from the proof of Proposition~\ref{proposition:generators} even with additional work 
are those in \S \ref{section:step3identifygen} below.}
Both are proved the same way, so we will give the details for the first.
We will show that it follows from the commutator identity\footnote{Recall that our
conventions are $[x,y] = x^{-1} y^{-1} x y$ and $x^y = y^{-1} x y$.} $[A,B_1 B_2] = [A,B_2] [A,B_1]^{B_2}$.
Recall that for $i=1,2$ we have $[A,B_i] \in (\Torelli_g)_{\beta}$, and there is some $n_i \geq 1$ and a product
of commutators $\fc_i$ in $(\Torelli_g)_{\beta}$ such that $[A,B_i]^{n_i} \fc = 1$ and $\LLBComm{A,B_i}$ is the image in $\Lambda_g$ of
\[\frac{1}{n_i} \fh([A,B_i]^{n_i} \fc_i) \in \HH_2(\Torelli_g;\Q).\]
Set $n = n_1 n_2$.  The product $n \LLBComm{A,B_i}$ is the image in $\Lambda_g$ of
\[\frac{n}{n_i} \fh([A,B_i]^{n_i} \fc_i) = \fh(([A,B_i]^{n_i} \fc_i)^{n/n_i}) = \fh([A,B_i]^{n} \fc'_i),\]
where $\fc'_i$ is the product of commutators in $(\Torelli_g)_{\beta}$ obtained by commuting all the $[A,B_i]^{n_i}$ factors to the
left.  Since inner automorphisms act trivially on homology, we can
conjugate our expression for $\LLBComm{A,B_2}$ by $B_1$ and see that $n \LLBComm{A,B_2}$ is also the image
in $\Lambda_g$ of $\fh(([A,B_2]^{B_1})^n \fc''_2)$ for some product of commutators $\fc''_2$ in $(\Torelli_g)_{\beta}$.
We then have that $n (\LLBComm{A,B_1} + \LLBComm{A,B_2})$ is the image in $\Lambda_g$ of
\[\fh([A,B_1]^{n} \fc'_1 ([A,B_2]^{B_1})^n \fc''_2) = \fh(([A,B_1] [A, B_2]^{B_1})^n \fc''')
                                                    = \fh([A,B_1 B_2]^n \fc'''),\]
where $\fc'''$ is a product of commutators in $(\Torelli_g)_{\beta}$.  This maps to $n \LLBComm{A,B_1 B_2}$, as
claimed.
\end{proof}

\subsection{Vanishing}

We next prove:

\begin{lemma}
\label{lemma:strongvanish}
The following hold:
\begin{itemize}
\item Let $A$ be a $\beta$-shifter and let $B$ be compatible
with $A$.  Assume that either $A(\beta) = \beta$ or $B(\alpha) = \alpha$.  Then $\LLBComm{A,B}=0$.
\item Let $B$ be an $\alpha$-shifter and let $A$ be compatible
with $B$.  Assume that either $B(\alpha) = \alpha$ or $A(\beta) = \beta$.  Then $\LLAComm{A,B}=0$.
\end{itemize}
\end{lemma}
\begin{proof}
Both are proved the same way, so we will give the details for the first.
Recall that $[A,B] \in (\Torelli_g)_{\beta}$, and
there exists $n \geq 1$ and a product $\fc$ of commutators in $(\Torelli_g)_{\beta}$ such that
$[A,B]^n \fc = 1$ and $\LLBComm{A,B}$ is the image in $\Lambda_g$ of
\[\frac{1}{n} \fh([A,B]^n \fc) \in \HH_2(\Torelli_g;\Q).\]
Assume first that $A(\beta) = \beta$.  Then $[A,B]$ is a commutator of two elements of $(\Torelli_g)_{\beta}$, so
$\fh([A,B]^n \fc)$ is in the image of the map
\[\HH_2((\Torelli_g)_{\beta};\Q) \longrightarrow \HH_2(\Torelli_g;\Q).\]
By the definition of $\Lambda_g$, the image of $\fh([A,B]^n \fc)$ in $\Lambda_g$ therefore vanishes.

Assume next that $B(\alpha) = \alpha$.  The commutator
$[A,B]$ then fixes both $\alpha$ and $\beta$, and therefore as we noted at the end of \S \ref{section:llcomm}
we can make our choices such that $\fc$ is a product of commutators in
$(\Torelli_g)_{\alpha,\beta} \cong \Torelli_{g-1}^1$.  The fact that $B$ fixes $\alpha$ also implies that $[A,B]$ is a commutator in
$(\Torelli_g)_{\alpha}$.  We conclude that $[A,B]^n \fc$ is a product of commutators in
$(\Torelli_g)_{\alpha}$, and thus that $\fh([A,B]^n \fc)$ lies in the image of the map
\[\HH_2((\Torelli_g)_{\alpha};\Q) \longrightarrow \HH_2(\Torelli_g;\Q).\]
By the definition of $\Lambda_g$, the image of $\fh([A,B]^n \fc)$ in $\Lambda_g$ therefore vanishes.
\end{proof}

\subsection{Compatible quotient}

Let $A$ be a nontrivial $\beta$-shifter and let $X$ be either the left or the right side of $\beta \cup A(\beta)$.
Recall from \S \ref{section:curvestabilizers} that
$\Torelli_g(X)$ denotes the subgroup of $\Torelli_g$ consisting of mapping classes supported on $X$.
Each $B \in \Torelli_g(X)$ is either left- or right-compatible with $A$ depending on which side $T$ is on.
For $B \in \Torelli_g(X)$, we thus have $\LLBComm{A,B} \in \Lambda_g$.  By Lemma~\ref{lemma:linearity},
this only depends on the image of $B$ in the abelianization of $\Torelli_g(X)$.  In fact, since
$\Lambda_g$ is a $\Q$-vector space it only depends on the image of $B$ in $\HH_1(\Torelli_g(X);\Q)$.

Lemma~\ref{lemma:strongvanish} says that
$\LLBComm{A,B} = 0$ for $B \in \Torelli_g(X)$ with $B(\alpha) = \alpha$.  We thus define:
\begin{itemize}
\item For a $\beta$-shifter $A$ and $X$ either the left or the right side of $\beta \cup A(\beta)$, define
$\Omega_{\beta}(A,X)$ to be the quotient of $\HH_1(\Torelli_g(X);\Q)$ by the subspace spanned by the homology
classes of elements of $\Torelli_g(X)$ that fix $\alpha$.
\end{itemize}
For $\kappa \in \Omega_{\beta}(A,X)$, the discussion above implies that we have a well-defined
$\LLBComm{A,\kappa} \in \Lambda_g$.

Similarly, for $B$ a nontrivial $\alpha$-shifter and $X$ either the left or right side of $\alpha \cup B(\alpha)$,
define $\Omega_{\alpha}(B,X)$ to be the quotient of $\HH_1(\Torelli_g(X);\Q)$ by the subspace spanned by the homology
classes of elements of $\Torelli_g(X)$ that fix $\beta$.  For $\kappa \in \Omega_{\alpha}(B,X)$, we have
a well-defined $\LLAComm{\kappa,B} \in \Lambda_g$.

\subsection{Identification of compatible quotient}
\label{section:identifycompatible}

The above has the following description.  For a subgroup $V$ of $H_{\Z}$,
let $V_{\Q} = V \otimes \Q$ be the corresponding subspace of $H$.

\begin{lemma}
\label{lemma:identifyomega}
The following hold:
\begin{itemize}
\item Let $A$ be a nontrivial $\beta$-shifter, let $X$ be either the left or right side of $\beta \cup A(\beta)$,
and let $V$ be the summand of $A$ on the same side\footnote{In other words, $V$ is the left-summand if
$X$ is the left side and the right-summand if $X$ is the right side.} as $X$.  Assume that the genus of $X$ is at least $3$.  Then
$\Omega_{\beta}(A,X) \cong \wedge^2 V_{\Q}$.
\item Let $B$ be a nontrivial $\alpha$-shifter, let $X$ be either the left or right side of $\alpha \cup B(\alpha)$,
and let $V$ be the summand of $B$ on the same side as $X$.  Assume that the genus of $X$ is at least $3$.
Then $\Omega_{\alpha}(B,X) \cong \wedge^2 V_{\Q}$.
\end{itemize}
\end{lemma}
\begin{proof}
Both bullet points are proved the same way, so we will give details for the first.  Let $h \geq 3$
be the genus of $X$.  Using the fact that
$h \geq 3$, Putman \cite{PutmanJohnson} proved that
\begin{equation}
\label{eqn:torelliabel} 
\HH_1(\Torelli_g(X);\Q) \cong \wedge^3 \HH_1(X;\Q).
\end{equation}
This isomorphism is given by the Johnson homomorphism, which we will say more about
in \S \ref{section:johnson} below.  Let $(\Torelli_g(X))_{\alpha}$ be the stabilizer of $\alpha$ in $\Torelli_g(X)$.  By
definition, $\Omega_{\beta}(A,X)$ is the quotient of $\HH_1(\Torelli_g(X);\Q)$ by the image of
$\HH_1((\Torelli_g(X))_{\alpha};\Q)$.

Let $S$ be the component of the complement of a regular neighborhood of
$\alpha \cup \beta \cup A(\beta)$ that is contained in $X$, so
$S \cong \Sigma_h^1$:\\
\Figure{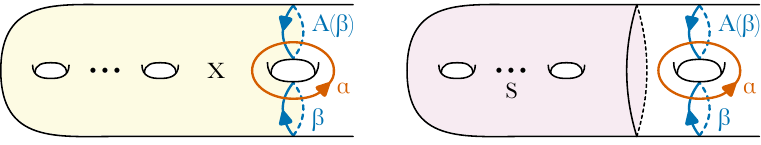}
By definition, we have $\HH_1(S) = V$.  We have
\[\HH_1(X;\Q) = \HH_1(S;\Q) \oplus \Span{[\beta]} = V_{\Q} \oplus \Span{[\beta]},\]
so
\begin{equation}
\label{eqn:decomposewedge3}
\wedge^3 \HH_1(X;\Q) \cong \left(\wedge^3 V_{\Q}\right) \oplus \left((\wedge^2 V_{\Q}) \wedge [\beta]\right).
\end{equation}
It follows from the calculations in \cite{PutmanJohnson} that under the isomorphism
\eqref{eqn:torelliabel}, the image of $\HH_1((\Torelli_g(X))_{\alpha};\Q)$
is the term $\wedge^3 V_{\Q}$ from \eqref{eqn:decomposewedge3}.  We conclude that
\[\Omega_{\beta}(A,X) \cong (\wedge^2 V_{\Q}) \wedge [\beta] \cong \wedge^2 V_{\Q}.\qedhere\]
\end{proof}

\subsection{Notation}

In light of Lemma~\ref{lemma:identifyomega}, if $A$ is a nontrivial $\beta$-shifter, $V$ is either
the left- or right-summand of $A$, and the genus of $V$ is at least $3$, then for
$\kappa \in \wedge^2 V_{\Q}$ we have a well-defined $\LLBComm{A,\kappa} \in \Lambda_g$.
Similarly, if $B$ is a nontrivial $\alpha$-shifter, $V$ is either the left- or right-summand
of $B$, and the genus of $V$ is at least $3$, then for $\kappa \in \wedge^2 V_{\Q}$ we have
a well-defined $\LLAComm{\kappa,B} \in \Lambda_g$.  These elements satisfy the following
linearity relations:

\begin{lemma}
\label{lemma:linearityrel1}
The following hold:
\begin{itemize}
\item Let $A$ be a nontrivial $\beta$-shifter and let $V$ be either the left- or right-summand of $A$.
Assume that the genus of $V$ is at least $3$.  Then for $\kappa_1,\kappa_2 \in \wedge^2 V_{\Q}$ and
$\lambda_1,\lambda_2 \in \Q$ we have
\[\LLBComm{A,\lambda_1 \kappa_1 + \lambda_2 \kappa_2} = \lambda_1 \LLBComm{A,\kappa_1} + \lambda_2 \LLBComm{A,\kappa_2}.\]
\item Let $B$ be a nontrivial $\alpha$-shifter and let $V$ be either the left- or right-summand of $B$.
Assume that the genus of $V$ is at least $3$.  Then for $\kappa_1,\kappa_2 \in \wedge^2 V_{\Q}$ and
$\lambda_1,\lambda_2 \in \Q$ we have
\[\LLAComm{\lambda_1 \kappa_1 + \lambda_2 \kappa_2,B} = \lambda_1 \LLAComm{\kappa_1,B} + \lambda_2 \LLAComm{\kappa_2,B}.\]
\end{itemize}
\end{lemma}
\begin{proof}
Immediate from the linearity relations in Lemma~\ref{lemma:linearity} along with the definitions.
\end{proof}

\subsection{Johnson homomorphism}
\label{section:johnson}

For later use, we give a computation involving the Johnson homomorphism used in the proof
of Lemma~\ref{lemma:identifyomega}.  To set it up, let $X$ be a subsurface
of $\Sigma_g$ with the following properties:
\begin{itemize}
\item $X$ is a connected subsurface with two boundary components; and
\item the complement $\Sigma_g \setminus \Int(X)$ is connected and has positive genus.
\end{itemize}
In this context, the Johnson homomorphism is a homomorphism
\[\tau\colon \Torelli_g(X) \longrightarrow \wedge^3 \HH_1(X;\Q)\]
that was defined by Putman \cite{PutmanJohnson} based on work of
Johnson \cite{JohnsonHomo} for closed surfaces.

Consider a bounding pair $T_{\delta} T_{\lambda}^{-1}$ in $\Torelli_g(X)$, so $\delta$ and $\lambda$ are disjoint
nonseparating curves on $\Sigma_g$ such that $\delta,\lambda \subset X$ and such that
$\delta \cup \lambda$ separates $\Sigma_g$.  Let $X'$ be the component of $X$ cut open along
$\delta \cup \lambda$ that is disjoint from $\partial X$.  We have $X' \cong \Sigma_k^2$ for some
$k \geq 0$.  Let $Y$ be a subsurface of $X'$ with $Y \cong \Sigma_k^1$:\\
\Figure{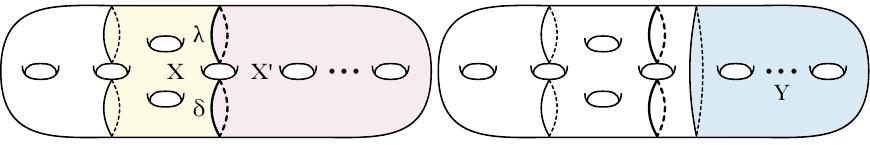}
Let $\{z_1,w_1,\ldots,z_k,w_k\}$ be a symplectic basis for $\HH_1(Y) \cong \Z^{2k}$.  Orient
$\delta$ arbitrarily.  Then
\[\tau(T_{\delta} T_{\lambda}^{-1}) = \pm (z_1 \wedge w_1 + \cdots + z_k \wedge w_k) \wedge [\delta] \in \wedge^3 \HH_1(X;\Q),\]
where the sign is $+1$ (resp.\ $-1$) if $Y$ is to the left (resp.\ right) of $\delta$.  It
is easy to see that this does not depend on our choices (the orientation of $\delta$ and the subsurface $Y$).

\section{Step 3.3: homological interpretation of shifters}
\label{section:step3shifters}

We now give a homological interpretation of the shifters
$A$ in $\LLBComm{A,\kappa}$ and $B$ in $\LLAComm{\kappa,B}$.

\subsection{Dependence on splitting}
\label{section:summanddependence}

Our main result is:

\begin{lemma}
\label{lemma:splittingdependence}
The following hold:
\begin{itemize}
\item Let $A$ and $A'$ be nontrivial $\beta$-shifters inducing the same symplectic splitting of $H_{\Z}$.
Let $W$ be either the left- or the right-summand of $A$ and $A'$.  Assume that $W$ has genus at least $3$, and
let $\kappa \in \wedge^2 W_{\Q}$.  Then $\LLBComm{A,\kappa} = \LLBComm{A',\kappa}$.
\item Let $B$ and $B'$ be nontrivial $\alpha$-shifters inducing the same symplectic splitting of $H_{\Z}$.
Let $W$ be either the left- or the right-summand of $B$ and $B'$.  Assume that $W$ has genus at least $3$, and
let $\kappa \in \wedge^2 W_{\Q}$.  Then $\LLBComm{\kappa,B} = \LLBComm{\kappa,B'}$.
\end{itemize}
\end{lemma}
\begin{proof}
Both are proved the same way, so we will give the details for the first.
It is enough to prove this for $\kappa$ the image of an element $B \in \Torelli_g$ that is either
left- or right-compatible with $A$.  The proof has two steps:

\begin{step}{1}
We have $\LLBComm{A,\kappa} = \LLBComm{A',\kappa}$ if $A(\beta) = A'(\beta)$.
\end{step}

Since $A(\beta) = A'(\beta)$, the element $B$ is also left- or right-compatible with $A'$ and $\LLBComm{A',\kappa} = \LLBComm{A',B}$.
Set $f = (A')^{-1} A$, so $f$ fixes both $\alpha$ and $\beta$.  The element $f$ is a trivial $\beta$-shifter
such that $B$ is left- or right-compatible with $f$, so Lemma~\ref{lemma:strongvanish} implies that $\LLBComm{f,B}=0$.
Consider the commutator identity
\[[A,B] = [A' f,B] = [A',B]^{f} [f,B].\]
An argument like in the proof of Lemma~\ref{lemma:linearity} shows that this
leads to the formula
\[\LLBComm{A,\kappa} = \LLBComm{A,B} = \LLBComm{A',B} + \LLBComm{f,B} = \LLBComm{A',B} = \LLBComm{A',\kappa}.\]

\begin{step}{2}
We have $\LLBComm{A,\kappa} = \LLBComm{A',\kappa}$ in general.
\end{step}

We claim that there exists some $f \in (\Torelli_g)_{\alpha,\beta}$ such that $f(A'(\beta)) = A(\beta)$.
Indeed, since $A$ and $A'$ induce the same symplectic splitting of $H_{\Z}$, we can apply \cite[Lemma 7]{JohnsonConjugacy} to find
$f' \in \Torelli_g$ such that
\[f'(\beta) = \beta \quad \text{and} \quad f'(A'(\beta)) = A(\beta).\]
Both $\{\alpha,\beta,A(\beta)\}$ and
\[\{f'(\alpha),f'(\beta),f'(A'(\beta))\} = \{f'(\alpha),\beta,A(\beta)\}\]
are mixed simplices of $\cC_{ab}(\Sigma_g)$, so by Lemma~\ref{lemma:mixpure} we can find
some $h \in \Torelli_g$ with
\[h(f'(\alpha)) = \alpha \quad \text{and} \quad h(\beta) = \beta \quad \text{and} \quad h(A(\beta)) = A(\beta).\]
The desired $f$ is then $f = h f'$.

We have
\[A^f(\beta) = f^{-1} A f(\beta) = f^{-1} A(\beta) = A'(\beta).\]
This implies that $A^f$ is a nontrivial $\beta$-shifter.  The element $B^f$ is left- or right-compatible with $A^f$, so by
Step 1 and the fact that inner automorphisms act trivially on homology we have
\[\LLBComm{A,\kappa} = \LLBComm{A,B} = \LLBComm{A^f,B^f} = \LLBComm{A',B^f} = \LLBComm{A',\kappa}.\qedhere\]
\end{proof}

\subsection{Realizing symplectic splittings}

The following says that all nontrivial symplectic splittings of $H_{\Z}$ can
be induced by a nontrivial $\beta$- and $\alpha$-shifters.

\begin{lemma}
\label{lemma:realizesplitting}
Let $H_{\Z} = V \oplus V^{\perp}$ be a nontrivial symplectic splitting of $H_{\Z}$.  Then:
\begin{itemize}
\item there exists a nontrivial $\beta$-shifter $A$ inducing $H_{\Z} = V \oplus W$; and
\item there exists a nontrivial $\alpha$-shifter $B$ inducing $H_{\Z} = V \oplus W$.
\end{itemize}
\end{lemma}
\begin{proof}
Both are proved the same way, so we will give the details for the first.
Let $X \cong \Sigma_{g-1}^1$ be the complement of a regular neighborhood of $\alpha \cup \beta$.  In \cite[Lemma 9]{JohnsonConjugacy}, Johnson
proved that there exists a subsurface $T \cong \Sigma_k^1$ of $X$ with $\HH_1(T) = V$.
We can then find a $b$-curve $\beta'$ as follows:\\
\Figure{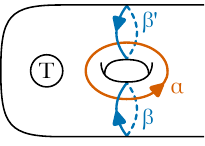}
Lemma~\ref{lemma:transitive} gives an $A \in \Torelli_g$ with $A(\beta) = \beta'$
and $A(\alpha) = \alpha$, i.e., a $\beta$-shifter whose left-summand is $V$.  The right-summand
of $A$ is then $V^{\perp}$, so $A$ induces the symplectic splitting $H_{\Z} = V \oplus V^{\perp}$.
\end{proof}

\subsection{Notation}

In light of Lemmas \ref{lemma:splittingdependence} and \ref{lemma:realizesplitting}, we introduce the following
notation.  Let $H_{\Z} = V \oplus V^{\perp}$ be a nontrivial symplectic splitting.  Let $W$ be either $V$ or
$V^{\perp}$, and assume that $W$ has genus at least $3$.  Choose $\kappa \in \wedge^2 W_{\Q}$.  Then:
\begin{itemize}
\item Let $A$ be a nontrivial $\beta$-shifter inducing the symplectic splitting $H_{\Z} = V \oplus V^{\perp}$.
Define $\LLBComm{V,\kappa} = \LLBComm{A,\kappa}$.
\item Let $B$ be a nontrivial $\alpha$-shifter inducing the symplectic splitting $H_{\Z} = V \oplus V^{\perp}$.
Define $\LLAComm{\kappa,V} = \LLAComm{\kappa,B}$.
\end{itemize}
Here $\LLBComm{A,\kappa}$ and $\LLAComm{\kappa,B}$ are as defined at the end of \S \ref{section:identifycompatible}.
These elements satisfy the following linearity relations:

\begin{lemma}
\label{lemma:linearityrel2}
Let $H_{\Z} = V \oplus V^{\perp}$ be a nontrivial symplectic splitting.  Let $W$ be either $V$ or
$V^{\perp}$, and assume that $W$ has genus at least $3$.  Then for $\kappa_1,\kappa_2 \in \wedge^2 W_{\Q}$
and $\lambda_1,\lambda_2 \in \Q$ we have
\begin{align*}
\LLBComm{V,\lambda_1 \kappa_1 + \lambda_2 \kappa_2} &= \lambda_1 \LLBComm{V,\kappa_1} + \lambda_2 \LLBComm{V,\kappa_2}, \\
\LLAComm{\lambda_1 \kappa_1 + \lambda_2 \kappa_2,V} &= \lambda_1 \LLAComm{\kappa_1,V} + \lambda_2 \LLAComm{\kappa_2,V}.
\end{align*}
\end{lemma}
\begin{proof}
Immediate from Lemma~\ref{lemma:linearityrel1}.
\end{proof}

\subsection{Orthogonal complement}

The following shows how changing $V$ to $V^{\perp}$ in the notation $\LLBComm{V,\kappa}$ and
$\LLAComm{\kappa,V}$ affects our generators:

\begin{lemma}
\label{lemma:orthogonalcomplement}
Let $H_{\Z} = V \oplus V^{\perp}$ be a nontrivial symplectic splitting.  Let $W$ be either
$V$ or $V^{\perp}$, and assume that $W$ has genus at least $3$.  Let $\kappa \in \wedge^2 W_{\Q}$.
Then $\LLBComm{V^{\perp},\kappa} = -\LLBComm{V,\kappa}$
and $\LLAComm{\kappa,V^{\perp}} = -\LLAComm{\kappa,V}$.
\end{lemma}
\begin{proof}
Both are proved the same way, so we will give the details for the first.
Let $A$ be a nontrivial $\beta$-shifter inducing the symplectic splitting $H_{\Z} = V \oplus V^{\perp}$.
It is enough to prove the lemma for $\kappa$ the image of some $B$ that is compatible with $A$, so
\begin{equation}
\label{eqn:orthogonalcomplement.1}
\LLBComm{V,\kappa} = \LLBComm{A,B}.
\end{equation}
We must give a similar formula for $\LLBComm{V^{\perp},\kappa}$.

Since $A$ is a $\beta$-shifter, $\beta$ is disjoint from $A(\beta)$.  Applying
$A^{-1}$ to this, we see that $A^{-1}(\beta)$ is disjoint from $\beta$, so $A^{-1}$ is a $\beta$-shifter.
Let $T$ and $T'$ be the left- and right-sides of $\beta \cup A(\beta)$:\\
\Figure{LeftRightShifter}
Since $T$ lies to the right of $A(\beta)$, it follows that
$A^{-1}(T)$ lies to
the right of $\beta$.  Similarly, $A^{-1}(T')$ lies to the left of $\beta$.  We have
$V \subset \HH_1(T)$ and $V^{\perp} \subset \HH_1(T')$, so $A^{-1}(V) \subset \HH_1(A^{-1}(T))$
and $A^{-1}(V^{\perp}) \subset \HH_1(A^{-1}(T'))$.  Since
$A \in \Torelli_g$ acts trivially on $H_{\Z}$, we deduce that $A^{-1}$ induces the symplectic
splitting $H_{\Z} = V^{\perp} \oplus V$.

The fact that $B$ is compatible with $A$ means that $B$ fixes $\beta$ and $A(\beta)$.  This
implies that $B^A = A^{-1} B A$ fixes $A^{-1}(\beta)$ and $\beta$, so $B^A$ is compatible with
$A^{-1}$.  Since $A \in \Torelli_g$ fixes $\kappa \in \wedge^2 W_{\Q}$, we have
\begin{equation}
\label{eqn:orthogonalcomplement.2}
\LLBComm{A^{-1},B^A} = \LLBComm{V^{\perp},A^{-1}(\kappa)} = \LLBComm{V^{\perp},\kappa}.
\end{equation}
In light of \eqref{eqn:orthogonalcomplement.1} and \eqref{eqn:orthogonalcomplement.2},
we must prove that $\LLBComm{A^{-1},B^A} = -\LLBComm{A,B}$.
This follows from the commutator identity $[A^{-1},B^A] = [A,B]^{-1}$ just like in the
proof of Lemma~\ref{lemma:linearity}.
\end{proof}

The following variant on Lemma~\ref{lemma:orthogonalcomplement} will also be useful:

\begin{lemma}
\label{lemma:linearity2}
The following hold:
\begin{itemize}
\item Let $A$ be a $\beta$-shifter and let $B$ be compatible with $A$.
Then $A^{-1}$ is a $\beta$-shifter, $B^A$ is compatible with $A$, and
$\LLBComm{A^{-1},B^A} = -\LLBComm{A,B}$.
\item Let $B$ be an $\alpha$-shifter and let $A$ be compatible with $B$.  
Then $B^{-1}$ is an $\alpha$-shifter, $A^B$ is compatible with $B$, and
$\LLAComm{A^B,B^{-1}} = -\LLAComm{A,B}$.
\end{itemize}
\end{lemma} 
\begin{proof}
The proof is similar to that of Lemma~\ref{lemma:orthogonalcomplement}, so we omit it.
\end{proof}

\section{Step 3.4: refined generating set}
\label{section:step3genset}

We now prove that only some of the $\LLBComm{V,\kappa}$ and $\LLAComm{\kappa,V}$
are needed to generate $\Lambda_g$.

\subsection{Main result}
Let $V$ be a genus-$1$ symplectic summand of $H_{\Z}$.  Since $H_{\Z}$
is the orthogonal complement in $\HH_1(\Sigma_g)$ of $\Span{a,b}$, it follows
that $H_{\Z}$ has genus $g-1$.  This implies that $V^{\perp}$ has genus
$g-2$.  Since $g \geq 5$ (see Assumption \ref{assumption:step3genus}),
we deduce that $V^{\perp}$ has genus at least $3$.  For $\kappa \in \wedge^2 V_{\Q}^{\perp}$,
it follows that $\LLBComm{V,\kappa}$ and $\LLAComm{\kappa,V}$ are defined.
These elements generate $\Lambda_g$:

\begin{lemma}
\label{lemma:refinedgens}
The vector space $\Lambda_g$ is spanned by elements
of the form $\LLBComm{V,\kappa}$ and $\LLAComm{\kappa,V}$ as $V$ ranges
over genus-$1$ symplectic summands of $H_{\Z}$ and $\kappa$ ranges
over elements of $\wedge^2 V_{\Q}^{\perp}$. 
\end{lemma}
\begin{proof}
Lemma~\ref{lemma:strongvanish} implies that $\LLBComm{A,B} = 0$ if $A$ is a trivial $\beta$-shifter and
that $\LLAComm{A,B} = 0$ if $B$ is a trivial $\alpha$-shifter.  This allows us to restrict attention
to nontrivial $\beta$- and $\alpha$-shifters.  In light of this, $\Lambda_g$ is spanned by the following elements:
\begin{itemize}
\item $\LLBComm{A,B}$ with $A$ a nontrivial $\beta$-shifter and $B$ compatible with $A$; and
\item $\LLAComm{A,B}$ with $B$ a nontrivial $\alpha$-shifter and $A$ compatible with $B$.
\end{itemize}
We must prove that each of these is a linear combination of our proposed generators.  We
will do this for $\LLBComm{A,B}$ with $A$ a nontrivial $\beta$-shifter and $B$ compatible with $A$.

For a nontrivial $\beta$-shifter $A'$, say that the {\em left-genus} of $A'$ is the genus
of the left side of $\beta \cup A'(\beta)$.  It is enough to prove that each $\LLBComm{A,B}$ is a linear
combination of elements of the following form:
\begin{itemize}
\item $\LLBComm{A',B'}$ with $A$ a $\beta$-shifter of left-genus $1$ and $B'$ right-compatible with $A'$.
\end{itemize}
We do this in two steps.  The first step ensures that the compatible element is right-compatible, and the
second that the $\beta$-shifter has left-genus $1$.

\begin{claim}{1}
Let $A$ be a nontrivial $\beta$-shifter and let $B$ be compatible with $A$.  Then $\LLBComm{A,B}$ is a linear
combination of elements of the form $\LLBComm{A',B'}$ with $A'$ a nontrivial $\beta$-shifter and $B'$ right-compatible with $A'$.
\end{claim}

Let $T$ and $T'$ be the left and right sides of $\beta \cup A(\beta)$, respectively.
Since $B$ fixes $\beta \cup A(\beta)$ and lies in $\Torelli_g$, we can write $B = B_1 B_2$ with
$B_1$ supported on $T$ and $B_2$ supported on $T'$.
By Lemmas \ref{lemma:linearity} and \ref{lemma:linearity2}, we have
\[\LLBComm{A,B} = \LLBComm{A,B_1} + \LLBComm{A,B_2} = -\LLBComm{A^{-1},B_1^{A}} + \LLBComm{A,B_2}.\]
Since $B_2$ is right-compatible with $A$, this reduces us to showing that $B_1^A$ is right-compatible with
$A^{-1}$.  Since $T$ is to the left of $\beta$, it is to the right of $A(\beta)$.  This implies that
$A^{-1}(T)$ is to the right of $\beta$.  Since $B_1^A$ is supported on $A^{-1}(T)$, this implies
that $B_1^A$ is right-compatible with $A^{-1}$, as desired.

\begin{claim}{2}
Let $A$ be a nontrivial $\beta$-shifter and let $B$ be right-compatible with $A$.  Then $\LLBComm{A,B}$ is a linear
combination of elements of the form $\LLBComm{A',B'}$ with $A'$ a $\beta$-shifter of left-genus $1$ and $B'$ right-compatible with $A'$.
\end{claim}

Let $T$ and $T'$ be the left and right sides of $\beta \cup A(\beta)$, 
so $B$ is supported on $T'$.  Let
$A(\beta) = \beta_0, \beta_1,\ldots, \beta_{h} = \beta$
be a sequence of disjoint $b$-curves in $T$ each of which intersect $\alpha$ once such that $\beta_{i-1} \cup \beta_i$ bounds
a genus-$1$ subsurface of $T$ for $1 \leq i \leq h$:\\
\Figure{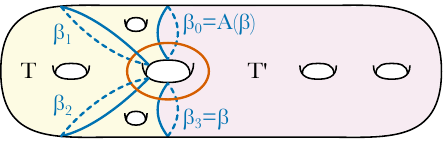}
By Lemma~\ref{lemma:transitive} the group $\Torelli_g$ acts transitively on mixed $1$-simplices of $\cC_{ab}(\Sigma_g)$, so we
can find $A_2,\ldots,A_{h} \in \Torelli_g$ such that $A_i(\alpha) = \alpha$ and
$A_i(\beta_{i}) = \beta_{i-1}$ for $2 \leq i \leq h$.  Pick $A_1 \in \Torelli_g$ such that
$A = A_1 \cdots A_h$, so $A_1(\alpha) = \alpha$ as well.
We then have
\begin{equation}
\label{eqn:pusharound}
\text{$A_i A_{i+1} \cdots A_{h}(\beta) = \beta_{i-1}$ for $1 \leq i \leq h$}.
\end{equation}
We have a commutator identity
\begin{align*}
[A_1 \cdots A_{h},B] &= [A_1, B]^{A_2 \cdots A_{h}} [A_2,B]^{A_3 \cdots A_h} \cdots [A_h,B] \stepcounter{equation}\tag{\theequation}\label{eqn:linearcommutator}\\
&= [A_1^{A_2 \cdots A_h},B^{A_2 \cdots A_h}] [A_2^{A_3 \cdots A_h},B^{A_3 \cdots A_h}] \cdots [A_h,B].
\end{align*}
We then have generators $\LLBComm{A_i^{A_{i+1} \cdots A_h},B^{A_{i+1} \cdots A_h}}$ associated
to the terms in this identity.  Indeed:
\begin{itemize}
\item $A_i^{A_{i+1} \cdots A_h}$ is a $\beta$-shifter since
\begin{equation}
\label{eqn:abetaexpression}
A_i^{A_{i+1} \cdots A_h}(\beta) = (A_{i+1} \cdots A_h)^{-1} (A_i \cdots A_h)(\beta) = (A_{i+1} \cdots A_h)^{-1}(\beta_{i-1})
\end{equation}
is disjoint from
\begin{equation}
\label{eqn:betaexpression}
\beta = (A_{i+1} \cdots A_h)^{-1}(\beta_i).
\end{equation}
Both \eqref{eqn:abetaexpression} and \eqref{eqn:betaexpression} use \eqref{eqn:pusharound}.  Since $\beta_{i-1} \cup \beta_i$ bounds
a genus-$1$ subsurface of $T$, it follows that $A_i^{A_{i+1} \cdots A_h}$ has left-genus $1$.
\item The element $B$ fixes each $\beta_j$, so $B^{A_{i+1} \cdots A_h}$ is compatible with $A_i^{A_{i+1} \cdots A_h}$ since it fixes
$\beta$ (cf.\ \eqref{eqn:betaexpression}) and $A_i^{A_{i+1} \cdots A_h}(\beta)$ (cf.\ \eqref{eqn:abetaexpression}).
In fact, by construction $B$ is right-compatible with $B^{A_{i+1} \cdots A_h}$.
\end{itemize}
Using \eqref{eqn:linearcommutator}, an argument similar to the one used in the
proof of Lemma~\ref{lemma:linearity} shows that
\begin{align*}
\LLBComm{A,B} &= \LLBComm{A_1 A_2 \cdots A_h,B} \\
              &= \LLBComm{A_1^{A_2 \cdots A_h},B^{A_2 \cdots A_h}} + \LLBComm{A_2^{A_3 \cdots A_h},B^{A_3 \cdots A_h}} + \cdots + \LLBComm{A_h,B}.
\end{align*}
The right hand side is a sum of terms in the desired generating set, as desired.
\end{proof}

\section{Step 3.5: identifying generators}
\label{section:step3identifygen}

The generating set for $\Lambda_g$ from Lemma~\ref{lemma:refinedgens} has some redundancies.

\subsection{Intersection form}
\label{section:intersectionform}

Identifying these redundancies requires some notation.  Let $W$ be a symplectic
summand of $H_{\Z}$.  The algebraic intersection form on $W$ identifies $W$ with its
dual.  This allows us to identify alternating bilinear forms on $W$ with
elements of $\wedge^2 W \subset \wedge^2 H_{\Z}$.  In particular, the algebraic intersection form
on $W$ is an element $\omega_W$ of $\wedge^2 H_{\Z}$.  If $\{a_1,b_1,\ldots,a_h,b_h\}$
is a symplectic basis for $W$, then $\omega_W = a_1 \wedge b_1 + \cdots + a_h \wedge b_h$.

\subsection{Redundancy}

With this notation, we have the following two lemmas:

\begin{lemma}
\label{lemma:identifygenerators1}
Let $V$ and $W$ be orthogonal genus-$1$ symplectic summands of $H_{\Z}$.
Then $\LLBComm{V,\omega_{W}} = -\LLAComm{\omega_V,W}$.
\end{lemma}
\begin{proof}
Since $V$ and $W$ are orthogonal genus-$1$ symplectic summands of $H_{\Z}$, there are 
disjoint subsurfaces $X \cong \Sigma_1^1$ and $Y \cong \Sigma_{1}^1$ of $\Sigma_g$ 
that are disjoint from $\alpha \cup \beta$ such
that $\HH_1(X) = V$ and $\HH_1(Y)=W$.  Pick curves $\alpha'$ and $\beta'$
as follows (cf.\ Example \ref{example:commex}):\\
\Figure{ExampleGenerator}
Set $A = T_{\alpha} T_{\alpha'}^{-1}$ and $B = T_{\beta'} T_{\beta}^{-1}$.  Then $A$ is a $\beta$-shifter
and $B$ is compatible with $A$, and also $B$ is an $\alpha$-shifter and $A$ is compatible with $B$:\\
\Figure{ExampleGeneratorCheck}
When you cut $\Sigma_g$ along $\beta \cup A(\beta)$ the subsurface
$X$ is to the left of $\beta$, so $A$ induces the symplectic splitting 
\[H_{\Z} = \HH_1(X) \oplus \HH_1(X)^{\perp} = V \oplus V^{\perp}.\]
Similarly, when you cut
$\Sigma_g$ along $\alpha \cup B(\alpha)$ the subsurface $Y$ is to the left of $\alpha$,
so $B$ induces the symplectic splitting $H_{\Z} = W \oplus W^{\perp}$.

Examining our proof of Lemma~\ref{lemma:identifyomega} and using the computation
of the Johnson homomorphism on a bounding pair in \S \ref{section:johnson}, we see that the fact that $Y$ is to the
left of $\beta'$ implies that the image
$B = T_{\beta'} T_{\beta}^{-1}$ in $\wedge^2 V^{\perp}_{\Q}$
is $\omega_W$.  Similarly, since $X$ is to the right of $\alpha$
the image of $A=T_{\alpha} T_{\alpha'}^{-1}$ in $\wedge^2 W^{\perp}_{\Q}$
is $-\omega_V$.  Together with the previous paragraph, this implies that (cf.\ Remark \ref{remark:distinct})
\[\LLBComm{V,\omega_{W}} = \LLBComm{A,B} = \LLAComm{A,B} = \LLAComm{-\omega_V,W} = -\LLAComm{\omega_V,W}.\qedhere\]
\end{proof}

\begin{lemma}
\label{lemma:identifygenerators2}
Let $V$ be a genus-$1$ symplectic summand of $H_{\Z}$.
Then $\LLBComm{V,\omega_{V^{\perp}}} = -\LLAComm{\omega_{V^{\perp}},V}$.
\end{lemma}
\begin{proof}
The proof is similar to that of Lemma~\ref{lemma:identifygenerators1}, but with a twist at the end.  Recall
that $H_{\Z}$ has genus $g-1$.
Pick disjoint subsurfaces $X \cong \Sigma_1^1$ and $Y \cong \Sigma_{g-2}^1$ of $\Sigma_g$ that are disjoint from $\alpha \cup \beta$ such
that $\HH_1(X) = V$ and $\HH_1(Y)=V^{\perp}$.  Choose curves $\alpha'$ and $\beta'$
as follows:\\
\Figure{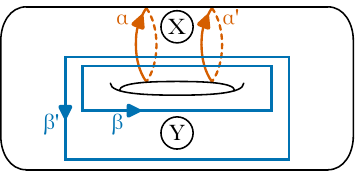}
Note that unlike in the proof of Lemma~\ref{lemma:identifygenerators1}, this depicts the whole surface and not just
part of it.

Set $A = T_{\alpha} T_{\alpha'}^{-1}$ and $B = T_{\beta'} T_{\beta}^{-1}$.  Then $A$ is a $\beta$-shifter
and $B$ is compatible with $A$, and also $B$ is an $\alpha$-shifter and $A$ is compatible with $B$:\\
\Figure{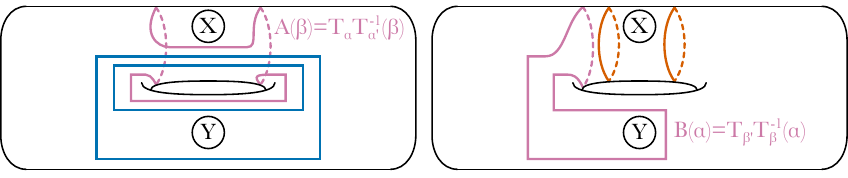}
When you cut $\Sigma_g$ along $\beta \cup A(\beta)$ the subsurface
$X$ is to the left of $\beta$, so $A$ induces the symplectic splitting
\[H_{\Z} = \HH_1(X) \oplus \HH_1(X)^{\perp} = V \oplus V^{\perp}.\]
Similarly, when you cut $\Sigma_g$ along $\alpha \cup B(\alpha)$ the subsurface $Y$ is to the left of $\alpha$,
so $B$ induces the symplectic splitting $H_{\Z} = V^{\perp} \oplus V$.

Examining our proof of Lemma~\ref{lemma:identifyomega} and using the computation
of the Johnson homomorphism on a bounding
pair in \S \ref{section:johnson}, we see that the fact that $Y$ is to the
left of $\beta'$ implies that the image
$B = T_{\beta'} T_{\beta}^{-1}$ in $\wedge^2 V^{\perp}$
is $\omega_{V^{\perp}}$.  This implies that
\begin{equation}
\label{eqn:identifygenerators2.1}
\LLBComm{V,\omega_{V^{\perp}}} = \LLBComm{A,B}.
\end{equation}
After reading the proof of Lemma~\ref{lemma:identifygenerators1}, you might think that $\LLAComm{A,B}$ equals $\LLAComm{-\omega_V,V^{\perp}}$.
However, $\LLAComm{-\omega_V,V^{\perp}}$ is not defined since
$V$ has genus $1$ and $1 < 3$ (cf.\ Lemma~\ref{lemma:identifyomega}).  To fix this, observe that
$\alpha'$ is homotopic to $B(\alpha)$:\\
\Figure{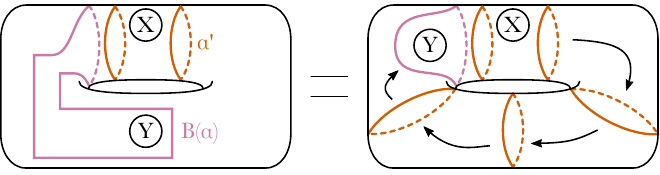}
This implies that $A=T_{\alpha} T_{\alpha'}^{-1}$ can be not only be realized by a mapping class
supported on the right of $\alpha \cup B(\alpha)$, but also by a mapping class supported on the left.
Since $Y$ is to the left of $\alpha$, the image of $A=T_{\alpha} T_{\alpha'}^{-1}$ in $\wedge^2 V^{\perp}$
is $\omega_{V^{\perp}}$.  It follows that
\begin{equation}
\label{eqn:identifygenerators2.2}
\LLBComm{\omega_{V^{\perp}},V^{\perp}} = \LLBComm{A,B}.
\end{equation}
Combining \eqref{eqn:identifygenerators2.1} and \eqref{eqn:identifygenerators2.2} with
Lemma~\ref{lemma:linearity2}, we see that
\[\LLBComm{V,\omega_{V^{\perp}}} = \LLBComm{A,B} = \LLAComm{A,B} = \LLAComm{\omega_{V^{\perp}},V^{\perp}} = -\LLAComm{\omega_{V^{\perp}},V}.\qedhere\]
\end{proof}

\section{Step 3.7: proof of Theorem \texorpdfstring{\ref{maintheorem:cokernel}}{B'}}
\label{section:step3proof}

Theorem~\ref{maintheorem:cokernel} asserts that $\Lambda_g$ is a finite-dimensional algebraic representation of $\Sp(H_{\Z})$.
As we will show in this final section, the generators and relations for $\Lambda_g$ we constructed in the last few sections are enough
to force this to hold.

\subsection{Presentation theorem}

The key is a recent theorem of the authors giving a presentation for what we call the symmetric
kernel representation.  Make the following definition:

\begin{definition}
Define $\fK(H)$ to be the vector space with the following presentation:
\begin{itemize}
\item {\bf Generators}.  For all genus-$1$ symplectic summands $V$ of $H_{\Z}$ and all
$\kappa \in \wedge^2 V_{\Q}^{\perp}$, generators $\Pres{V,\kappa}$ and $\Pres{\kappa,V}$.
\item {\bf Relations}.  The following families of relations:
\begin{itemize}
\item For all genus-$1$ symplectic summands $V$ of $H_{\Z}$ and all $\kappa_1,\kappa_2 \in \wedge^2 V_{\Q}^{\perp}$
and all $\lambda_1,\lambda_2 \in \Q$, the linearity relations
\begin{align*}
\Pres{V,\lambda_1 \kappa_1 + \lambda_2 \kappa_2} &= \lambda_1 \Pres{V,\kappa_1} + \lambda_2 \Pres{V,\kappa_2} \quad \text{and} \\ 
\Pres{\lambda_1 \kappa_1 + \lambda_2 \kappa_2,V} &= \lambda_1 \Pres{\kappa_1,V} + \lambda_2 \Pres{\kappa_2,V}.
\end{align*}
\item For all orthogonal genus-$1$ symplectic summands $V$ and $W$ of $H_{\Z}$, the relation
\[\Pres{V,\omega_W} = \Pres{\omega_V,W}.\]
\item For all genus-$1$ symplectic summands $V$ of $H_{\Z}$, the relation
\[\Pres{V,\omega_{V^{\perp}}} = \Pres{\omega_{V^{\perp}},V}.\qedhere\]
\end{itemize}
\end{itemize}
\end{definition}

The actions of $\Sp(H_{\Z})$ on $H_{\Z}$ and $H$ induce an action of $\Sp(H_{\Z})$
on $\fK(H)$.  We have:\footnote{This theorem requires that $H$ has genus at least $4$, which
follows from our assumption that $g \geq 5$ (Assumption \ref{assumption:step3genus}) since $H$
is the orthogonal complement in $\HH_1(\Sigma_g;\Q)$ of $\Span{a,b}$.}

\begin{theorem}[{\cite[Theorem A.6]{MinahanPutmanRepPresentations}}]
\label{theorem:presentation}
The representation $\fK(H)$ is a finite-dimensional algebraic representation of $\Sp(H_{\Z})$.
\end{theorem}

\begin{remark}
In fact, what \cite{MinahanPutmanRepPresentations} proves is that $\fK(H)$ is a subquotient
of $(\wedge^2 H)^{\otimes 2}$.  More precisely, let $(\wedge^2 H)/\Q$ be the quotient
of $\wedge^2 H$ by the one-dimensional trivial representation spanned by
the element $\omega$ representing the algebraic intersection pairing.  For
$\kappa \in \wedge^2 H$, let $\okappa$ be its image in $(\wedge^2 H)/\Q$.  Then
\cite{MinahanPutmanRepPresentations} proves that the map $\Phi\colon \fK(H) \rightarrow ((\wedge^2 H)/\Q)^{\otimes 2}$
defined by
\[\Phi(\Pres{V,\kappa}) = \oomega_V \otimes \okappa \quad \text{and} \quad \Phi(\Pres{\kappa,V}) = \okappa \otimes \oomega_V\]
is a well-defined isomorphism onto its image $\cK(H)$, which \cite{MinahanPutmanRepPresentations} calls
the {\em symmetric kernel}.  It is the kernel of a
contraction $((\wedge^2 H)/\Q)^{\otimes 2} \longrightarrow \Sym^2(H)$.  See \cite{MinahanPutmanRepPresentations}
for more details.
\end{remark}

\subsection{The proof}

We close the paper by proving Theorem~\ref{maintheorem:cokernel}.

\newtheorem*{maintheorem:cokernel}{Theorem~\ref{maintheorem:cokernel}}
\begin{maintheorem:cokernel}
The vector space $\Lambda_g$ is a finite-dimensional algebraic representation of $\Sp(H_{\Z})$.
\end{maintheorem:cokernel}
\begin{proof} 
Theorem~\ref{theorem:presentation} says that $\fK(H)$ is a finite-dimensional algebraic representation of
$\Sp(H_{\Z})$, so it is enough to construct an $\Sp(H_{\Z})$-equivariant surjection 
$\rho\colon \fK(H) \rightarrow \Lambda_g$.

We define $\rho$ on generators, and then check that the corresponding map takes
relations to relations.  Let $V$ be a genus-$1$ symplectic summand of $H_{\Z}$ and let
$\kappa \in \wedge^2 V_{\Q}^{\perp}$.
We then have generators $\Pres{V,\kappa}$ and $\Pres{\kappa,V}$ for $\fK(H)$.  
Define
\[\rho(\Pres{V,\kappa}) = \LLBComm{V,\kappa} \quad \text{and} \quad \rho(\Pres{\kappa,V}) = -\LLAComm{\kappa,V}.\]
The minus sign is there to ensure:

\begin{unnumberedclaim}
The map $\rho$ takes relations to relations, and thus gives a well-defined map.
\end{unnumberedclaim}
\begin{proof}[Proof of claim]
It follows from Lemma~\ref{lemma:linearityrel2} that $\rho$ respects the linearity relations, so we must check the other two:
\begin{itemize}
\item Let $V$ and $W$ be genus-$1$ symplectic summand of $H_{\Z}$ with $W \subset V^{\perp}$.  
We can then apply Lemma~\ref{lemma:identifygenerators1} to see that
\end{itemize}
\[\rho(\Pres{V,\omega_W}) = \LLBComm{V,\omega_W} = -\LLAComm{\omega_V,W} = \rho(\Pres{\omega_V,W}).\]
\begin{itemize}
\item Let $V$ be a genus-$1$ symplectic summand of $H_{\Z}$.  We
can apply Lemma~\ref{lemma:identifygenerators2} to see that
\end{itemize}
\[\rho(\Pres{V,\omega_{V^{\perp}}}) = \LLBComm{V,\omega_{V^{\perp}}} = -\LLAComm{\omega_{V^{\perp}},V}  = \rho_2(\Pres{\omega_{V^{\perp}},V}).\qedhere\]
\end{proof}

The map $\rho$ is $\Sp(H_{\Z})$-equivariant by construction, and is surjective since its
image contains all the generators for $\Lambda_g$ identified by
Lemma~\ref{lemma:refinedgens}.  The theorem follows.
\end{proof}

\end{document}